\documentclass[11pt]{article}
\usepackage[margin=.8in]{geometry}
\usepackage[T1]{fontenc}
\usepackage{lmodern}
\usepackage{upquote}
\usepackage{mathtools,amsthm,amssymb,mathrsfs}
\usepackage[colorlinks, linkcolor=red, citecolor=blue]{hyperref}
\usepackage{xcolor,graphicx}
\usepackage{tikz}
\usepackage{tikz-cd}
\usepackage{booktabs,array,microtype}
\usepackage[shortlabels]{enumitem}
\setlist{itemsep=1pt, parsep=0pt, topsep=3pt, partopsep=0pt }

\theoremstyle{plain}
\newtheorem{theorem}{Theorem}[section]
\newtheorem{proposition}[theorem]{Proposition}
\newtheorem{corollary}[theorem]{Corollary}
\newtheorem{lemma}[theorem]{Lemma}
\newtheorem{conjecture}[theorem]{Conjecture}

\theoremstyle{definition}
\newtheorem{definition}[theorem]{Definition}
\newtheorem{problem}[theorem]{Problem}
\newtheorem{example}[theorem]{Example}
\newtheorem{remark}[theorem]{Remark}

\renewcommand{\mathcal}{\mathscr}
\newcommand{\ds}{\displaystyle}

\newcommand{\LL}{\mathcal{L}}    
\newcommand{\GD}{\mathsf{GD}}    
\newcommand{\BZ}{\mathcal{B}}    
\newcommand{\cA}{\mathcal{A}}    
\newcommand{\cF}{\mathcal{F}}    
\newcommand{\cC}{\mathcal{C}}    
\newcommand{\cD}{\mathcal{D}}
\newcommand{\cE}{\mathcal{E}}    
\newcommand{\cG}{\mathcal{G}}    
\newcommand{\cV}{\mathcal{V}}    
\newcommand{\fC}{\mathfrak C}    
\newcommand{\fT}{\mathfrak T}    
\newcommand{\ba}{\mathbf{a}}
\newcommand{\bb}{\mathbf{b}}
\newcommand{\bc}{\mathbf{c}}
\newcommand{\be}{\mathbf{e}}
\newcommand{\bq}{\mathbf{q}}
\newcommand{\bu}{\mathbf{u}}
\newcommand{\bv}{\mathbf{v}}
\newcommand{\bw}{\mathbf{w}}
\newcommand{\bx}{\mathbf{x}}
\newcommand{\by}{\mathbf{y}}
\newcommand{\bz}{\mathbf{z}}
\newcommand{\bOne}{\mathbf{1}}

\newcommand{\EA}{\mathsf{EA}} 
\newcommand{\XB}{\mathsf{XB}} 
\newcommand{\Hil}{\operatorname{Hilb}}
\newcommand{\ST}{\operatorname{ST}} 
\newcommand{\mult}{\operatorname{mult}}
\newcommand{\wt}{\mathsf{wt}}
\newcommand{\eps}{\varepsilon}
\newcommand{\kp}{\varkappa}
\newcommand{\ext}{\mathrm{ext}}
\renewcommand{\sp}{\mathrm{sp}}
\newcommand{\FF}{\mathbb F}
\newcommand{\ZZ}{\mathbb Z}
\newcommand{\QQ}{\mathbb Q}
\newcommand{\RR}{\mathbb R}

\numberwithin{equation}{section}
\allowdisplaybreaks[2]
\title{The Loopy Polynomial: from Tutte's Universal $V$-Function to Bizonotopal Geometry}

\author{
Anatol Kirillov\thanks{Yanqi Lake Beijing Institute of Mathematical Sciences
  and Applications, Huairou District, Beijing, China,
  \texttt{kirillov@bimsa.cn}}
\quad
Gleb Nenashev\thanks{Department of Mathematics and Computer Science,
  St.~Petersburg State University, St.~Petersburg, 199178, Russia,
  \texttt{glebnen@gmail.com}}
\quad
Boris Shapiro\thanks{Department of Mathematics, Stockholm University,
  SE-106 91 Stockholm, Sweden, and Department of Mathematics with Computer
  Science, Guangdong Technion--Israel Institute of Technology, Shantou,
  China, \texttt{shapiro@math.su.se}}
\quad
Arkady Vaintrob\thanks{Department of Mathematics, University of Oregon,
  Eugene, OR 97403, USA, \texttt{vaintrob@uoregon.edu}}
}

\date{}

\begin{document}

\maketitle
\begin{abstract}
We study the loopy polynomial $\LL_G$, a multivariate graph invariant
arising from bizonotopal graphical algebras and defined by deletion and
loopy contraction, in which the contracted edge is retained as a loop.
It determines the Tutte polynomial and admits a spanning-forest expansion
recording external activity componentwise. It also determines Stanley's
chromatic symmetric function for every graph; the degree sequence,
induced edge count profile, and independence polynomial for loopless
graphs; and the clique and matching polynomials for simple graphs.
Separating component size from external activity gives a refinement
that is equivalent to the extended $U$-polynomial and the extended
polychromate. The ordinary $U$-polynomial, Tutte's universal $V$-function,
and Stanley's Tutte symmetric function are specializations of this
refinement.

We conjecture that $\LL_G$ and $U_G$ have the same distinguishing power
on simple graphs, and verify this through eleven vertices. We exhibit
five-vertex loopless multigraphs with equal $U$-polynomials but distinct
loopy and extended $U$-polynomials, solving a problem of Merino and Noble.
We also disprove Loebl's conjecture for chordal graphs and its restriction
to split graphs, with counterexamples on eleven and fourteen vertices,
respectively.

For the score polytope $P_G$ of the external bizonotopal algebra, we give
a geometric realization of loopy deletion-contraction. The resulting
forest-indexed geometric parking complexes partition the lattice points
of $P_G$ and are piecewise integral-affinely parametrized by boxes whose
interval lengths are the component indices in the forest expansion.
\end{abstract}

\tableofcontents
\medskip
{\small

}

\section{Introduction}\label{sec:intro}
The loopy polynomial $\LL_G$ studied in this paper arose from the theory of
bizonotopal graphical algebras developed in~\cite{KNSV}.
In that work, we proved that the Hilbert series of the external and central bizonotopal algebras
satisfy a deletion-contraction recurrence in which the usual graph-theoretic contraction
is replaced by \emph{loopy contraction}:
given a graph $G$ (possibly with loops and multiple edges) and a non-loop edge $e$, identify its endpoints and retain $e$ as a loop at the merged vertex.
 We illustrate
deletion $G-e$, ordinary contraction $G_e$, and loopy contraction $G/e$ in Figure~\ref{fig:opers}.

\begin{figure}[ht]
\centering
\parbox[b]{0.8in}{   \centering
\begin{tikzpicture} [scale=.65]
\coordinate (A) at (1, 0);
\coordinate (B) at (2, 0);
\coordinate (E) at (0, 0.5);
\coordinate (F) at (0, 0);
\coordinate (G) at (0, -0.5);
\coordinate (H) at (3, 0.5);
\coordinate (I) at (3, -0.5);
\coordinate (M) at (1.5, 0);
\draw (A) -- (B);
\draw (A) -- (E);
\draw (A) -- (F);
\draw (A) -- (G);
\draw (H) -- (B);
\draw (I) -- (B);
\filldraw[black] (A) circle (2pt) node[below] {\small $u$};
\filldraw[black] (B) circle (2pt) node[below] {\small $v$};
\draw (M) circle (0pt) node[above] {\small $e$};
\end{tikzpicture}
  \\
  $G$         } \qquad
\parbox[b]{0.8in}{   \centering
\begin{tikzpicture} [scale=.65]
\coordinate (A) at (1, 0);
\coordinate (B) at (2, 0);
\coordinate (E) at (0, 0.5);
\coordinate (F) at (0, 0);
\coordinate (G) at (0, -0.5);
\coordinate (H) at (3, 0.5);
\coordinate (I) at (3, -0.5);
\coordinate (M) at (1.5, 0);
\draw (A) -- (E);
\draw (A) -- (F);
\draw (A) -- (G);
\draw (H) -- (B);
\draw (I) -- (B);
\filldraw[black] (A) circle (2pt) node[below] {\small $u$};
\filldraw[black] (B) circle (2pt) node[below] {\small $v$};
\end{tikzpicture}
\\
$G-e$  }   \qquad
\parbox[b]{0.8in}{   \centering
  \begin{tikzpicture}  [scale=.65]
    \coordinate (A) at (1, 0);
    \coordinate (E) at (0, 0.5);
    \coordinate (F) at (0, 0);
    \coordinate (G) at (0, -0.5);
    \coordinate (H) at (2, 0.5);
    \coordinate (I) at (2, -0.5);
    \coordinate (M) at (1, 0);
    \draw (A) -- (E);
    \draw (A) -- (F);
    \draw (A) -- (G);
    \draw (H) -- (A);
    \draw (I) -- (A);
    \filldraw[black] (A) circle (2pt) node[below] {\small $w$};
  \end{tikzpicture}
  \\
  $G_e$ }  \qquad
\parbox[b]{0.8in}{   \centering
  \begin{tikzpicture}[scale=.65]
    \coordinate (A) at (1, 0);
    \coordinate (E) at (0, 0.5);
    \coordinate (F) at (0, 0);
    \coordinate (G) at (0, -0.5);
    \coordinate (H) at (2, 0.5);
    \coordinate (I) at (2, -0.5);
    \coordinate (M) at (1, .7);
    \draw (A) -- (E);
    \draw (A) -- (F);
    \draw (A) -- (G);
    \draw (H) -- (A);
    \draw (I) -- (A);
    \draw (A) to[out=40,in=140,distance=1.6cm] (A);
    \filldraw[black] (A) circle (2pt) node[below] {\small $w$};
    \draw (M) circle (0pt) node[above] {\small $e$};
  \end{tikzpicture}
  \\      $G/e$    }
\caption{Deletion, ordinary contraction, and loopy contraction of a non-loop edge $e$.}
\label{fig:opers}
\end{figure}
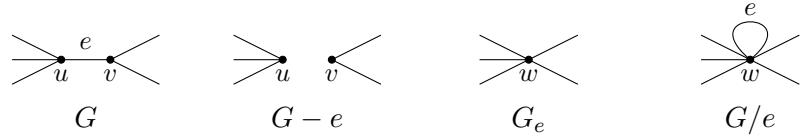

In this paper we study the invariant associated with this operation, place it among classical graph polynomials, and develop its connection with the score polymatroid of the external bizonotopal algebra.

\begin{definition}\label{def:main}
  The \emph{loopy polynomial} of a graph $G=(V,E)$ is the polynomial
  \[
    \LL_G(t,\bx)=\LL_G(t,x_0,x_1,\ldots) \in\ZZ[t,x_0,x_1,\ldots]
  \]
  determined by the following properties:
  \begin{enumerate}[(i)]
  \item for every non-loop edge $e\in E$,
    \begin{equation}\label{eq:loopy}
      \LL_G=\LL_{G/e}+t\,\LL_{G-e};
    \end{equation}
  \item if $G_1$ and $G_2$ are disjoint graphs, then
    \[
      \LL_{G_1\sqcup G_2}=\LL_{G_1}\LL_{G_2};
    \]
  \item if $L_m$ denotes the graph with one vertex and $m$ loops, then
    \begin{equation}\label{eq:init}
      \LL_{L_m}=x_m.
    \end{equation}
  \end{enumerate}
\end{definition}
In Section~\ref{sec:properties} we prove that this is well-defined and that $\LL_G$,
like the Tutte polynomial,  has a forest activity expansion.
Fix a linear order on $E(G)$. For a component $T$ of a spanning forest $F$,
let $\eps_F(T)$ be the number of externally active
edges assigned to $T$ and $\ds \eps(F):= \sum_T\eps_F(T)$.  Then
\begin{equation}\label{eq:forest-intro}
  \LL_G(t,\bx)  =  \sum_{F\in\cF(G)} t^{|E(G)|-|F|-\eps(F)}\prod_{T\in c(F)}
  x_{|E(T)|+\eps_F(T)}.
\end{equation}
The right-hand side is independent of the edge order. Unlike the usual
Tutte activity expansion, it runs over all spanning forests and assigns
to each component the index $|E(T)|+\eps_F(T)=|V(T)|-1+\eps_F(T)$.

In Section~\ref{sec:properties} we show that $\LL_G$ determines the Tutte polynomial $T_G$.
The converse fails: $\LL_G$ is not a matroid invariant.  For example, it distinguishes the path $P_4$
from the star $K_{1,3}$, although both have Tutte polynomial $X^3$.
We also show that $\LL_G$ is irreducible precisely when $G$ is connected and its irreducible
factorization recovers the loopy polynomials of the connected components.

For loopless graphs, differentiating with respect to $x_0$ gives a
vertex-deletion identity that recovers the degree sequence. Higher
derivatives recover the induced edge count profile
\[
  N_G(s,j)=\bigl|\{S\subseteq V(G):|S|=s,\ e_G(S)=j\}\bigr|,
\]
and hence the independence polynomial and, for simple graphs, the clique
polynomial. The forest expansion also recovers Stanley's chromatic
symmetric function~\cite{Sta} for every graph and the matching polynomial
for simple graphs.

The activity expansion~\eqref{eq:forest-intro} connects $\LL_G$ with other classical
invariants: Tutte's universal $V$-function $\cV_G$~\cite{Tu1},
Brylawski's polychromate $\eta_G$~\cite{Bry}, Stanley's Tutte symmetric function
$\XB_G$~\cite{Sta2}, and the $U$-polynomial of Noble and Welsh~\cite{NW}.
By work of Wang and Sachs~\cite{WS} and Bollob\'as, Pebody and Riordan~\cite{BPR},
Tutte's universal $V$-function also admits a spanning-forest expansion in which external
activity is distributed among components.
We prove in Section~\ref{subsec:loopy-V} that
\[
  \LL_G(1,\bx)    =    \left. \cV_{G^\circ}(\bv)  \right|_{v_r=x_{r-1}\ (r\ge 1)},
\]
where $G^\circ$ is the graph obtained from $G$ by adding a new loop at every vertex.
Thus $\LL_G(1,\bx)$ is a relabeled universal $V$-function of $G^\circ$.
The grading adds no information: weighted homogeneity recovers
$\LL_G(t,\bx)$ from $\LL_G(1,\bx)$.

A second comparison comes from separating the two component statistics in \eqref{eq:forest-intro}.
In Section~\ref{sec:refined} we introduce the \emph{refined loopy polynomial}
\[
  \widehat{\LL}_G(\bu)  =  \sum_{F\in\cF(G)} \prod_{T\in c(F)}u_{|V(T)|,\eps_F(T)}
\]
and prove that it is related by an invertible componentwise binomial transform to the extended
$U$-polynomial $U_G^{\ext}$.
By the theorem of Merino and Noble~\cite{MN}, the
latter is equivalent to the extended polychromate $\eta_G^\ext$.
The resulting dependencies are summarized by
\begin{equation}\label{eq:common-diagram-intro}
  \begin{array}{ccccc}
    && U_G^{\ext}\ \longleftrightarrow\ \widehat{\LL}_G  \ \longleftrightarrow\ \eta_G^{\ext} &&\\[2mm]
    & \swarrow & \downarrow & \searrow &\\[-1mm]
    U_G\ \longleftrightarrow\ \eta_G\ \longleftrightarrow\ \XB_G & & \cV_G & & \LL_G .
  \end{array}
\end{equation}

The ordinary $U$-polynomial retains component sizes and total activity;
the universal $V$-function retains component activities; and the loopy
polynomial retains the diagonal indices $s-1+r$. The precise
specializations are given in Section~\ref{sec:refined}.

These interdependencies lead to two closely related equivalence problems.
The first concerns two specializations of the common refinement.
We conjecture that on simple graphs the loopy polynomial and the
$U$-polynomial have the same distinguishing power, that is
\begin{equation}\label{eq:main-conjecture-intro}
  \LL_G=\LL_H  \quad\Longleftrightarrow \quad U_G=U_H.
\end{equation}

The second problem asks whether the ordinary and refined loopy polynomials have
the same distinguishing power on loopless multigraphs. The only nontrivial implication is
\[
  \LL_G=\LL_H \quad \Longrightarrow \quad  \widehat{\LL}_G=\widehat{\LL}_H.
\]
Our computations support both equivalence phenomena.  We verify the equivalence
\eqref{eq:main-conjecture-intro} for all simple graphs on at most eleven
vertices, and find no counterexample to the second problem among
$48\,229\,871$ connected loopless multigraphs.  The analogous equivalence
between the ordinary and extended $U$-polynomials does fail: we exhibit two
loopless multigraphs on five vertices with equal $U$-polynomials but distinct
extended $U$-polynomials, thereby solving a problem of Merino and Noble~\cite{MN}.
The same pair has distinct loopy polynomials.  Loops already
give a three-vertex counterexample to~\eqref{eq:main-conjecture-intro}, so the restriction
to simple graphs there is essential.

The collision census has a further consequence. Loebl~\cite{Loe} (see~\cite[Section~2.3]{LS})
conjectured that the $U$-polynomial distinguishes non-isomorphic chordal graphs.
Our census contains four pairs of non-isomorphic connected simple chordal
graphs on eleven vertices with equal loopy and $U$-polynomials, disproving this conjecture.
No such pair exists on at most ten vertices, so these counterexamples have minimum order.
One pair consists of $2$-trees. Thus the analogue of Stanley's tree
conjecture obtained by replacing trees and $X$ with $2$-trees and $U$
also fails.

The restriction to split graphs, attributed to Loebl and Sereni~\cite{LS}
by Aliste-Prieto, Crew, Spirkl, and Zamora~\cite[\S6.1]{ACSZ}, fails first
on fourteen vertices. Coning gives split collision pairs at every larger
order. The split graphs in which every vertex of the stable part has
degree two, which motivated the conjecture, remain distinguished by $U$
through twenty-one vertices in our computation.

Several structural results supplement these computations. On simple
graphs, equality of $\XB$ or $\widehat{\LL}$ is preserved by
complementation and joins. One
direction of~\eqref{eq:main-conjecture-intro} becomes a compatibility
problem for connected-block transforms. For simple unicyclic graphs,
comparison reduces to a rooted-tree injectivity problem, and we prove
the equivalence when every branch attached to the cycle has fewer
vertices than the cycle. Finally, normalized $\LL_G$ satisfies the
triangular modular four-term relation also satisfied by Stanley's Tutte
symmetric function~\cite{OS}.

Our geometric results return to the bizonotopal origin of $\LL_G$.
For $S\subseteq V(G)$, let
\[
  \kp_G(S)  = \left|\{e\in E(G):e\text{ has at least one endpoint in } S\}\right|.
\]
This function is submodular, and the \emph{score polytope}
\[
  P_G = \left\{ \ba\in\RR_{\ge 0}^{V(G)}: \ba(S)\le\kp_G(S) \text{ for all }S\subseteq V(G) \right\}
\]
is the independence polytope of the score polymatroid introduced in~\cite{KNSV}.
Its lattice points index the monomial basis of the external bizonotopal algebra.

We realize deletion by a translated copy of $P_{G-e}$ and loopy
contraction by a piecewise integral-affine section of the projection
$P_G\to P_{G/e}$. Their lattice points are disjoint and cover all of
$P_G\cap\ZZ^{V(G)}$.

Iterating this construction along a deletion-contraction tree gives, for every
spanning forest $F$, a geometric parking complex
$  \fC_F\subseteq P_G  $
piecewise integral-affinely parametrized by the box
\[
  \prod_{T\in c(F)} [0,|E(T)|+\eps_F(T)].
\]
The complexes $\fC_F$ are pairwise disjoint, and their lattice points partition
the lattice points of $P_G$.

The corresponding lattice-point generating function is
\[
  \sum_{\ba\in P_G\cap\ZZ^{V(G)}}q^{|\ba|}  = \LL_G(q,1,1+q,1+q+q^2,\ldots),
\]
which is precisely the Hilbert series of the external bizonotopal algebra
from~\cite{KNSV}.  Thus the loopy deletion-contraction recurrence, originally
discovered at the level of this Hilbert series, lifts to the geometry of its
underlying polymatroid.

The function $\kp_G$ also appears in the higher vertex-deletion
identity of Section~\ref{sec:vertex-deletion}.

Sections~\ref{sec:properties}--\ref{sec:vertex-deletion} develop the loopy polynomial, its forest
expansion and specializations, and vertex deletion. Sections~\ref{sec:relations}--\ref{sec:loops-weights}
compare it with classical invariants, introduce the refinement, and study the effect of adding loops.
Section~\ref{sec:tutte-symmetric} formulates the equivalence conjecture and presents the computational
evidence and counterexamples, and Section~\ref{sec:structural-evidence} gives structural evidence.
Section~\ref{sec:polyhedral} presents the geometric realization of loopy deletion-contraction,
Section~\ref{sec:outlook} collects open questions, and Appendix~\ref{sec:the-computation} describes
the computations.

\paragraph{Conventions and notation.}
Throughout, by ``graph'' we understand a finite undirected multigraph. Loops and multiple edges
are allowed unless explicitly stated otherwise, and multiple edges are regarded as distinct.
We write $V(G)$ and $E(G)$ for the vertex and edge sets, $n(G)$ for the number of vertices,
$m(G)$ for the number of edges (with loops and parallel copies counted with multiplicity),
$c(G)$ for the set of connected components, and $k(G)$ for the number of connected components.
The rank and cyclomatic number (the first Betti number) are
$ r(G)=n(G)-k(G)$ and $\nu(G)=m(G)-r(G)=m(G)-n(G)+k(G)$.
When the graph is fixed we usually abbreviate these to $n, m, k, r, \nu$.
The empty graph has $k=0$ and is not considered connected.
We call a graph \emph{unicyclic} if $\nu(G)=1$, without requiring
connectedness. Thus it has one connected component containing a unique
cycle, and all its other components are trees.

For a vertex $v\in V(G)$ we denote by $d_G(v)$, $\ell_G(v)$, and $\delta_G(v)$ its degree,
the number of loops, and the number of non-loop edges incident to it, respectively, so $d_G(v)=\delta_G(v)+2\ell_G(v)$.
Let $\ell(G)=\sum_{v\in V(G)}\ell_G(v)$ be the total number of loops of $G$.
For $S\subseteq V(G)$ let $e_G(S)$ be the number of edges with both
endpoints in $S$, and let $\kp_G(S)$ be the number of edges with at least one endpoint in $S$, with each edge counted once. In particular,
\[
  e_G(\{v\})=\ell_G(v) \quad \text{ and } \quad \kp_G(\{v\})=\delta_G(v) + \ell_G(v) = d_G(v)-\ell_G(v).
\]

For $e\in E(G)$, we write $e=uv$ to indicate that $e$ has endpoints $u,v\in V(G)$.
The notation $G-e$ means deletion of that particular edge. If $e$ is a non-loop edge,
$G/e$ denotes its loopy contraction, whereas $G_e$ denotes ordinary contraction.
For a vertex $v$, $G-v$ denotes \emph{vertex deletion}: the graph obtained by removing $v$
together with all edges incident to $v$.  We write $\overline G$ for the complement of a simple graph $G$.
If $e_1,\ldots,e_k$ are specified edges, the notation $ G+\{e_1,\ldots,e_k\} $
means the graph obtained from $G$ by adjoining these edges as new edges, parallel
to any existing ones.

A \emph{spanning forest} of $G$ means any spanning acyclic subgraph, not necessarily maximal.
We identify a spanning forest with its edge set and write $\cF(G)$ for the set of spanning
forests and $c(F)$ for the set of connected components of $(V(G),F)$.  Loops are never edges of a forest.

Whenever activity is used, the edge set is equipped with a linear order and we
use the least-edge convention.  For a spanning forest $F$ and an edge $e\notin
F$, we call $e$ \emph{externally active} relative to $F$ if either $e$ is a
loop, or the endpoints of $e$ lie in the same component of $F$ and $e$ is the
least edge of the unique cycle contained in $F\cup\{e\}$.  Thus a loop is a
one-edge cycle and is externally active relative to every spanning forest,
whereas an edge joining two distinct components of $F$ closes no cycle and is
never active.  An externally active edge is \emph{assigned} to the component
of $F$ containing its endpoints.  For $T\in c(F)$ we write $\eps_F(T)$ for the
number of externally active edges assigned to $T$.
We denote by $\EA(F)$ the set of all externally active edges of $F$. The total external activity of $F$ is
\[
  \eps_G(F)=\sum_{T\in c(F)}\eps_F(T).
\]
It will be denoted $\eps(F)$ when the ambient graph is clear.
The \emph{weight} $\wt_G(F)$ of a spanning forest is defined in~\eqref{eq:forest-weight}.
It is the contribution of $F$ to the forest expansion of $\LL_G$.

For graph invariants $P$ and $Q$ on a class $\cG$, we say that they have
the same distinguishing power, or induce the same fibers, if $P_G=P_H$
if and only if $Q_G=Q_H$ for all $G,H\in\cG$.  When an explicit invertible specialization
or coefficient transformation is available, we state this stronger relation separately.

We write $\RR^I$ for the real vector space with coordinates indexed by $I$, and $\ZZ^I \subseteq \RR^I$
for its integer lattice. Thus $\ZZ^{V}$ has coordinates labeled by the vertices of $G$,
and coordinates of $\RR^{c(F)}$ are  labeled by the components of a spanning forest $F$.

For a vector $\ba=(a_v)_{v\in V}$ and $S\subseteq V$, we write
$\ba(S):=\sum_{v\in S}a_v$ and $|\ba|:=\ba(V)=\sum_{v\in V}a_v$.
For a polynomial or formal series $P$ and a monomial $M$, the standard notation
$[M]P$ denotes the coefficient of $M$ in $P$.  We write $p_k$ for the $k$th power-sum symmetric function.

Below $T_G$ denotes the ordinary Tutte polynomial,
$\cV_G$ Tutte's 1947 universal $V$-function, $U_G$ the Noble--Welsh $U$-polynomial,
and $U_G^{\ext}$ its extended version.
We use $\eta_G$ and $\eta_G^{\ext}$ for the ordinary and extended polychromates,
$X_G$ for Stanley's chromatic symmetric function, and $\XB_G$ for Stanley's Tutte symmetric function.
The refined loopy polynomial introduced in Section~\ref{sec:refined} is denoted $\widehat{\LL}_G$.
Finally, $\cA_G^e,\cA_G^c$ denote the external and central graphical zonotopal algebras~\cite{PS,HR},
and $\BZ_G^e,\BZ_G^c$ the external and central bizonotopal algebras of~\cite{KNSV}.

Negative powers of $t$ are interpreted in the Laurent polynomial ring.

\section{The loopy polynomial: definition and basic properties}\label{sec:properties}

\subsection{Loopy deletion-contraction and well-definedness}\label{subsec:well-definedness}

We first verify that the recursion~\eqref{eq:loopy} is independent of
the order in which non-loop edges are processed.

\begin{proposition}\label{prop:well-def}
  The polynomial obtained from Definition~\ref{def:main} is independent of all
  choices made in the loopy deletion-contraction recursion.
\end{proposition}

\begin{proof}
We use induction on the number $N$ of non-loop edges.  There is nothing to
prove for $N=0$.  Suppose $N>0$, and let $e$ and $f$ be two non-loop edges
that could be processed first.

Assume first that $e$ and $f$ are not parallel.  After applying loopy
deletion-contraction first to $e$ and then to $f$, we obtain
\[
  \LL_{(G/e)/f}  +t\LL_{(G/e)-f}  + t\LL_{(G-e)/f}  + t^2\LL_{G-e-f}.
\]
Reversing the order gives the same four graphs, with the two mixed
terms interchanged: deletion and loopy contraction commute for distinct
nonparallel edges, and the two contractions identify the same vertices
and retain both edges as loops.
All terms occurring after these first two operations have fewer than $N$ non-loop edges,
so the induction hypothesis applies to their further evaluation.

Suppose now that $e$ and $f$ are parallel.  Loopy contraction of either one turns the other into a loop,
so it is no longer processed, and loops are never chosen by the recursion.  Expanding first in $e$ gives
\[
  \LL_{G/e}  + t\LL_{(G-e)/f}  + t^2\LL_{G-e-f},
\]
whereas expanding first in $f$ gives
\[
  \LL_{G/f}  + t\LL_{(G-f)/e}  + t^2\LL_{G-e-f}.
\]
Again the corresponding graphs are canonically isomorphic: $G/e\cong G/f$ and $(G-e)/f\cong(G-f)/e$.
The induction hypothesis completes the argument.
\end{proof}

\subsection{The forest expansion and Tutte specialization} \label{subsec:forest-expansion}

Fix a linear order on $E(G)$ and use the activity convention stated
above. Unlike the usual Tutte expansion over maximal spanning forests,
which records internal and external activity~\cite{Tu2,EM}, the expansion
below runs over all spanning forests and records external activity
componentwise.

For a spanning forest $F$ of a graph $G$ with a fixed ordering of $E(G)$ we write
\begin{equation}\label{eq:forest-weight}
  \wt_G(F)  :=  t^{|E(G)|-|F|-\eps(F)} \prod_{T\in c(F)} x_{|E(T)|+\eps_F(T)}
\end{equation}
for the \emph{weight} of $F$ in $G$. Here $|E(T)|=|V(T)|-1$, since the components of a forest are trees.

\begin{theorem}[The forest expansion]\label{thm:forests-sum}
For every graph $G$ and every ordering of $E(G)$,
\begin{equation}\label{eq:forests}
  \LL_G(t,\bx) = \sum_{F\in\cF(G)}  t^{|E(G)|-|F|-\eps(F)} \prod_{T\in c(F)}  x_{|E(T)|+\eps_F(T)}.
\end{equation}
That is, $\ds \LL_G=\sum_{F\in\cF(G)}\wt_G(F)$.
In particular, the right-hand side is independent of the edge order.
\end{theorem}

\begin{proof}
Let $\Phi_G$ denote the right-hand side of~\eqref{eq:forests},
and let $e$ be the largest non-loop edge of $G$.
We partition the spanning forests of $G$ according to whether they contain $e$.

Suppose first that $e\notin F$.  Then $F$ is a spanning forest of $G-e$.
Moreover, $e$ is not externally active relative to $F$. Indeed, edges with endpoints
in different components are never active, and if the endpoints of $e$ lie
in the same component of $F$, the cycle in $F\cup\{e\}$ contains another
edge, and $e$, being the largest edge, cannot be its least edge.
The external activities and their component assignments are therefore the same in $G$ and in $G-e$.
Since $|E(G)|=|E(G-e)|+1$, the forests omitting $e$ contribute  $ t\,\Phi_{G-e}$.

Now suppose that $e\in F$.  Contracting $e$ gives a bijection
\[
  F\longmapsto F/e:=F\setminus\{e\}
\]
from spanning forests of $G$ containing $e$ to spanning forests of $G/e$.
Recall that loopy contraction retains $e$ as a loop.  Thus $e$ is externally
active in $G/e$ relative to $F/e$.

All other external activities are unchanged. An edge joining distinct
components remains inactive. If $f\in E(G)\setminus F$ has both endpoints
in one component, its fundamental cycle is either unaffected by the
contraction or is obtained by deleting $e$ from the corresponding cycle in $G$.  Since $e$ is the largest edge, deleting it does not change whether
$f$ is the least edge of that cycle.  In particular, if $f$ is parallel to
$e$, then $f<e$ and $f$ is active both before contraction, as the least edge
of the two-edge cycle, and after contraction, when it becomes a loop.
Component assignments are also unchanged.

Let $T$ be the component of $F$ containing $e$.  In $F/e$ it is replaced by a component $T/e$ with
\[
  |V(T/e)|=|V(T)|-1,  \qquad  \eps_{F/e}(T/e)=\eps_F(T)+1.
\]
Hence
\[
  |V(T/e)|-1+\eps_{F/e}(T/e)  =  |E(T)|+\eps_F(T),
  \quad |F/e|=|F|-1, \quad \text{and}  \quad  \eps(F/e)=\eps(F)+1,
\]
while $|E(G/e)|=|E(G)|$.  Therefore the power of $t$ is unchanged:
\[
  |E(G/e)|-|F/e|-\eps(F/e)  =  |E(G)|-|F|-\eps(F).
\]
It follows that the contribution of the forests containing $e$ is $\Phi_{G/e}$, which gives
$ \Phi_G=\Phi_{G/e}+t\Phi_{G-e}. $

Finally, if $G$ has only loop edges, each of them is externally active
for the unique spanning forest, so
\[
  \Phi_G=\prod_{v\in V(G)}x_{\ell_G(v)},
\]
agreeing with the initial conditions and multiplicativity in Definition~\ref{def:main}.
By Proposition~\ref{prop:well-def}, $\Phi_G=\LL_G$.
\end{proof}

The following interval decomposition converts edge-subset sums into
forest sums.

\begin{lemma}\label{lem:activity-intervals}
For any edge order, the Boolean lattice $2^{E(G)}$ is the disjoint union of the intervals
\begin{equation}\label{eq:activity-intervals}
[F,F\cup\operatorname{EA}(F)]  :=  \{F\cup J:J\subseteq\operatorname{EA}(F)\},  \qquad F\in\mathcal F(G).
\end{equation}
\end{lemma}

\begin{proof}
Use induction on the number of non-loop edges.  A loop is externally active for
every forest and may be either present or absent independently,
so loops cause no difficulty.  Let $e$ be the largest non-loop edge.
Under our least-edge convention for external activity, $e$ is never externally active.
The subsets not containing $e$ are partitioned by the intervals for $G-e$.
The subsets containing $e$ correspond, after contraction of $e$, to the subsets of $G_e$,
and forests containing $e$ correspond to forests of $G_e$.
Because $e$ is largest, contraction preserves the external activity of every other edge.
The induction hypothesis therefore gives the required disjoint partition in both classes.
\end{proof}

The forest expansion~\eqref{eq:forests} also gives a direct specialization of $\LL_G$
to the Tutte polynomial $T_G$.
We use the rank-nullity form of the Tutte polynomial,  see~\cite{Tu2,EM}. Let
$r(A)=n(G)-k\bigl(V(G),A\bigr)$ be the rank of a spanning subgraph with edge set $A\subseteq E(G)$.
Then
\begin{equation}\label{eq:tutte-rank-nullity}
  T_G(X,Y)  =   \sum_{A\subseteq E(G)} (X-1)^{\,r(E)-r(A)}\,(Y-1)^{\,|A|-r(A)} .
\end{equation}
This is multiplicative over connected components, and a loop contributes a factor $Y$.

First we use Lemma~\ref{lem:activity-intervals} to write $T_G$ as a sum over \emph{all} spanning forests,
weighted by external activity.

\begin{corollary}[All-forest form of the Tutte polynomial]
\label{cor:tutte-all-forests}
For every graph $G$ and every ordering of $E(G)$,
\begin{equation}\label{eq:tutte-all-forests}
  (X-1)^{k}\,T_G(X,Y)  =  \sum_{F\in\cF(G)}(X-1)^{|c(F)|}\,Y^{\eps(F)} .
\end{equation}
\end{corollary}

\begin{proof}
Group the summands of~\eqref{eq:tutte-rank-nullity} by the intervals of Lemma~\ref{lem:activity-intervals}.
Let $A=F\cup S$ with $S\subseteq\EA(F)$.
An externally active edge has both endpoints in one component of $F$,
so adding the edges of $S$ merges no components and $k\bigl(V(G),A\bigr)=|c(F)|$.  Hence
\[
  r(A)=n-|c(F)|=|F|,  \qquad  |A|-r(A)=|S|, \qquad    r(E)-r(A)=|c(F)|-k .
\]
The interval corresponding to $F$ therefore contributes
\[
  (X-1)^{\,|c(F)|-k} \sum_{S\subseteq\EA(F)}(Y-1)^{|S|}  = (X-1)^{\,|c(F)|-k}\,Y^{\eps(F)},
\]
since $\eps(F)=|\EA(F)|$.  Summing over $F\in\cF(G)$ and
multiplying by~$(X-1)^{k}$  gives~\eqref{eq:tutte-all-forests}.
\end{proof}

\begin{corollary}[Tutte specialization]\label{cor:tutte-specialization}
Let $n=|V(G)|$, $m=|E(G)|$, and $k=k(G)$.  Then
\begin{equation}\label{eq:tutte-specialization}
  \LL_G\bigl(  t,(X-1)Y,(X-1)tY^2,(X-1)t^2Y^3,\ldots \bigr) = t^m(X-1)^kY^nT_G(X,Y).
\end{equation}
In particular,
\begin{equation}\label{eq:tutte-specialization-t1}
  T_G(X,Y) = (X-1)^{-k}Y^{-n} \LL_G\bigl(1,(X-1)Y,(X-1)Y^2,\ldots\bigr).
\end{equation}
\end{corollary}

\begin{proof}
Substitute $ x_i=(X-1)t^iY^{i+1}$ in~\eqref{eq:forests}.
Then the weight~\eqref{eq:forest-weight} of a spanning forest $F$ becomes
\[
  t^{m-|F|-\varepsilon(F)}  \prod_{T\in c(F)}  \left((X-1)t^{|E(T)|+\varepsilon_F(T)}
    Y^{|V(T)|+\varepsilon_F(T)}\right) =  t^m(X-1)^{|c(F)|}Y^{n+\varepsilon(F)}
\]
because $\ds  \sum_{T\in c(F)}|V(T)|=n $ and $\ds  \sum_{T\in c(F)}(|V(T)|-1)=|F| $.

Now summing over $F$  and applying  Corollary~\ref{cor:tutte-all-forests} we obtain~\eqref{eq:tutte-specialization}.
Setting $t=1$ gives~ \eqref{eq:tutte-specialization-t1}.
\end{proof}

\begin{corollary}\label{cor:forest-L}
If $G$ is a forest, then
\[
  \LL_G(t,\bx)  =  \sum_{A\subseteq E(G)} t^{|E(G)|-|A|}  \prod_{T\in c(A)}x_{|V(T)|-1}.
\]
\end{corollary}

\begin{proof}
If $G$ is a forest and $A\subseteq E(G)$, then every edge of   $E(G)\setminus A$
has its endpoints in different components of $A$.
Thus there are no externally active edges, and the result follows from Theorem~\ref{thm:forests-sum}.
\end{proof}

\subsection{Degree, extremal terms, and multiplicativity}\label{sec:corols}

We collect several immediate consequences of the forest expansion.
For a monomial in the variables $\bx$, its \emph{ordinary $\bx$-degree} means the
total degree obtained by assigning degree $1$ to every $x_i$.

\begin{corollary}\label{th:cor}
Let $n=|V(G)|$, $m=|E(G)|$, and
let $ \nu(G)=m-n+k(G)$ be the cyclomatic number of $G$.
\begin{enumerate}[(1)]
\item If $G$ is loopless, then the unique term of $\LL_G$ of $x_0$-degree $n$ is \   $ t^m x_0^n$.

\item If $G$ is simple, then for $n\ge 2$
  \[
    [x_0^{n-2}x_1]\,\LL_G(t,\bx)=m\,t^{m-1}.
  \]

\item With a grading
$ \deg t=1, \ \deg x_i=i,$
the polynomial $\LL_G$ is homogeneous of degree $m$.  In other words,
\begin{equation}\label{eq:scaling-L}
  \LL_G(t,\bx)  =   t^m\LL_G\bigl(1,x_0,t^{-1}x_1,t^{-2}x_2,\ldots\bigr).
\end{equation}

\item The number of edges $m(G)$ is recoverable from the ungraded polynomial
$\LL_G(1,\bx)$  as the maximum weighted $\bx$-degree with $\deg x_i=i$.
Thus the polynomials  $\LL_G(t,\bx)$ and $\LL_G(1,\bx)$ determine one another.

\item The maximum  $\bx$-degree of a monomial occurring in $\LL_G$ is $n$, and the
minimum is $k$. In particular, $\LL_G$ determines $n$ and $k$ and detects whether $G$
is connected.

\item The Hilbert series of $\cA_G^e$ is
\[
  \Hil(\cA_G^e;t)  = \LL_G(t,1,1,\ldots)  =  t^{\nu(G)}T_G(1+t,t^{-1}).
\]

\item If $G$ is connected and $\LL_G^{(1)}$ is the part of $\LL_G$ of ordinary
$\bx$-degree $1$,   then the Hilbert series of $\cA_G^c$ is
\[
  \Hil(\cA_G^c;t)  =  \left.\LL_G^{(1)}(t,\bx)\right|_{x_i=1}  =  t^{\nu(G)}T_G(1,t^{-1}).
\]

\item The Hilbert series of $\BZ_G^e$ is obtained by the substitution $x_i\longmapsto 1+t+\cdots+t^i$, that is \begin{equation}\label{eq:external-biz-specialization}
    \Hil(\BZ_G^e;t)  =  \LL_G\bigl(t,1,1+t,1+t+t^2,\ldots\bigr).
  \end{equation}

\item The Hilbert series of $\BZ_G^c$ is obtained by \
$  x_0\longmapsto0, \quad  x_i\longmapsto1+t+\cdots+t^{i-1} \quad(i\ge 1)$,
that is:
\begin{equation}\label{eq:central-biz-specialization}
  \Hil(\BZ_G^c;t)
  =
  \LL_G\bigl(t,0,1,1+t,1+t+t^2,\ldots\bigr).
\end{equation}
\end{enumerate}
\end{corollary}

\begin{proof}
For~(1), the empty forest has $n$ singleton components and, since $G$ is loopless, no external activity.
Thus it contributes $t^m x_0^n$ to~\eqref{eq:forests}.  Every nonempty
forest has a nonsingleton component and therefore has smaller $x_0$-degree.

For~(2), a term $x_0^{n-2}x_1$ can only come from a spanning forest consisting of one edge and $n-2$
isolated vertices.  Since $G$ is simple, such a forest has no external activity.
There are $m$ choices of the edge, each contributing $t^{m-1}x_0^{n-2}x_1$.

For~(3), every term indexed by a spanning forest $F$ has total weighted degree
\[
  |E(G)|-|F|-\eps(F)  +\sum_{T\in c(F)}  \bigl(|E(T)|+\eps_F(T)\bigr)
  = m-|F|-\eps(F)+|F|+\eps(F) = m.
\]
Equivalently, the weighted $\bx$-degree of the forest monomial is
$|F|+\eps(F)$, which gives~\eqref{eq:scaling-L} term by term.

For~(4), it remains to see that $m$ itself is recoverable from
$\LL_G(1,\bx)$.
Construct a maximal spanning forest greedily by processing the non-loop edges
from largest to smallest and adding an edge whenever it does not create a cycle.
Any omitted non-loop edge is then the least edge of its fundamental cycle, and any loop is externally active.
Every edge outside this forest is externally active, so $|F|+\eps(F)=m$.
A monomial of weighted $\bx$-degree $m$ therefore occurs in $\LL_G(1,\bx)$.
All forest monomials have weighted degree at most $m$, so $m$ is equal to the maximum such degree.
The scaling identity therefore reconstructs the full polynomial from its specialization at $t=1$.

For~(5), the ordinary $\bx$-degree of the term associated with $F$ is
$|c(F)|$.  Its largest possible value is $n$, attained only by the empty
forest, while its smallest possible value is $k(G)$, attained by every maximal spanning forest.

For~(6), setting all $x_i=1$ in~\eqref{eq:forests} gives
\[
  \LL_G(t,1,1,\ldots)= \sum_{F\in\cF(G)}  t^{m-|F|-\eps(F)} =    t^{m-r(G)}T_G(1+t,t^{-1}),
\]
which is the standard activity expression for the external zonotopal Hilbert series (see~\cite{PS,HR}).

For~(7), when $G$ is connected, the terms in~\eqref{eq:forests} of ordinary $\bx$-degree $1$
are precisely those indexed by spanning trees.  Selecting these terms and setting $x_i$ equal
to $1$ gives
\[
  \LL_G^{(1)}(t,1,1,\ldots) = t^{m-n+1}T_G(1,t^{-1}),
\]
which by~\cite{PS,HR} is equal to the Hilbert series of the central zonotopal algebra.

Finally,~(8) and~(9) follow from~\cite[Theorem~3.6]{KNSV}.  The external and  central bizonotopal
Hilbert series satisfy the same loopy deletion-contraction recurrence and multiplicativity as $\LL_G$,
while on the one-vertex graph $L_i$ with $i$ loops their initial values are, respectively,
\[
  1+t+\cdots+t^i   \qquad\text{and}\qquad   1+t+\cdots+t^{i-1}.
\]
The substitutions in~\eqref{eq:central-biz-specialization} and
\eqref{eq:external-biz-specialization} therefore give these Hilbert series.
\end{proof}

\begin{proposition}[Extremal terms]
  \label{prop:extremal-terms}
Let $m_i$ be the number of edges of the $i$th connected component of $G$. Then
\begin{equation}\label{eq:L-at-zero-components}
  \LL_G(0,\bx)=\prod_{i=1}^{k}x_{m_i}.
\end{equation}
Moreover, the unique term of the maximal $\bx$-degree $n(G)$ is
\begin{equation}\label{eq:max-x-degree-term}
  t^{m(G)-\ell(G)}\prod_{v\in V(G)}x_{\ell_G(v)}.
\end{equation}
For loopless $G$ this is the term $t^{m}x_0^{n}$ of Corollary~\ref{th:cor}(1).
Thus, besides the order and number of edges, $\LL_G$ determines the number of components $k$,
the rank, the multiset of edge counts $\{m(G_i)\}_{i=1}^k$, and the multiset
$\{\ell_G(v):v\in V(G)\}$ of loop multiplicities at the vertices.
\end{proposition}

\begin{proof}
Suppose first that $G$ is connected and has $m$ edges.  At $t=0$ the deletion
term in loopy deletion-contraction disappears.  Contracting the non-loop
edges successively therefore reduces $G$ to the one-vertex graph $L_m$ and
gives $\LL_G(0,\bx)=x_m$.  Applying multiplicativity to the connected
components gives~\eqref{eq:L-at-zero-components}.

For the second assertion, a forest term has ordinary $\bx$-degree equal to
its number of components.  Degree $n(G)$ can therefore occur only for the empty forest.
Relative to the empty forest every loop is externally active and is assigned to its vertex,
while no non-loop edge is externally active. Therefore the empty forest contributes
\[
  t^{m(G)-\ell(G)}\prod_{v\in V(G)}x_{\ell_G(v)},
\]
and no other forest contributes in that degree.  The stated elementary data
are now read from~\eqref{eq:L-at-zero-components},
\eqref{eq:max-x-degree-term}, and Corollary~\ref{th:cor}(3)--(4).
\end{proof}

Component factorization will allow invariant comparisons to be reduced
to connected graphs in Section~\ref{subsec:connected-reduction}.

\begin{proposition}[Connectedness and irreducibility]
\label{prop:L-irreducible}
A graph $G$ is connected if and only if $\LL_G(t,\bx)$ is irreducible in $\ZZ[t,x_0,x_1,\ldots]$.
\end{proposition}

\begin{proof}
For the empty graph, $\LL_G=1$ is not irreducible. If $G$ has at least
two connected components, multiplicativity gives a nontrivial
factorization into their loopy polynomials.

Conversely, suppose that $G$ is connected and let $m=m(G)$.  If $m=0$, then
$G$ is a single isolated vertex and $\LL_G=x_0$.  Assume $m>0$.  By
Proposition~\ref{prop:extremal-terms},
\[
  \LL_G(0,\bx)=x_m.
\]
By Corollary~\ref{th:cor}(3), $\LL_G$ is homogeneous of weighted degree $m$
for $\deg t=1$ and $\deg x_i=i$. So no variable $x_i$ with $i>m$ occurs,
and any monomial containing $x_m$ can contain no positive power of $t$ and no
$x_i$ with $i>0$ other than $x_m$.  A factor $x_0^j$ is not excluded by
homogeneity, since $\deg x_0=0$, but the specialization at $t=0$ rules it out:
it shows that the only such monomial is $x_m$ itself, with coefficient $1$.  Thus
\[
  \LL_G=x_m+Q(t,x_0,\ldots,x_{m-1}).
\]
Regarded as a polynomial in $x_m$ over $R=\ZZ[t,x_0,\ldots,x_{m-1}]$, it is therefore monic of degree one.
If $\LL_G=AB$, one factor, say $A$, is independent of $x_m$.  Comparing the
coefficients of $x_m$ gives $A\,[x_m]B=1$ in $R$.  Hence $A$ is a unit, so $\LL_G$ is irreducible.
\end{proof}

\begin{corollary}[Component factorization]\label{cor:L-component-factorization}
If  $G_1,\ldots,G_k$ are the connected components of $G$, then
\[
  \LL_G=\prod_{i=1}^k\LL_{G_i}
\]
is the irreducible factorization of $\LL_G$, up to the order of the factors.
In particular, $\LL_G$ determines the multiset $\{\LL_{G_1},\ldots,\LL_{G_k}\}$
and hence the multisets of the orders and sizes of the connected components and the multiset $\{\ell_G(v)\}$ of loop data via~\eqref{eq:max-x-degree-term}.
\end{corollary}

\begin{proof}
Each factor is irreducible by Proposition~\ref{prop:L-irreducible}.
Only finitely many variables occur, so the identity takes place in a finitely generated ring
$\ZZ[t,x_0,\ldots,x_N]$ which is a unique factorization domain.
The factorization is therefore unique up to units and the order of the factors.
The units of $\ZZ[t,x_0,\ldots,x_N]$ are $\pm 1$, but by the proof of
Proposition~\ref{prop:L-irreducible} a connected graph $D$ with $m(D)$ edges has
\[
  \LL_{D}=x_{m(D)}+Q\bigl(t,x_0,\ldots,x_{m(D)-1}\bigr),
\]
so the coefficient of $x_{m(D)}$ in $\LL_{D}$ is $1$ and no sign ambiguity can occur.
Hence associated irreducible factors are equal,
and the multiset $\{\LL_{G_1},\ldots,\LL_{G_k}\}$ is determined by $\LL_G$.
\end{proof}

\subsection{The \texorpdfstring{$4$}{4}-term relation and  multistars}
\label{sec:4-terms}

\subsubsection{The \texorpdfstring{$4$}{4}-term relation}
\label{subsubsec:four-term}
Let $u,v,w$ be three distinct vertices and write
\[
  e_1=uv,    \qquad    e_2=vw,    \qquad    e_3=uw.
\]
Recall the convention that $G+\{e_1,\ldots,e_s\}$ means that new copies of the indicated edges
are adjoined, even if edges with the same endpoints are already present in $G$.

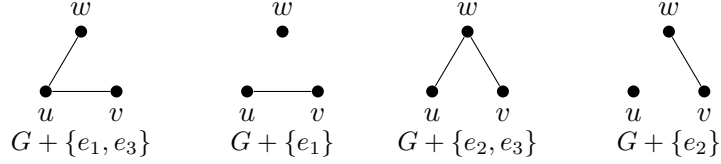
\begin{figure}[ht]
\centering
\begin{tikzpicture}[scale=.72,
  vertex/.style={circle,fill=black,inner sep=1.6pt}]

  \begin{scope}[xshift=0cm]
    \node[vertex,label=below:$u$] (u) at (0,0) {};
    \node[vertex,label=below:$v$] (v) at (1.3,0) {};
    \node[vertex,label=above:$w$] (w) at (.65,1.1) {};
    \draw (u)--(v);
    \draw (u)--(w);
    \node at (.65,-.95) {\small $G+\{e_1,e_3\}$};
  \end{scope}

  \begin{scope}[xshift=3.7cm]
    \node[vertex,label=below:$u$] (u) at (0,0) {};
    \node[vertex,label=below:$v$] (v) at (1.3,0) {};
    \node[vertex,label=above:$w$] (w) at (.65,1.1) {};
    \draw (u)--(v);
    \node at (.65,-.95) {\small $G+\{e_1\}$};
  \end{scope}

  \begin{scope}[xshift=7.1cm]
    \node[vertex,label=below:$u$] (u) at (0,0) {};
    \node[vertex,label=below:$v$] (v) at (1.3,0) {};
    \node[vertex,label=above:$w$] (w) at (.65,1.1) {};
    \draw (v)--(w);
    \draw (u)--(w);
    \node at (.65,-.95) {\small $G+\{e_2,e_3\}$};
  \end{scope}

  \begin{scope}[xshift=10.8cm]
    \node[vertex,label=below:$u$] (u) at (0,0) {};
    \node[vertex,label=below:$v$] (v) at (1.3,0) {};
    \node[vertex,label=above:$w$] (w) at (.65,1.1) {};
    \draw (v)--(w);
    \node at (.65,-.95) {\small $G+\{e_2\}$};
  \end{scope}

\end{tikzpicture}
\caption{The four local edge configurations in
  Theorem~\ref{thm:4term}.  The common graph $G$ is suppressed.}
\label{fig:4term}
\end{figure}

\begin{theorem}\label{thm:4term}
For the four graphs shown in Figure~\ref{fig:4term} we have
 \begin{equation}\label{eq:4term}
    \LL_{G+\{e_1,e_3\}}-t\,\LL_{G+\{e_1\}} = \LL_{G+\{e_2,e_3\}}-t\,\LL_{G+\{e_2\}}.
 \end{equation}
\end{theorem}

\begin{proof}
Apply loopy deletion-contraction to the newly added copy of $e_3$ in
$G+\{e_1,e_3\}$.  This gives
\[
  \LL_{G+\{e_1,e_3\}}-t\,\LL_{G+\{e_1\}} = \LL_{(G+\{e_1,e_3\})/e_3}.
\]
Similarly,
\[
  \LL_{G+\{e_2,e_3\}}-t\,\LL_{G+\{e_2\}} = \LL_{(G+\{e_2,e_3\})/e_3}.
\]
After contracting $e_3=uw$, the vertices $u$ and $w$ are identified.
Thus the new edge $e_1=uv$ in the first graph and the new edge
$e_2=vw$ in the second become corresponding edges between the
merged vertex and $v$, while $e_3$ becomes a loop at the merged vertex.
Edges already in $G$ are transformed identically.  So
\[
  (G+\{e_1,e_3\})/e_3  \cong (G+\{e_2,e_3\})/e_3,
\]
which proves~\eqref{eq:4term}.
\end{proof}

It is useful to record the normalized form of the relation~\eqref{eq:4term}.  Define
\[
  \overline{\LL}_G:=t^{-|E(G)|}\LL_G.
\]
Since the four graphs in~\eqref{eq:4term} differ from $G$ by one or two
new edges, Theorem~\ref{thm:4term} is equivalent to
\begin{equation}\label{eq:normalized-4term}
  \overline{\LL}_{G+\{e_1,e_3\}} - \overline{\LL}_{G+\{e_1\}} =
  \overline{\LL}_{G+\{e_2,e_3\}} - \overline{\LL}_{G+\{e_2\}}.
\end{equation}
This is the triangular modular relation of Orellana and Scott~\cite{OS}, the relation satisfied also
by the chromatic and Tutte symmetric functions.
We return to this comparison in Section~\ref{subsec:further-common-structure}.

\begin{remark}\label{rem:vassiliev-4term}
The $4$-term relation of Theorem~\ref{thm:4term} should not be confused with a different $4$-term
relation on graphs arising from the theory of Vassiliev knot invariants, see~\cite{Li,La}.
The latter arises from the $4$-term relation for chord diagrams by passing to their intersection graphs.
In graph-theoretic form it involves two Vassiliev moves on a pair of vertices:
toggling their adjacency and, in the second move, swapping the adjacency of
one of them to the other neighbors of the second.
By contrast, Theorem~\ref{thm:4term} compares four graphs obtained by adjoining edges
among three fixed vertices.
It is easy to see that the normalized loopy polynomial fails to satisfy the Vassiliev
relation already for three-vertex graphs.
\end{remark}

\subsubsection{Reduction to multistars}
\label{subsubsec:multistars}

The $4$-term relation of Theorem~\ref{thm:4term} gives
a useful tool for computing $\LL_G$ recursively, different from the deletion-contraction formula.
It allows the computation to be reduced to multistars.

\begin{definition}
\label{def:multistar}
A graph $G$ is called a
\emph{multistar centered at $v\in V(G)$} if every non-loop edge is incident to $v$ and
$G$ has at most one pair of distinct vertices joined by parallel edges.
Loops may occur at arbitrary vertices, see Figure~\ref{fig:multistar}.
\end{definition}

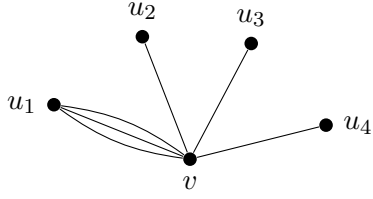
\begin{figure}[ht]
  \centering
\begin{tikzpicture}[scale=.9,
    vertex/.style={circle,fill=black,inner sep=1.8pt}]
    \node[vertex,label=below:$v$] (v) at (0,0) {};
    \node[vertex,label=left:$u_1$] (u1) at (-2,0.8) {};
    \node[vertex,label=above:$u_2$] (u2) at (-.7,1.8) {};
    \node[vertex,label=above:$u_3$] (u3) at (.9,1.7) {};
    \node[vertex,label=right:$u_4$] (u4) at (2,0.5) {};

    \draw (v) to[bend left=14] (u1);
    \draw (v) -- (u1);
    \draw (v) to[bend right=14] (u1);
    \draw (v) -- (u2);
    \draw (v) -- (u3);
    \draw (v) -- (u4);
\end{tikzpicture}
  \caption{A multistar centered at $v$.  Loops, which may occur at any vertex, are not shown.}
  \label{fig:multistar}
\end{figure}

\begin{proposition}[Reduction to multistars]\label{prop:4term-reduction}
Let $H$ be a connected graph with $m\ge 1$ non-loop edges. Then there exist a multistar $H^\ast$, graphs
$K_1,\ldots,K_N$ with $m-1$ non-loop edges, and signs $\sigma_i\in\{1,-1\}$ such that
\begin{equation}\label{eq:multistar-reduction}
  \LL_H  = \LL_{H^\ast} + t\sum_{i=1}^N\sigma_i\LL_{K_i}.
\end{equation}
Thus the $4$-term relation, multiplicativity, and the values on multistars determine $\LL_G$
for every graph $G$.
\end{proposition}

\begin{proof}
The basic move is the following consequence of Theorem~\ref{thm:4term}. Suppose in a graph $H$ a vertex $w$
is connected to two other vertices $v$ and $u$ by edges $b=wu$ and $c=wv$. Let $a=vu$ be a new edge.
Consider the graphs $ H_0=H-\{b,c\}$ and $H'=H_0+\{a,c\}$. Applying~\eqref{eq:4term} to the triple $u,v,w$ gives
\begin{equation}\label{eq:edge-slide}
   \LL_H  =  \LL_{H'}  +  t\bigl(\LL_{H_0+b}-\LL_{H_0+a}\bigr).
\end{equation}
We refer to the passage from $H$ to $H'$ as sliding the edge $b=wu$ across the edge $c=vw$
to the edge $a=vu$. The graphs $H$ and $H'$ have the same number of non-loop edges,
whereas graphs $H_0+a$ and $H_0+b$ in~\eqref{eq:edge-slide} have one fewer non-loop edge.
  \medskip

Fix a vertex $v\in V(H)$. We will use edge slides to make every non-loop edge incident with $v$. If some
non-loop edge is not incident to $v$, connectivity of $H$ gives a path $v-w-u$ whose second edge is not
incident with $v$. Indeed, if every other vertex is adjacent to $v$, take any edge $wu$ not incident
with $v$. Otherwise take the first two edges of a shortest path from $v$
to a nonneighbor. The path need not be induced, so the new edge $vu$ may
be parallel to an existing edge.

Let $c=vw$.  Applying the edge slide~\eqref{eq:edge-slide} replaces $b=wu$ in $H$ by the new edge $a=vu$,
which decreases the number of non-loop edges not incident to $v$ by one and produces
two correction terms with $m-1$ non-loop edges. Repeating this operation gives
  \[
    \LL_H  =  \LL_{H_1} + t\sum_i\sigma_i\LL_{K_i},
  \]
where every non-loop edge of $H_1$ is incident to $v$ and every $K_i$ has $m-1$ non-loop edges.

It remains to concentrate all excess multiplicity at one neighbor.  Let $u$ be a vertex adjacent to $v$.
Suppose another vertex $u'\ne u$ is connected to $v$ by at least two parallel edges.
We transfer one copy of $vu'$ to $vu$ by two edge slides.

First slide one copy of $vu'$ across $uv$ replacing that copy of $vu'$ by a new edge $uu'$.
Since at least one copy of $vu'$ remains, we can now
slide the new edge $uu'$ across a remaining copy of $vu'$ to a new copy of $vu$.
This decreases the multiplicity of the edge $vu'$ by one, while every correction term produced
by either slide has $m-1$ non-loop edges.
Repeating this two-slide operation for every $u'\ne u$ with
$\mult(vu')>1$ produces a multistar $H^\ast$ centered at $v$.
Collecting all correction terms gives~\eqref{eq:multistar-reduction}.

For a disconnected graph we apply the argument to its components.
Since every correction graph has fewer non-loop edges, the final assertion follows
by induction on their number using multiplicativity.
\end{proof}

The loopy polynomial of multistars can be evaluated explicitly using the forest expansion.

\begin{proposition}[Loopy polynomial of a multistar]
\label{prop:multistar-formula}
Let $H$ be a multistar centered at $v$ with other vertices $u_1,\ldots,u_d$ such that
$v$ is connected to $u_1$ by $r\ge 1$ edges and to $u_i$, for $2\le i\le d$, by a single edge $vu_i$.
Let $\lambda_0$ be the number of loops at $v$ and $\lambda_i$ the number of loops at $u_i$.
For $S\subseteq I:=\{2,\ldots,d\}$ define
\[
  a(S):=\lambda_0+|S|+\sum_{i\in S}\lambda_i.
\]
Then
\begin{equation}\label{eq:multistar-formula}
  \LL_H(t,\bx)=t^{r+d-1}\sum_{S\subseteq I}t^{-|S|} \prod_{i\in I\setminus S}x_{\lambda_i}
  \left(x_{a(S)}x_{\lambda_1}+\sum_{j=1}^r t^{-j}x_{a(S)+\lambda_1+j}\right).
\end{equation}
For a one-vertex multistar with $\lambda_0$ loops we have $\LL_H=x_{\lambda_0}$
by~\eqref{eq:init}.
\end{proposition}

\begin{proof}
Apply Theorem~\ref{thm:forests-sum}.
Let $b_1<\ldots<b_r$ be the ordering of parallel edges.
A spanning forest of $H$ contains an arbitrary subset
\[
  F_S=\{vu_i : i \in S\}, \qquad   S \subseteq I,
\]
of the simple spokes, and contains at most one of the parallel edges $b_1,\ldots,b_r$.

If no $b_j$ is chosen, then $u_1$ is a singleton component.
No edge in the parallel collection is externally active, while all loops are externally active.
The resulting term is
\[
  t^{\,r+d-1-|S|}  x_{a(S)}x_{\lambda_1}  \prod_{i\in I\setminus S}x_{\lambda_i}.
\]

If $b_j$ is chosen, then precisely
$  b_1,\ldots,b_{j-1}$
of the parallel edges are externally active: each forms a two-edge
cycle with $b_j$ and is the least edge of that cycle.
Thus the central component contributes
$ x_{a(S)+\lambda_1+j},$
and the power of $t$ is  $r+d-1-|S|-j.$
Summing these terms over $S$ and $j$ gives
\eqref{eq:multistar-formula}.
\end{proof}

\section{Examples and explicit formulas}\label{sec:examples}

We collect here the values of $\LL_G$ for a few families for which the forest expansion can
be evaluated in closed form.  We end with the smallest example showing that
$\LL_G$ is not a matroid invariant.

\begin{proposition}[Two-vertex graphs]\label{prop:multiedge}
Let $G_m$ have two vertices joined by $m\ge 1$ multiple edges.  Then
\begin{equation}\label{eq:multiedge}
  \LL_{G_m}=x_m+tx_{m-1}+\cdots+t^{m-1}x_1+t^mx_0^2 .
\end{equation}
In particular $\LL_{P_2}=x_1+tx_0^2$.  Compare with
Proposition~\ref{prop:multistar-formula}, of which this is the case of a single
leaf.
\end{proposition}

\begin{proof}
A spanning forest is either empty or a single edge.  For $F=\varnothing$ the
two vertices are separate components, so no edge has both endpoints in one
component and $\eps(F)=0$.  The weight is $t^{m}x_0^2$.  If
$F=\{e_j\}$, where $e_j$ is the $j$-th edge in the chosen order, then for
$e\ne e_j$ the unique cycle of $F\cup\{e\}$ is $\{e_j,e\}$, so $e$ is
externally active exactly when $e<e_j$.  Hence $\eps(F)=j-1$ and the weight is $t^{m-j}x_{j}$.
Summing over $j=1,\dots,m$ gives~\eqref{eq:multiedge}.
\end{proof}

\begin{proposition}[Paths]\label{prop:path}
For the path graph $P_m$ on $m\ge 1$ vertices, put
\[
  \Lambda_m=\left\{(a_0,\ldots,a_{m-1})\in\ZZ_{\ge 0}^{m}:
    \sum_{i=0}^{m-1}(i+1)a_i=m\right\},\qquad |\ba|=\sum_i a_i.
\]
Then
\[
  \LL_{P_m}(t,\bx) = \sum_{\ba\in\Lambda_m} \binom{|\ba|}{a_0,\ldots,a_{m-1}} t^{|\ba|-1}\prod_{i=0}^{m-1}x_i^{a_i}.
\]
\end{proposition}

\begin{proof}
Every subset of $E(P_m)$ is a spanning forest, and no nonforest edge lies on
a cycle, so $\eps(F)=0$ for all $F$.  Deleting $m-1-|F|$ of the $m-1$
edges leaves $m-|F|$ components whose orders form a composition of $m$ into
$m-|F|$ parts. Writing $a_i$ for the number of components with $i+1$ vertices
gives $|\ba|=m-|F|$ and $\sum_i(i+1)a_i=m$.
The number of compositions with these part multiplicities is the multinomial coefficient
$\ds \binom{|\ba|}{a_0,\ldots,a_{m-1}}$, and the weight of each such forest is
$t^{(m-1)-|F|}\prod_i x_i^{a_i}=t^{|\ba|-1}\prod_i x_i^{a_i}$.
\end{proof}

For example,
\begin{equation}
  \label{eq:path4}
  \LL_{P_3}=x_2+2tx_0x_1+t^2x_0^3, \qquad \LL_{P_4}=x_3+2t x_0x_2+t x_1^2+3t^2x_0^2x_1+t^3x_0^4.
\end{equation}

The next family contains the first simple graph in which external activity contributes
(it is also the case $\eps(F)\le1$ of the unicyclic analysis of
Section~\ref{subsec:unicyclic}).

\begin{proposition}[Cycles]\label{prop:cycle}
For the simple cycle $C_m$, $m\ge 3$, let
\[
  \Lambda'_m = \left\{(a_0,\ldots,a_{m-2})\in\ZZ_{\ge 0}^{m-1}:
   \sum_{i=0}^{m-2}(i+1)a_i=m,\quad |\ba|\ge 2\right\}  \qquad \text{and} \qquad
  \gamma_{\ba}=\frac{m}{|\ba|}  \binom{|\ba|}{a_0,\ldots,a_{m-2}}.
\]
Then
\[
  \LL_{C_m}=x_m+(m-1)tx_{m-1}  + \sum_{\ba\in\Lambda'_m}\gamma_{\ba}t^{|\ba|} \prod_{i=0}^{m-2}x_i^{a_i}. \]
\end{proposition}

\begin{proof}
The spanning forests of $C_m$ are the proper subsets of $E(C_m)$.  If
$|F|=m-1$, the single nonforest edge $e$ closes the whole cycle, so it is
externally active precisely when it is the least edge.  This gives $x_m$ for
that one forest and $tx_{m-1}$ for each of the remaining $m-1$.  If
$|F|\le m-2$, then at least two edges are missing, every nonforest edge joins
two distinct components, and $\eps(F)=0$.  The components are the arcs
cut out by the $p:=m-|F|$ missing edges. Fix an orientation of the
cycle, choose a distinguished missing edge in $m$ ways, and list the
component orders from that edge. For prescribed multiplicities there
are $\binom{p}{a_0,\ldots,a_{m-2}}$ such lists. Each forest has $p$
choices of distinguished missing edge, giving
$\frac mp\binom{p}{a_0,\ldots,a_{m-2}}=\gamma_{\ba}$ forests.
Each contributes
$t^{m-|F|}\prod_i x_i^{a_i}=t^{|\ba|}\prod_i x_i^{a_i}$.
\end{proof}

In particular,
\[ \LL_{C_3}=x_3+2tx_2+3t^2x_0x_1+t^3x_0^3,
  \qquad  \LL_{C_4}=x_4+3tx_3+2t^2x_1^2+4t^2x_0x_2 + 4t^3x_0^2x_1+t^4x_0^4.
\]

\begin{proposition}[Stars with loops]\label{prop:star}
Let $S_{m,\ell}$ be a star with $m$ non-loop edges and $\ell$ loops at its center.
Then
\begin{equation}\label{eq:star}
  \ds \LL_{S_{m,\ell}}=
  \sum_{i=0}^{m}\binom mi(tx_0)^i x_{\ell+m-i},
\end{equation}
and in particular $\ds \LL_{S_m}=\sum_{i=0}^{m}\binom mi(tx_0)^i x_{m-i}$.
The case $\ell=0$, $m=3$ is used in Remark~\ref{rem:nonmatroidal-example}.
\end{proposition}

\begin{proof}
For $m\ge1$, specialize Proposition~\ref{prop:multistar-formula} to
$r=1$, $d=m$, $\lambda_0=\ell$, and $\lambda_i=0$ at every leaf,
and group by the number of omitted spokes. For $m=0$, the formula is
$\LL_{L_\ell}=x_\ell$.
\end{proof}

\begin{remark}[A nonmatroidal comparison]\label{rem:nonmatroidal-example}
The path and star formulas provide a simple illustration of the
nonmatroidal nature of the loopy polynomial.  The four-vertex trees
$P_4$ and $K_{1,3}$ have the same Tutte polynomial $X^3$,  whereas the formulas~\eqref{eq:path4} and~\eqref{eq:star} give
  \[
    \LL_{P_4}-\LL_{K_{1,3}}=t(x_1^2-x_0x_2).
  \]
Thus $\LL$ is not determined by the cycle matroid.
\end{remark}

\section{Vertex deletion and subgraph profiles}\label{sec:vertex-deletion}

For a loopless graph, $x_0$ marks singleton components of a spanning
forest. Differentiation with respect to $x_0$ gives vertex-deletion
identities and induced-subgraph statistics.

\subsection{The vertex-deletion identity}\label{subsec:degree-sequence}

\begin{proposition}[Vertex-deletion identity]\label{prop:vertex-deletion}
Let $G$ be a loopless graph.  Then
\begin{equation}\label{eq:vertex-deletion}
  \frac{\partial}{\partial x_0}\LL_G(t,\bx)
  = \sum_{v\in V(G)}  t^{d(v)}\LL_{G-v}(t,\bx),
\end{equation}
\end{proposition}

\begin{proof}
Differentiate the forest expansion~\eqref{eq:forests}.
Since $G$ is loopless, a component of a spanning forest contributes $x_0$ precisely when
it consists of a single isolated vertex $v$.  After marking this singleton
component and deleting it, what remains is an arbitrary spanning forest $F'$ of $G-v$.

No edge incident with $v$ is externally active with respect to
$\{v\}\sqcup F'$, since its endpoints lie in different components.
The external activity of every remaining edge is exactly its external activity
in $G-v$.  Passing from $G$ to $G-v$ removes precisely $d(v)$ edges, and therefore
\[
  \wt_G(\{v\}\sqcup F')  =  x_0\,t^{d(v)}\wt_{G-v}(F').
\]
Summing over $v$ and $F'$ gives~\eqref{eq:vertex-deletion}.
If a forest has several isolated vertices, it occurs once for each choice of the marked
singleton, as required by differentiation in~\eqref{eq:vertex-deletion}.
\end{proof}

We will use the following specialization.

\begin{lemma}\label{lem:constant-specialization}
For every graph $G$ with $n$ vertices,
\[
  \left. \LL_G(t,\bx)  \right|_{x_0=x_1=\cdots=(1-t)^{-1}}  =(1-t)^{-n}.
\]
\end{lemma}

\begin{proof}
Work over $\QQ(t)$ and substitute $c=(1-t)^{-1}$.  If $G$ has no non-loop
edges, multiplicativity and the defining initial condition give $c^n$.
If $e$ is a non-loop edge, induction and loopy deletion-contraction give
\[
  c^{n-1}+t c^n =c^n\bigl((1-t)+t\bigr) = c^n.
\]
\end{proof}

\subsection{The degree sequence}\label{subsec:degree-sequence-thm}

\begin{theorem}\label{thm:L-degree-sequence}
The loopy polynomial determines the degree sequence of every loopless graph.
More precisely, if $G$ has $n$ vertices, then
\begin{equation}\label{eq:degree-generating-function}
  \sum_{v\in V(G)}t^{d(v)} = (1-t)^{n-1}
  \left.  \frac{\partial}{\partial x_0}\LL_G(t,\bx) \right|_{x_0=x_1=\cdots=(1-t)^{-1}}.
\end{equation}
\end{theorem}

\begin{proof}
Apply Lemma~\ref{lem:constant-specialization} to every term on the right-hand
side of~\eqref{eq:vertex-deletion}.  Since each $G-v$ has $n-1$ vertices,
\[
  \left. \frac{\partial}{\partial x_0}\LL_G(t,\bx) \right|_{x_i=(1-t)^{-1}}
  = (1-t)^{-(n-1)} \sum_{v\in V(G)}t^{d(v)},
\]
giving~\eqref{eq:degree-generating-function}.  The coefficient of $t^d$ in this polynomial
is the number of vertices of degree $d$.
\end{proof}

The loopless hypothesis is needed for this proof: a variable $x_r$ with
$r>0$ may arise either from a looped singleton or from a nonsingleton
forest component. The refinement separates these cases; see
Theorem~\ref{thm:degree-sequence-loops}.

\subsection{Higher vertex deletion}\label{subsec:induced-edge-profile}

Higher derivatives mark several singleton components simultaneously.

Recall that for $S\subseteq V(G)$, $\kp_G(S)$ is the number of edges of $G$ having
at least one endpoint in $S$, each edge counted once; equivalently, $\kp_G(S)=m(G)-m(G-S)$.

\begin{proposition}[Higher vertex-deletion identity]
\label{prop:higher-vertex-deletion}
Let $G$ be a loopless graph on $n$ vertices.  For $0\le k\le n$,
\begin{equation}\label{eq:higher-vertex-deletion}
  \frac{\partial^k}{\partial x_0^k}\LL_G(t,\bx)
  =  k!\sum_{\substack{S\subseteq V(G)\\ |S|=k}}  t^{\kp_G(S)}\LL_{G-S}(t,\bx).
\end{equation}
\end{proposition}

\begin{proof}
In a loopless graph, differentiating $k$ times marks an ordered choice
of $k$ singleton forest components. For a fixed set $S$ of these
vertices, deleting them leaves an arbitrary spanning forest $F'$ of
$G-S$. Edges incident with $S$ are inactive, and the activities of all
remaining edges are unchanged. Thus, with $F$ obtained by adjoining the
singletons to $F'$,
\[
  \wt_G(F)=x_0^k t^{\kp_G(S)}\wt_{G-S}(F').
\]
There are $k!$ orders of the marked vertices. Summing over $S$ and $F'$
proves~\eqref{eq:higher-vertex-deletion}, including $k=0$ and $k=n$.
\end{proof}

\begin{remark}\label{rem:kappa-bridge}
The function $\kp_G$ in~\eqref{eq:higher-vertex-deletion} is also the
rank function defining the score polytope in Section~\ref{sec:polyhedral}.
\end{remark}

\subsection{Induced edge count and \texorpdfstring{$\kp$}{kappa}-profiles}\label{subsec:induced-profiles}

\begin{corollary}\label{cor:kappa-profile}
Let $G$ be a loopless graph on $n$ vertices.  Then, for $0\le k\le n$,
\begin{equation}\label{eq:kappa-profile}
  \sum_{\substack{S\subseteq V(G)\\ |S|=k}} t^{\kp_G(S)}  = \frac{(1-t)^{n-k}}{k!}   \left.
    \frac{\partial^k}{\partial x_0^k}\LL_G(t,\bx)  \right|_{x_0=x_1=\cdots=(1-t)^{-1}} .
\end{equation}
\end{corollary}

\begin{proof}
Every graph $G-S$ occurring in~\eqref{eq:higher-vertex-deletion} has $n-k$ vertices,
so Lemma~\ref{lem:constant-specialization} gives
\[
  \left.  \LL_{G-S}(t,\bx) \right|_{x_i=(1-t)^{-1}} =(1-t)^{-(n-k)}.
\]
Substitution in~\eqref{eq:higher-vertex-deletion} gives
\eqref{eq:kappa-profile}.
\end{proof}

Recall that $e_G(S)$ denotes the number of edges of the induced subgraph
$G[S]$, counted with multiplicity.  For $0\le s\le n$ and $j\ge 0$ put
\[
  N_G(s,j)  := \left|\{S\subseteq V(G): |S|=s,\ e_G(S)=j\}  \right|.
\]
We call the numerical array $\bigl(N_G(s,j)\bigr)_{s,j}$ the \emph{induced edge count profile} of $G$.
Equivalently, for each $s$ it is encoded by the polynomial
\begin{equation}\label{eq:induced-edge-polynomial}
  \cE_{G,s}(q) := \sum_{\substack{S\subseteq V(G)\\ |S|=s}}q^{e_G(S)} =\sum_{j\ge 0}N_G(s,j)q^j.
\end{equation}

\begin{theorem}\label{thm:induced-edge-profile}
The loopy polynomial of a loopless graph determines its induced edge count profile. More explicitly, define
\[
  \Delta_{G,k}(t) :=  \frac{(1-t)^{n-k}}{k!} \left.
    \frac{\partial^k}{\partial x_0^k}\LL_G(t,\bx)  \right|_{x_0=x_1=\cdots=(1-t)^{-1}} .
\]
Then
\begin{equation}\label{eq:induced-edge-from-L}
  \cE_{G,s}(q)  =  q^{m(G)}\Delta_{G,n-s}(q^{-1}), \qquad 0\le s\le n.
\end{equation}
\end{theorem}

\begin{proof}
By Corollary~\ref{cor:kappa-profile},
\[
  \Delta_{G,n-s}(t) = \sum_{\substack{R\subseteq V(G)\\ |R|=n-s}} t^{\kp_G(R)}.
\]
Put $S=V(G)\setminus R$.  Then $|S|=s$, and an edge is not incident with
$R$ precisely when both of its endpoints lie in $S$.  Hence
\[
  \kp_G(R)=m(G)-e_G(S).
\]
Therefore
\[
  \Delta_{G,n-s}(t)  =  \sum_{|S|=s}t^{m(G)-e_G(S)}.
\]
Replacing $t$ by $q^{-1}$ and multiplying by $q^{m(G)}$ gives
\eqref{eq:induced-edge-from-L}.
\end{proof}

Several familiar graph invariants are immediate consequences.

\begin{corollary}\label{cor:independence-clique}
Let $G$ be a loopless graph.
\begin{enumerate}[(1)]
\item The loopy polynomial determines the independence polynomial
  \begin{equation*}
    I_G(z)=\sum_{s=0}^{n} i_s(G)z^s,
  \end{equation*}
  where $i_s(G)$ is the number of independent sets of cardinality $s$.  Indeed,
  \[
    i_s(G)=[q^0]\cE_{G,s}(q).
  \]

\item If $G$ is simple, then $\LL_G$ determines the clique polynomial.  The
  number $c_s(G)$ of $s$-vertex cliques is
  \[
    c_s(G) = [q^{\binom{s}{2}}]\cE_{G,s}(q).
  \]

\item If $G$ is simple, then $\LL_G$ determines the induced edge count
  profile of the complement.  More precisely,
  \[
    \cE_{\overline G,s}(q)  =  q^{\binom{s}{2}}\cE_{G,s}(q^{-1}).
  \]
\end{enumerate}
\end{corollary}

\begin{proof}
A set $S$ is independent precisely when $e_G(S)=0$, which proves~(1).
For a simple graph, an $s$-vertex set induces $\binom{s}{2}$ edges precisely
when it is a clique, proving~(2).  Finally, (3) follows, since $ e_{\overline G}(S)=\binom{s}{2}-e_G(S) $
for every $s$-vertex set $S$.
\end{proof}

This recovers the induced edge count profile of $\overline G$, not its
full loopy polynomial; the latter determination is
Problem~\ref{prob:complement}.

\section{Relations with other graph polynomials}\label{sec:relations}

We compare $\LL_G$ with Tutte's universal $V$-function, the chromatic
and Tutte symmetric functions, the $U$-polynomial, and the polychromate.

\subsection{Tutte's universal \texorpdfstring{$V$}{V}-function}\label{subsec:universal-V}

A $V$-function with values in a commutative ring $R$ is a multiplicative graph invariant $V$ satisfying
\begin{equation}\label{eq:V-deletion-contraction}
  V_G=V_{G-e}+V_{G_e}
\end{equation}
for every non-loop edge $e$, where $G_e$ denotes ordinary contraction.
Tutte proved that the values of such a function on the graphs $L_r$ may be
prescribed arbitrarily and determine the function uniquely~\cite{Tu1}.
Equivalently, there is a universal polynomial
$ \cV_G(\bv)=\cV_G(v_0,v_1,\ldots)$
characterized by~\eqref{eq:V-deletion-contraction}, multiplicativity, and the initial conditions
$\cV_{L_r}(\bv)=v_r,$ where $L_r$ is the graph consisting of one vertex with $r$ loops.
Bollob\'as, Pebody and Riordan~\cite{BPR} proved that for connected graphs it is enough to impose
the deletion-contraction relation only on edges that are neither bridges nor loops.
The resulting invariant still factors through Tutte's universal $V$-function.

Tutte~\cite{Tu1} expressed the universal $V$-function as a sum over spanning subgraphs. Define
\begin{equation}\label{eq:Tutte-Z}
  Z_G(\bz)=  Z_G(z_0,z_1,\ldots):= \sum_{A\subseteq E(G)} \prod_{D\in c(A)} z_{\nu(D)},
  \quad \text{where} \quad \nu(D):=|E(D)|-|V(D)|+1,
\end{equation}
with each component of a spanning subgraph contributing a variable indexed by its cyclomatic number.
Then $Z_G$ is a binomial transform of $\cV_G$:
\begin{equation}\label{eq:V-Z-binomial}
  Z_G(\bz)  =  \left.   \cV_G(\bv)  \right|_{v_r=\sum_{j=0}^r\binom{r}{j}z_j }
\end{equation}
and so, by binomial inversion,
\[
  \cV_G(\bv)  =  \left.  Z_G(\bz)  \right|_{z_r=\sum_{j=0}^r(-1)^{r-j}\binom{r}{j}v_j}.
\]

An activity expansion for the universal $V$-function was given by Wang and Sachs~\cite{WS} and
by Bollob\'as, Pebody and Riordan~\cite{BPR}.

\begin{theorem}[\cite{WS,BPR}]
\label{thm:V-forest}
For every ordering of $E(G)$,
\begin{equation}\label{eq:V-forest}
  \cV_G(\bv)  =  \sum_{F\in\cF(G)} \prod_{T\in c(F)}v_{\eps_F(T)}.
\end{equation}
In particular, the right-hand side is independent of the edge order.
\end{theorem}

\begin{remark}\label{rem:V-EMM}
Tutte's universal $V$-function differs from the vertex- and edge-weighted
$V$-polynomial of Ellis-Monaghan and Moffatt used for the Potts model in
an external field~\cite{MM}. Specializing the latter to unit vertex
weights and a common edge weight gives an invariant equivalent to $U$;
see~\cite[Theorem~26.12]{Nob}.
\end{remark}

\subsection{The loopy polynomial as a shifted \texorpdfstring{$V$}{V}-function}\label{subsec:loopy-V}

For a graph $G$, denote by $G^\circ$ the graph obtained by adding a new loop at every
vertex of $G$.  We call these the \emph{baseline loops}.  The following identity relates
the loopy polynomial and Tutte's universal $V$-function.

\begin{proposition}\label{prop:loopy-V}
  Let $m=|E(G)|$.  Then
  \begin{equation}\label{eq:loopy-V-t1}
    \LL_G(1,\bx)   =    \left.  \cV_{G^\circ}(\bv)  \right|_{v_r=x_{r-1}\ (r\ge 1)}.
  \end{equation}
  More generally,
  \begin{equation}\label{eq:loopy-V}
    \LL_G(t,\bx)  =  t^m  \left.  \cV_{G^\circ}(\bv)  \right|_{v_r=t^{1-r}x_{r-1}\ (r\ge 1)}.
  \end{equation}
\end{proposition}

\begin{proof}
  The sets of spanning forests of $G^\circ$ and $G$ are the same.
  Since every loop is externally active, the baseline loops add $|V(T)|$ active edges
  to a component $T$. Therefore the number of externally active
  edges assigned to $T$, when forest $F$ is regarded as a forest of $G^\circ$, is
  $  |V(T)|+\eps_F(T)  $. Applying Theorem~\ref{thm:V-forest} we obtain
  \[
    \cV_{G^\circ}(\bv)  =  \sum_{F\in\cF(G)} \prod_{T\in c(F)} v_{|V(T)|+\eps_F(T)}.
  \]
  The substitution $v_r=x_{r-1}$ gives~\eqref{eq:loopy-V-t1}.
  For the second identity, substitute instead
  $v_r=t^{1-r}x_{r-1}$.  The resulting power of
  $t$, before multiplication by $t^m$, is
  \[
    \sum_{T\in c(F)} \bigl(1-|V(T)|-\eps_F(T)\bigr)  =|c(F)|-|V(G)|-\eps(F)   =-|F|-\eps(F).
  \]
  Thus~\eqref{eq:loopy-V} is the same as the forest expansion
  \eqref{eq:forests}.
\end{proof}

There is also a recursive way to see~\eqref{eq:loopy-V-t1}.  If $e=uv$ is
a non-loop edge of $G$, ordinary contraction of $e$ in $G^\circ$ leaves the
two baseline loops at $u$ and $v$ as two loops at the contracted vertex.  On
the other hand, loopy contraction of $e$ in $G$ leaves $e$ itself as a loop,
and passing to $(G/e)^\circ$ adds one baseline loop at the new vertex.  Thus
\[
  G^\circ_e\cong (G/e)^\circ.
\]
In this sense the baseline loops convert ordinary contraction into loopy contraction.

\subsection{The chromatic symmetric function and matchings}\label{subsec:X-and-matchings}

The forest expansion~\eqref{eq:forests} connects
the loopy polynomial with several classical invariants of graphs.

Let $X_G$ denote Stanley's \emph{chromatic symmetric function}~\cite{Sta},
with the convention that $X_G=0$ when $G$ has a loop.
In power-sum coordinates, it is given by
\begin{equation}\label{eq:X-powersum}
  X_G(p) : = \sum_{A\subseteq E(G)}(-1)^{|A|} \prod_{D\in c(A)}p_{|V(D)|}.
\end{equation}
The activity intervals of Lemma~\ref{lem:activity-intervals} allow us to
rewrite~\eqref{eq:X-powersum} in terms of external activities:
\begin{equation}\label{eq:X-zero-activity}
  X_G = \sum_{\substack{F\in\cF(G)\\ \eps(F)=0}}(-1)^{|F|} \prod_{T\in c(F)}p_{|V(T)|}.
\end{equation}
Indeed, the contributions from the interval of $F$ are multiplied by
$\ds \sum_{J\subseteq\EA(F)}(-1)^{|J|}$, which vanishes unless $\EA(F)=\varnothing$.

The zero-activity part can be recovered from the loopy polynomial,
since the external activity of a forest $F$ is determined by the monomial
$\ds M=\prod_{i\ge 0}x_i^{a_i}$ it contributes to the forest expansion~\eqref{eq:forests}:
\begin{equation}\label{eq:activity-from-monomial}
\eps(F)=\operatorname{wdeg}M+\deg M-|V(G)|, \qquad \text{where} \quad
\deg M=\sum_i a_i
\quad  \text{and} \quad
\operatorname{wdeg}M=\sum_i i a_i.
\end{equation}

\begin{theorem}\label{thm:L-determines-X}
The loopy polynomial determines the chromatic symmetric function. Specifically,
\[
  X_G  = (-1)^m  \sum_{\substack{\ba:\\
      \sum_{i\ge 0}(i+1)a_i=n}}  c_{\ba} \prod_{i\ge 0}p_{i+1}^{a_i},
\]
where $n=|V(G)|$, $m=|E(G)|$, and the coefficients $c_\ba$ are given by
\[
  \LL_G(-1,\bx)=\sum_{\ba}c_{\ba} \prod_{i\ge 0}x_i^{a_i}.
\]
\end{theorem}

\begin{proof}
For a monomial $\ds \prod_{i\ge 0}x_i^{a_i}$ in~\eqref{eq:forests} corresponding to a spanning forest $F$,
the condition  $\ds  \sum_i(i+1)a_i=n$ is equivalent by~\eqref{eq:activity-from-monomial} to
$\eps(F)=0$.  For such a forest, the variable $x_i$ comes from a component with $i+1$
vertices. Its contribution to $\LL_G(-1,\bx)$ has the sign  $(-1)^{m-|F|}.$
Multiplying by $(-1)^m$ changes this to $(-1)^{|F|}$, and replacing
$x_i$ by $p_{i+1}$ gives~\eqref{eq:X-zero-activity}.
\end{proof}

Another example is the matching polynomial.  Denote by $m_k(G)$ the number of matchings of
$G$ of size $k$ and let
\begin{equation*}
  M_G(q)=\sum_{k\ge 0}m_k(G)q^k .
\end{equation*}

\begin{proposition}\label{prop:matching}
For a simple graph $G=(V,E)$ the loopy polynomial determines the matching polynomial $M_G$. Specifically,
if $n=|V|$, $m=|E|$, and $0\le k\le \lfloor n/2 \rfloor$, then
\[
  m_k(G)=[t^{m-k}x_1^k x_0^{n-2k}]\,\LL_G(t,\bx).
\]
\end{proposition}

\begin{proof}
  A spanning forest contributes the monomial $x_1^k x_0^{n-2k}$ only if its
  nontrivial components are $k$ disjoint edges.  Since $G$ is simple, such a
  forest has no externally active edges: there is no loop and no second edge
  inside any two-vertex component.  Thus these forests are precisely the
  $k$-edge matchings, and each contributes $t^{m-k}$.
\end{proof}

\subsection{The \texorpdfstring{$U$}{U}-polynomial}
\label{subsec:U}

The ordinary $U$-polynomial of Noble and Welsh~\cite{NW,Nob} is
\begin{equation}\label{eq:U}
  U_G(\bz,y)  =  \sum_{A\subseteq E(G)} \left(\prod_{D\in c(A)}z_{|V(D)|}\right)  (y-1)^{\sum_{D\in c(A)}\nu(D)}.
\end{equation}
The activity intervals give the spanning-forest form of
Noble and Welsh~\cite{NW}; see also~\cite{Nob}.

\begin{proposition}\label{prop:U-forest}
For every graph $G$,
\begin{equation}\label{eq:U-forest}
  U_G(\bz,y)  =  \sum_{F\in\cF(G)}  y^{\eps(F)}  \prod_{T\in c(F)}z_{|V(T)|}.
\end{equation}
\end{proposition}

\begin{proof}
Fix an edge order and apply Lemma~\ref{lem:activity-intervals}, which writes
$2^{E(G)}$ as the disjoint union of the intervals
$[F,F\cup\EA(F)]$ over spanning forests $F$.  Let
$A=F\cup J$ with $J\subseteq\EA(F)$.  Every edge of $J$ has both
endpoints in a single component of $F$, so $(V(G),A)$ and $F$ have the same
components. In particular
\[
  \prod_{D\in c(A)}z_{|V(D)|}=\prod_{T\in c(F)}z_{|V(T)|} \qquad \text{and}
  \qquad  \sum_{D\in c(A)}\nu(D)=|A|-|F|=|J| .
\]
Summing~\eqref{eq:U} over the interval therefore gives
\[
  \Bigl(\prod_{T\in c(F)}z_{|V(T)|}\Bigr)  \sum_{J\subseteq\EA(F)}(y-1)^{|J|}
  =    y^{\eps(F)}\prod_{T\in c(F)}z_{|V(T)|},
\]
since $\eps(F)=|\EA(F)|$.  Summing over $F$
gives~\eqref{eq:U-forest}.
\end{proof}

\subsection{The polychromate}\label{subsec:polychromate}

Brylawski introduced the \emph{polychromate} in his intersection theory for graphs~\cite{Bry}.
We write it as
\begin{equation}\label{eq:polychromate}
  \eta_G(\bq,y) := \sum_{\pi\vdash V(G)} y^{e_G(\pi)} \prod_{B\in\pi}q_{|B|},
\end{equation}
where the sum is over all set partitions of $V(G)$ and $e_G(\pi)$ is the
number of edges of $G$ whose two endpoints lie in the same block of $\pi$.
Edges are counted with multiplicity, and a loop is counted in the block
containing its vertex.  Thus, unlike the $U$-polynomial~\eqref{eq:U},
which is organized by spanning subgraphs and their connected components,
the polychromate is organized by partitions of the vertex set and records the sizes
of their blocks together with the number of edges internal to the blocks.

Sarmiento~\cite{Sar} proved that the polychromate and the $U$-polynomial determine one another.
Her paper is stated for simple graphs.  Merino and Noble~\cite{MN}
subsequently recovered the equivalence in a framework that
allows loops and multiple edges (see also~\cite{Nob}). The
change of coordinates is triangular with respect to refinement of set
partitions, and the coefficients count coarsenings of a fixed partition type.
In the Merino--Noble formulation it may also be understood as a change
between two natural bases of the corresponding Tutte symmetric function.

Both sides of this equivalence admit natural refinements in which the
component or block data are retained separately. These refinements will be discussed in
Section~\ref{sec:refined}. In particular, the extended polychromate
will provide the partition-side description of the refined loopy polynomial
(see Section~\ref{subsec:block-statistics}).

\subsection{The Tutte symmetric function}
\label{sec:tutte-symm-funct}

Recall the definition of the \emph{Tutte symmetric function} $\XB_G$ in Stanley's
normalization~\cite{Sta2}:
\begin{equation}\label{eq:XB-coloring}
  \XB_G(q;\bw): = \sum_{\kp:V(G)\to\ZZ_{>0}}   (1+q)^{b_G(\kp)}\bw_\kp,
\end{equation}
where the sum is over all colorings $\kp$ of $G$ by positive integers, $b_G(\kp)$ is
the number of monochromatic edges of $\kp$, and $\ds \bw_\kp=\prod_{v\in V(G)}w_{\kp(v)}$.
Expanding the factor $(1+q)^{b_G(\kp)}$ over subsets of monochromatic edges
gives the power-sum expansion~\cite[Theorem 3.2]{Sta2}
\begin{equation}\label{eq:XB-subgraph}
  \XB_G(q;\bw)  =   \sum_{A\subseteq E(G)}q^{|A|} \prod_{D\in c(A)}p_{|V(D)|}.
\end{equation}
In particular, $\XB_G(-1;\bw)=X_G$.

The relation of $\XB$ with the $U$-polynomial is explicit and holds for all graphs.
We have
\begin{equation}\label{eq:XB-from-U}
  \XB_G(q;\bw)  =  \left.  U_G(\bz,1+q) \right|_{z_s=q^{s-1}p_s}.
\end{equation}
Conversely, the coefficient of each monomial $z_\lambda$ in $U_G$ is
recovered from the coefficient of $p_\lambda$ in $\XB_G$.
Hence the two invariants are equivalent~\cite{NW,MN,Nob}.
Together with Sarmiento's theorem and its Merino--Noble formulation,
recalled in Section~\ref{subsec:polychromate}, this gives for all graphs the equivalence
\begin{equation}
  \label{eq:U-eta-XB}
  U_G\ \longleftrightarrow\ \eta_G\ \longleftrightarrow\ \XB_G.
\end{equation}

\subsection{Induced-subgraph statistics}
\label{subsec:classical-induced-statistics}

Induced-subgraph statistics can be extracted from the partition and
coloring descriptions. Let $m_\lambda$ denote the monomial symmetric
function, and for an integer partition $\lambda$ put
\[
  a(\lambda)=\prod_{j\ge1}r_j(\lambda)!,
  \qquad r_j(\lambda)=|\{i:\lambda_i=j\}|.
\]
For every graph $G$ and $\lambda\vdash |V(G)|$, the definitions give
\begin{equation}\label{eq:eta-XB-coefficients}
  [m_\lambda]\XB_G(y-1;\bw)
  =a(\lambda)[q_\lambda]\eta_G(\bq,y),
  \qquad q_\lambda=\prod_iq_{\lambda_i}.
\end{equation}
Indeed, for a fixed color monomial of type $\lambda$, each partition
of this type admits exactly $a(\lambda)$ assignments of its blocks
to the specified colors. Its internal edges are exactly the resulting
monochromatic edges. This argument also allows loops and multiple edges.

Now suppose that $G$ is loopless, and put $n=|V(G)|$.
A partition of type $(s,1^{n-s})$, with $s\ge2$, is determined by its
$s$-element block, and all internal edges lie in that block. Hence
\begin{equation}\label{eq:edge-profile-from-polychromate}
  [q_s q_1^{n-s}]\eta_G(\bq,y)=\cE_{G,s}(y),
  \qquad 2\le s\le n.
\end{equation}
Combining this with~\eqref{eq:eta-XB-coefficients} gives
\begin{equation*}
  \cE_{G,s}(y)
  =\frac{1}{(n-s)!}[m_{(s,1^{n-s})}]\XB_G(y-1;\bw),
  \qquad 1\le s\le n.
\end{equation*}
The case $s=1$ in the latter formula is separate: the coefficient of
$m_{(1^n)}$ is $n!$, and $\cE_{G,1}(y)=n$. The analogous
polychromate coefficient is $[q_1^n]\eta_G=1$, so
\eqref{eq:edge-profile-from-polychromate} does not extend to $s=1$.
Also $\cE_{G,0}(y)=1$.

Thus each of $\eta_G$, $U_G$, and $\XB_G$ determines the full
induced edge count profile. For $n\ge2$,
\[
  \cE_{G,n-1}(y)=\sum_{v\in V(G)}y^{m(G)-d_G(v)},
\]
so the degree sequence is recovered as well; the smaller cases are
immediate. Setting $y=0$ recovers independent-set counts, and for simple
graphs the coefficient of $y^{\binom{s}{2}}$ gives the number of
$s$-cliques. Theorem~\ref{thm:induced-edge-profile} obtains the same
profile directly from $\LL_G$ by higher vertex deletion.

For comparison, the specialization of Noble and Welsh
\cite[Theorem~5.2]{NW}, after changing from the stability-polynomial
normalization to the independence polynomial, gives directly
\[
  I_G(z)=\left.U_G(\bz,y)\right|_
  {y=0,\ z_1=1+z,\ z_j=(-1)^{j+1}z^j\ (j\ge2)}
\]
for loopless $G$.

In the next section we introduce a common refinement of $U_G$ and $\LL_G$ that retains
component size and activity separately. How much information the two specializations
lose is taken up in Section~\ref{sec:tutte-symmetric}.

\section{The refined loopy polynomial}\label{sec:refined}

We now retain component size and external activity separately. The
resulting refinement is equivalent to the extended $U$-polynomial and
determines the further statistics studied in this section.

\subsection{The refined loopy polynomial and the extended \texorpdfstring{$U$}{U}-polynomial}%
\label{subsec:refined-loopy}\label{subsec:refined-extended}

\begin{definition}\label{def:refined-loopy}
Fix an ordering of $E(G)$ and introduce commuting variables
\[
  u_{s,r},\qquad s\ge 1,\quad r\ge 0.
\]
The \emph{refined loopy polynomial} of $G$ is, a priori with respect to the
chosen order,
\begin{equation}\label{eq:refined-loopy}
\widehat{\LL}_G(\bu)
:=
\sum_{F\in\cF(G)}
\prod_{T\in c(F)}
u_{|V(T)|,\eps_F(T)}.
\end{equation}
\end{definition}

By Theorem~\ref{thm:V-forest},
\begin{equation}\label{eq:refined-to-V}
  \cV_G(\bv)  = \left.   \widehat{\LL}_G(\bu) \right|_{u_{s,r}=v_r},
\end{equation}
while the forest expansion gives
\begin{equation}\label{eq:refined-to-loopy}
  \LL_G(t,\bx)  = t^{|E(G)|} \left. \widehat{\LL}_G(\bu) \right|_{ u_{s,r}=t^{-(s-1+r)}x_{s-1+r}   }.
\end{equation}
Order-independence follows from the comparison below with the extended
$U$-polynomial.

Recall that $\nu(D)$ denotes the cyclomatic number of a component $D$ of a spanning subgraph.
Following Merino and Noble~\cite{MN},
we use the term \emph{extended $U$-polynomial} for
\begin{equation}\label{eq:extended-U}
  U_G^{\ext}(\bz)  := \sum_{A\subseteq E(G)} \prod_{D\in c(A)} z_{|V(D)|,\nu(D)}.
\end{equation}
It is called the \emph{strong} $U$-polynomial in~\cite{Mar,Nob}.
Merino and Noble proved its equivalence with the extended polychromate
of Section~\ref{subsec:block-statistics}.

\begin{theorem}\label{thm:refined-extended-U}
For every graph $G$,
\begin{equation}\label{eq:refined-to-extended-U}
  U_G^{\ext}(\bz)  = \left.  \widehat{\LL}_G(\bu)  \right|_{ u_{s,r}=\sum_{j=0}^r\binom{r}{j}z_{s,j} }.
\end{equation}
Conversely,
\begin{equation}\label{eq:extended-U-to-refined}
  \widehat{\LL}_G(\bu) = \left. U_G^{\ext}(\bz) \right|_{ z_{s,r}=\sum_{j=0}^r(-1)^{r-j}\binom{r}{j}u_{s,j} }.
\end{equation}
In particular, $\widehat{\LL}_G$ is independent of the edge ordering, and
$\widehat{\LL}$ and the extended $U$-polynomial contain the same information.
\end{theorem}

\begin{proof}
Fix an edge order and apply Lemma~\ref{lem:activity-intervals}.  Let $F$ be a
spanning forest and let $T$ be one of its components.  Put
\[
  s=|V(T)|, \qquad r=\eps_F(T).
\]
Every edge in $\EA(F)$ assigned to $T$ has both endpoints in $T$.
If $j$ of these $r$ active edges are added to $F$, then the corresponding component of the resulting
spanning subgraph still has $s$ vertices and has cyclomatic number $j$: the tree $T$
contributes $s-1$ edges and the additional active edges contribute $j$ independent excess edges.
Therefore the total contribution of the interval $[F,F\cup\EA(F)]$ to~\eqref{eq:extended-U} is
\[
  \prod_{T\in c(F)} \left(\sum_{j=0}^{\eps_F(T)}\binom{\eps_F(T)}{j} z_{|V(T)|,j} \right).
\]
Summing over $F$ gives~\eqref{eq:refined-to-extended-U}.

For each fixed first index $s$, the substitution in
\eqref{eq:refined-to-extended-U} is the ordinary binomial transform in the
second index.  Binomial inversion therefore gives
\eqref{eq:extended-U-to-refined}.  Since the extended $U$-polynomial is
independent of the edge order, the same is true of $\widehat{\LL}_G$.
\end{proof}

Theorem~\ref{thm:refined-extended-U} is the component-size refinement of the
classical transform~\eqref{eq:V-Z-binomial}.  Indeed, forgetting the first
index in~\eqref{eq:extended-U} gives Tutte's spanning-subgraph sum $Z_G$~\eqref{eq:Tutte-Z}, while forgetting the first index in~\eqref{eq:refined-loopy} gives the
Wang--Sachs forest expansion~\eqref{eq:V-forest}.  Thus the same binomial
transform occurs before and after component sizes are retained.

\begin{corollary}[Specializations]
\label{cor:refined-specialized}
For every graph $G$, the ordinary $U$-polynomial and the Tutte symmetric
function are specializations of $\widehat{\LL}_G$:
\begin{equation}\label{eq:U-refined}
 U_G(\bz,y)=\left.\widehat{\LL}_G(\bu)\right|_{u_{s,r}=z_sy^r},
\end{equation}
\[
 \XB_G(q;\bw)=\left.\widehat{\LL}_G(\bu)\right|_
 {u_{s,r}=q^{s-1}(1+q)^rp_s}.
\]
In particular,
\begin{equation*}
 \XB_G(q;\bw)=\sum_{F\in\cF(G)}q^{|F|}(1+q)^{\eps(F)}
                   \prod_{T\in c(F)}p_{|V(T)|}.
\end{equation*}
If $m=|E(G)|$, then the loopy polynomial is obtained directly from the
extended $U$-polynomial by
\[
 t^{-m}\LL_G(t,\bx)=\left.U_G^{\ext}(\bz)\right|_{z_{s,r}=\sum_{j=0}^r(-1)^{r-j}\binom{r}{j} t^{-(s-1+j)}x_{s-1+j}}.
\]
\end{corollary}

\begin{proof}
Under $z_{s,j}\mapsto z_s(y-1)^j$, the binomial transform in
\eqref{eq:refined-to-extended-U} becomes
\[
 \sum_{j=0}^r\binom{r}{j}z_s(y-1)^j=z_sy^r,
\]
which gives~\eqref{eq:U-refined}. Substituting $z_s=q^{s-1}p_s$ and
$y=1+q$ gives the $\XB$ formula by~\eqref{eq:XB-from-U}; its activity
expansion uses $\sum_{T\in c(F)}(|V(T)|-1)=|F|$. Finally, composing
binomial inversion~\eqref{eq:extended-U-to-refined} with the graded
diagonal specialization~\eqref{eq:refined-to-loopy} gives the last formula.
\end{proof}

Thus the ordinary $U$-polynomial retains component sizes and total
activity, while the loopy polynomial retains the diagonal
indices $s-1+r$.

\begin{remark}\label{rem:X-zero-activity-part}
The chromatic symmetric function is the zero-activity specialization
\[
 X_G=\left.\widehat{\LL}_G\right|_
 {u_{s,0}=(-1)^{s-1}p_s,\ u_{s,r}=0\ (r>0)}.
\]
Indeed, the product of the signs over the forest components is
$(-1)^{|F|}$, so the identity follows from~\eqref{eq:X-zero-activity}.
Theorem~\ref{thm:L-determines-X} recovers this part after the diagonal
specialization to $\LL_G$.
\end{remark}

\begin{remark}[Comparison with the universal $V$-function]
  On loopless multigraphs, $U$ and $\cV$ need not determine one another.
  The forests $P_4$ and $K_{1,3}$ have the same $\cV$ but distinct
  $U$-polynomials. Conversely, the two graphs on $\{0,1,2,3,4,5\}$ with
  edge multisets
  \[
    \{03,04,05,14^3,15,24,35,45\},\qquad
    \{03,04,05,14,15,24,35^3,45\}
  \]
  have equal $U$-polynomials but distinct $\cV$-functions: only the first
  has a spanning subgraph consisting of two unicyclic components.
  The equivalence of $U$, $\eta$, and $\XB$ gives the same comparison
  with the other two ordinary invariants.
\end{remark}

\subsection{The extended polychromate and block statistics}\label{subsec:block-statistics}

The \emph{extended polychromate}, introduced as the strong polychromate
by Bollob\'as and Riordan~\cite[Section~3]{BR}, is
\[
  \eta^{\ext}_G(\bq)=\sum_{\pi\vdash V(G)}
     \prod_{B\in\pi}q_{|B|,e_G(B)}.
\]
By~\cite[Corollary~4.6]{MN} and
Theorem~\ref{thm:refined-extended-U}, it is equivalent to
$\widehat{\LL}_G$.

\begin{proposition}[Block statistics]\label{prop:block-statistics}
Let $G$ be a graph, possibly with loops and multiple edges, and let $b\ge1$.
Then $\widehat{\LL}_G$ determines the number of partitions of $V(G)$
into $b$ blocks realizing any prescribed multiset
\[
  \bigl\{(|B|,e_G(B))\bigr\}_{B\in\pi}
\]
of block sizes and induced edge counts.
\end{proposition}

\begin{proof}
Each such partition contributes a monomial of ordinary $\bq$-degree $b$
to $\eta^{\ext}_G$. Its factors record the stated multiset, and its
coefficient counts the partitions realizing that multiset.
\end{proof}

Crew's generalized degree polynomial~\cite{Crew} is
\begin{equation}\label{eq:gdp}
  \GD_G(q,r,t)=\sum_{S\subseteq V(G)}q^{|S|}r^{e_G(S)}
                        t^{\kp_G(S)-e_G(S)}.
\end{equation}
The exponent of $t$ counts edges with exactly one endpoint in $S$.
Setting $t=1$ gives the induced edge count profile; setting $r=t$ gives
the distribution of $\kp_G(S)$ by subset size. Both are therefore
contained in this joint statistic, which the block data recover as follows.

\begin{corollary}\label{cor:gdp}
For every graph $G$, the refined loopy polynomial $\widehat{\LL}_G$ determines the
generalized degree polynomial $\GD_G$.
More generally, it determines all higher-order analogues obtained by applying the formula
of~\cite[Definition~5.1]{LT} to $G$, with edges counted with multiplicity.
\end{corollary}

\begin{proof}
For $k\ge 1$, write the higher-order polynomial in our notation as
\[
  \GD_G^{(k)}(\mathbf q,\mathbf r,t)  =\sum_{B_0\sqcup\cdots\sqcup B_k=V(G)}
  t^{\,m-\sum_{i=0}^k e_G(B_i)} \prod_{i=1}^k q_i^{|B_i|}r_i^{e_G(B_i)}.
\]
The sum is over ordered tuples of pairwise disjoint sets, each possibly empty;
$B_0$ is the complementary remainder. The exponent of $t$ counts edges joining
different sets, each once. In particular, $\GD_G^{(1)}(q_1,r_1,t)=\GD_G(q_1,r_1,t)$.
For forests this is the polynomial of~\cite[Definition~5.1]{LT}.
Here we use the same formula for arbitrary graphs.

Assume first that $V(G)\ne\varnothing$. For each partition $\pi$ of $V(G)$ into $b\le k+1$ nonempty
blocks, assign its blocks injectively to the labeled slots $0,1,\ldots,k$ and fill every unused slot
with $\varnothing$.  These assignments produce each tuple in the sum exactly once.
Their total contribution depends only on $m$ and the multiset $\{(|B|,e_G(B)):B\in\pi\}$,
which are determined by $\widehat{\LL}_G$ and Proposition~\ref{prop:block-statistics}.
Distinct blocks with the same statistics are still assigned separately, so summing over
all injections gives the correct multiplicities. Summing over $b$ proves the claim.
If $V(G)=\varnothing$, the unique all-empty tuple contributes $1$.
\end{proof}

\begin{remark}[Crew's conjecture]\label{rem:gdp-ordinary}
The ordinary invariants do not suffice. The pairs in
Example~\ref{ex:loop-necessary} and Theorem~\ref{thm:MN-solved} have equal
$U$-polynomials, and hence equal Tutte symmetric functions, but distinct $\GD$.
Thus $\GD_G$ is not determined by $U_G$ on graphs with loops, nor on loopless multigraphs.

The question for simple graphs has been formulated by Crew in~\cite{Crew}.
In a footnote to the conjecture about trees
she asks whether $\XB_G$ determines $\GD_G$ for ``all'' graphs.
Crew has confirmed to us that she meant simple graphs.

Call a pair of graphs with equal $U$-polynomials but distinct extended
$U$-polynomials a \emph{Merino--Noble pair} (see~\eqref{eq:MN-open-problem}).
If no Merino--Noble pair exists among simple graphs, then $U_G$ determines
$U^{\ext}_G$ and hence $\widehat{\LL}_G$, so Corollary~\ref{cor:gdp} would give Crew's conjecture.
Proposition~\ref{prop:computational-evidence} establishes this conclusion
for simple graphs on at most eleven vertices.
\end{remark}

\subsection{The Ising and homomorphism polynomials}
\label{subsec:ising}

The two-block statistics also describe the order-two homomorphism
polynomial. This identifies Crew's question with a question of
Markstr\"om.

For $q\ge 2$ let
\begin{equation}\label{eq:hom-poly}
  P_q(G)  = \sum_{\varphi:V(G)\to[q]} \ \prod_{v\in V(G)}x_{\varphi(v)} \ \prod_{uv\in E(G)}y_{\varphi(u)\varphi(v)}
\end{equation}
denote the \emph{homomorphism polynomial of order $q$}~\cite{Mar}: the
generating function for homomorphisms of $G$ into the complete graph on $q$
vertices with a loop at every vertex, all vertices and edges carrying formal weights.
For $q=2$ it also determines the \emph{bivariate Ising polynomial} of
Andr\'en and Markstr\"om~\cite{AM},
\begin{equation}
\label{eq:ising}
Z(G,x,y) = \sum_{\sigma:V(G)\to\{\pm1\}}x^{E(\sigma)}y^{M(\sigma)},  \qquad
E(\sigma)=\sum_{uv\in E(G)}\sigma(u)\sigma(v),   \qquad   M(\sigma)=\sum_{v\in V(G)}\sigma(v),
\end{equation}
the partition function of the Ising model in an external field.
Neither of these is a specialization of the Tutte polynomial,
nor does either of them determine $T$. Indeed, for loopless graphs $Z$ determines
the degree sequence, which $T_G$ does not, while $T_G$ determines the chromatic polynomial,
which $P_2$ does not~\cite{AM,Mar}.

\begin{proposition}\label{prop:ising}
For every graph $G$, the homomorphism polynomial $P_2(G)$ and the generalized degree
polynomial $\GD_G$  determine one another. Consequently the refined loopy polynomial
$\widehat{\LL}_G$ determines $P_2(G)$ and $Z(G,x,y)$.
\end{proposition}

\begin{proof}
  A map $\varphi:V(G)\to\{1,2\}$ is the same as a subset
  $S=\varphi^{-1}(1)$, and its contribution to~\eqref{eq:hom-poly} is
  \[
    x_1^{|S|}\,x_2^{\,n-|S|}\,  y_{11}^{\,e_G(S)}\,y_{22}^{\,e_G(\overline S)}\, y_{12}^{\,\kp_G(S)-e_G(S)} .
  \]
  Hence $P_2(G)$ records the joint distribution of
  $\bigl(|S|,e_G(S),e_G(\overline S)\bigr)$ over vertex subsets $S\subseteq V(G)$.
  Since
  \[
    \kp_G(S)-e_G(S)=m-e_G(S)-e_G(\overline S),
  \]
  we only need to show that $m$ is determined by $P_2$ and $\GD$.
  This is immediate: the set $S=V(G)$  contributes $q^n r^m$  to $\GD$ and $x_1^n y_{11}^m$ to $P_2$.
  Thus  $P_2(G)$ and $\GD_G$ determine one another.

  To connect them with $Z$, for $\sigma:V(G)\to\{\pm 1\}$ in~\eqref{eq:ising} put  $S=\sigma^{-1}(+1)$.
  Then $M(\sigma)=2|S|-n$ and $E(\sigma)=m-2\bigl(\kp_G(S)-e_G(S)\bigr),$
  since an edge contributes $+1$ to $E(\sigma)$ when its endpoints lie on
  the same side of the partition and $-1$ otherwise. Therefore
  \[
    Z(G,x,y)=x^{m}y^{-n}\,\GD_G\bigl(y^{2},\,1,\,x^{-2}\bigr),
  \]
  so the bivariate Ising polynomial is the specialization $r=1$ of~\eqref{eq:gdp}, up to a monomial factor.

  Corollary~\ref{cor:gdp} now gives the last assertion.
\end{proof}

\begin{remark}[Strictness of the two determinations]\label{rem:ising-strict}
  Proposition~\ref{prop:ising} exhibits the chain
  \[
    \widehat{\LL}_G \ \Longrightarrow \ P_2(G)\equiv\GD_G  \ \Longrightarrow \  Z(G,x,y),
  \]
  and both implications are strict. For the first, the graphs \verb+G?qbrg+ and \verb+G?ovdW+,
  non-isomorphic on eight vertices with twelve edges and degree
  sequence $(2,2,3,3,3,3,4,4)$, have the same $P_2$ but different loopy polynomials.
  For the second, \verb+F?rDo+ and \verb+F?qaw+, non-isomorphic on seven vertices with eight edges
  and degree sequence $(1,1,2,2,3,3,4)$, have
  the same bivariate Ising polynomial but different generalized degree polynomials.
  In these graphs, the numbers of subsets $S$ with
  $|S|=3$, $e_G(S)=0$ and $\kp_G(S)-e_G(S)=5$  are, respectively, $2$ and $3$.
  The cut size determines $e_G(S)+e_G(\overline S)=m-\bigl(\kp_G(S)-e_G(S)\bigr)$, but not the two summands separately.
  On seven vertices there are $853$ connected simple graphs, $853$ distinct generalized
  degree polynomials, and only $811$ distinct bivariate Ising polynomials.
\end{remark}

\begin{remark}[Rotor swaps]\label{rem:rotors}
Tutte's rotor construction~\cite{Tu3} produces non-isomorphic graphs with equal Tutte
polynomials. Markstr\"om and Wierman~\cite{MW} use rotors to produce large
families of graphs sharing several other invariants.
However, this operation need not preserve the loopy polynomial, as the explicit pair below shows.

Let $R$ be a graph with three distinguished vertices $a_1,a_2,a_3$ admitting an
automorphism $\sigma$ that cyclically permutes them:
$\sigma: a_1\mapsto a_2 \mapsto a_3\mapsto a_1$.
Let $R'$ be $R$ with the labels $a_2$ and $a_3$ exchanged.  Given a graph $H$ with
three distinguished vertices $b_1,b_2,b_3$, denote by $H[R]$ the graph obtained
by identifying $a_i$ with $b_i$ for each $i$, and by $H[R']$ the graph obtained from $R'$ in the same way.
Tutte~\cite{Tu3} proved that $T_{H[R]}=T_{H[R']}$ for every such $H$.

For $\omega\subseteq E(R)$ let $s_i(\omega)$ be the order of the component of
$(V(R),\omega)$ containing $a_i$, and let $\mu(\omega)$ be the multiset of
orders of the components containing none of the $a_i$.
Every $A\subseteq E(H[R])$ splits uniquely as $A\cap E(H)$ together with
$\omega=A\cap E(R)$, and the multiset of component orders of $(V(H[R]),A)$
is determined by $A\cap E(H)$, the partition of $\{1,2,3\}$ induced by
$\omega$, the triple $(s_1,s_2,s_3)$ and $\mu(\omega)$.  Since $U_G$ is a sum
over subsets of $E(G)$ recording component orders, this gives a criterion.
Consider the multiset
\begin{equation}\label{eq:rotor-criterion}
  \mathcal{T}(R) \ :=\ \Bigl\{\ \bigl(|\omega|,\ (s_1,s_2,s_3),\ \mu(\omega)\bigr)
  \ :\ \omega\subseteq E(R),\ \ a_1,a_2,a_3 \text{ in three distinct components}\ \Bigr\}.
\end{equation}
If $\mathcal{T}(R)$ is invariant under exchanging $s_2$ and $s_3$, then $U_{H[R]}=U_{H[R']}$ for every $H$.

Only the separating $\omega$ enter $\mathcal{T}(R)$, because the others are already
accounted for: the three ways of merging two of the $a_i$ form a single $\sigma$-orbit,
so their contributions coincide and the exchange permutes them. When all three merge,
then the triple is constant.
On separating subsets, $\sigma$ gives invariance under cyclic
permutation of $(s_1,s_2,s_3)$; this does not by itself imply invariance
under a transposition.
For the Tutte polynomial the same bookkeeping applies with the ordered triple replaced
by the number of components, on which the exchange acts trivially---that is Tutte's theorem.
The gap between the two is the passage from a count
of components to recording their orders.

Construct $R_0$ from the hexagon $a_1c_1a_2c_2a_3c_3$ by adding,
for each edge $a_ic_i$, a new internally disjoint path of length two
with the same endpoints. The resulting graph has nine vertices and
twelve edges; rotation by two steps induces $\sigma$, but no automorphism
fixes $a_1$ and exchanges $a_2,a_3$.
Here $\mathcal{T}(R_0)$ fails to be invariant already at $|\omega|=4$,
where the record with triple $(1,2,3)$ and $\mu=\{1,1,1\}$ occurs three times
and the one with $(1,3,2)$ occurs twice.  Taking for $H$ three disjoint paths of lengths $0,1,2$, with
$b_1,b_2,b_3$ as their respective initial vertices,
gives non-isomorphic graphs on twelve vertices and fifteen edges given by graph6 strings
\begin{center}
  \verb+KDpcaOcO@??@+ \quad\text{and}\quad \verb+KDpcaOcGA??@+.
\end{center}
These graphs have equal Tutte polynomials, equal bivariate Ising polynomials $Z(G,x,y)$,
and equal induced edge count profiles $N_G(s,j)$, but different $\LL_G$, $\widehat{\LL}_G$,
and $U_G$. Three further pairs, on twelve and thirteen vertices, arise in the same way
from other small rotors.

The search considered the following family.  Fix a permutation
$\sigma$ of order three whose vertex orbits are the terminals
$\{a_1,a_2,a_3\}$, one or two further free orbits $\{u_1,u_2,u_3\}$ and
$\{v_1,v_2,v_3\}$, and at most one fixed vertex $c$. Let $R$ range over the
connected graphs on that vertex set whose edge set is a union of $\sigma$-orbits of edges,
and keep only those admitting no automorphism that fixes $a_1$ and exchanges $a_2$ and $a_3$.
There are $8\,064$ such $R$, on nine or ten vertices and with $12,15,\dots,33$ edges.
Candidates built from a single inner orbit admit a terminal reflection and are therefore excluded.
For none of them is $\mathcal{T}(R)$ invariant, and $R_0$ is the smallest.
Indeed, $|\omega|$ is one of the coordinates of the record, so $\mathcal{T}(R)$
splits as a disjoint union over $|\omega|$ and it is enough to exhibit one
value of $|\omega|$ whose slice is not invariant. For every $R$ in the family
this already happens at $|\omega|=4$ or $|\omega|=5$.
Whether a rotor without a terminal-reflection automorphism outside this family
satisfies the criterion remains open.
\end{remark}

The maximum cut size is the largest exponent of $t$ in $\GD_G$ and
is therefore determined by $\widehat{\LL}_G$; compare~\cite[Theorem~4.3]{AM}.

\begin{corollary}\label{cor:crew-markstrom}
On any class of graphs, $U_G$ determines $\GD_G$ if and only if $U_G$ determines $P_2(G)$.
In particular, on simple graphs the conjecture of Crew in Remark~\ref{rem:gdp-ordinary}
and Problem~7.6.4 of Markstr\"om~\cite{Mar} are the same question.
\end{corollary}

\begin{proof}
Immediate from Proposition~\ref{prop:ising} and the equivalence of $U_G$ and $\XB_G$~\cite{NW,Sar,MN,Nob}.
\end{proof}

\begin{remark}
\label{rem:markstrom-problem}
By Corollary~\ref{cor:crew-markstrom} and Remark~\ref{rem:gdp-ordinary},
Proposition~\ref{prop:computational-evidence} settles Problem~7.6.4 of Markstr\"om~\cite{Mar}
affirmatively for simple graphs on at most eleven vertices.

The polynomial $P_2$, however, does not determine $U$ or $\LL$: the pair of
Remark~\ref{rem:ising-strict} has equal $P_2$ but unequal loopy and $U$-polynomials.
Among eight-vertex graphs there are $29$ classes of non-isomorphic graphs with
equal $P_2$, and only $8$ with equal $\LL$.
\end{remark}

\begin{remark}[Failure of determination for $P_3$]\label{rem:P3-boundary}
The relation $\widehat{\LL}_G\to P_2$ of Proposition~\ref{prop:ising} does not extend to $P_q$ with $q\ge 3$.
A map $\varphi:V(G)\to[q]$ is   an ordered partition $V_1,\dots,V_q$ of $V(G)$ into possibly empty blocks,
and $P_q(G)$ records the block sizes $|V_i|$, the internal edge counts $e_G(V_i)$,
\emph{and} the individual cross counts $\bigl|[V_i,V_j]\bigr|$ for every pair $i<j$.
When $q=2$ there is a single cross count and it is forced, being
$m-e_G(V_1)-e_G(V_2)$.  When $q\ge 3$ only the total
$\sum_{i<j}\bigl|[V_i,V_j]\bigr|=m-\sum_i e_G(V_i)$ is forced, and the individual counts are new data.
The block description alone does not recover these individual cross counts.

This is not merely a gap in the argument.  Each of the eight collision classes of $\widehat{\LL}$
on eight vertices is separated by $P_3$, so $\widehat{\LL}_G$ does not determine $P_3(G)$
(compare with~\cite[Observation~7.6.3]{Mar}).   Thus $q=2$ is the boundary after which the homomorphism
polynomials cease to be determined by $\widehat{\LL}$.
\end{remark}

\subsection{Degree sequences of graphs with loops}\label{subsec:degree-loops}
Recall that $\delta_G(v)=d_G(v)-2\ell_G(v)$ counts incident non-loop edges.

It is convenient to restore to $\widehat{\LL}_G$ the grading that $\LL_G$ carries.  Put
\begin{equation}\label{eq:graded-refined}
  \widetilde{\LL}_G(t,\bu)  := \sum_{F\in\cF(G)} t^{\,m(G)-|F|-\eps(F)} \prod_{T\in c(F)}u_{|V(T)|,\,\eps_F(T)} ,
\end{equation}
so that $\LL_G(t,\bx)=\widetilde{\LL}_G(t,\bu)\big|_{u_{s,r}=x_{s-1+r}}$
and, by~\eqref{eq:refined-to-loopy},
$\widetilde{\LL}_G(t,\bu)=t^{m(G)}\widehat{\LL}_G(\bu) \big|_{u_{s,r}\mapsto t^{-(s-1+r)}u_{s,r}}$.
Since $n(G)$ and $m(G)$ are recoverable from $\widehat{\LL}_G$, the
scaling identity shows that $\widetilde{\LL}_G$ and $\widehat{\LL}_G$
determine one another. We use the graded form to state vertex deletion.

\begin{proposition}[Refined vertex-deletion identity]
\label{prop:refined-vertex-deletion}
Let $G$ be a graph, possibly with loops and multiple edges, and let $j\ge 0$.
Then
\begin{equation}
  \label{eq:refined-vertex-deletion}
  \frac{\partial}{\partial u_{1,j}}\widetilde{\LL}_G(t,\bu)   =
  \sum_{\substack{v\in V(G)\\ \ell_G(v)=j}} t^{\,\delta_G(v)}\,\widetilde{\LL}_{G-v}(t,\bu).
\end{equation}
\end{proposition}

\begin{proof}
Differentiate~\eqref{eq:graded-refined}.  A component $T$ of a spanning
forest $F$ contributes the variable $u_{1,j}$ precisely when $T=\{v\}$ is a
singleton with $\eps_F(\{v\})=j$.  Now if $v$ is isolated in $F$, the
edges assigned to $\{v\}$ that are externally active are exactly the loops at $v$. Indeed,
a loop is always active and is assigned to the component of its vertex,
whereas a non-loop edge at $v$ has its endpoints in distinct components of
$F$ and so is neither active nor assigned to $\{v\}$.
Hence $\eps_F(\{v\})=\ell_G(v)$ for every such $F$, and the singleton
components contributing $u_{1,j}$ are the vertices with $\ell_G(v)=j$.

Fix a vertex $v$ with $\ell_G(v)=j$.  Marking the component $\{v\}$ and
deleting it leaves an arbitrary spanning forest $F'$ of $G-v$, and this
correspondence is a bijection.  Every edge with both endpoints in
$V(G)\setminus\{v\}$ has the same external activity status with respect to
$F$ in $G$ as with respect to $F'$ in $G-v$, since its fundamental cycle avoids $v$.
The componentwise activities of the components inherited from $F'$ are unchanged.  Therefore
\[
  \eps(F)=\ell_G(v)+\eps(F'),  \qquad  |F|=|F'| .
\]
Deleting $v$ removes the $\ell_G(v)$ loops at $v$ together with the
$\delta_G(v)$ non-loop edges at $v$, so
$m(G)=m(G-v)+\delta_G(v)+\ell_G(v)$, and the two occurrences of $\ell_G(v)$
cancel:
\[
  m(G)-|F|-\eps(F)  =  \delta_G(v)+\bigl(m(G-v)-|F'|-\eps(F')\bigr).
\]
The term of $F$ in $\widetilde{\LL}_G$ is thus $u_{1,\ell_G(v)}\,t^{\delta_G(v)}$ times the term of $F'$
in $\widetilde{\LL}_{G-v}$.  Summing over $v$ with $\ell_G(v)=j$ and over $F'$
gives~\eqref{eq:refined-vertex-deletion}.  A forest with several such singleton components
occurs once for each choice of the marked component.
\end{proof}

For loopless $G$ only $j=0$ occurs, $u_{1,0}$ specializes to $x_0$ and $\delta_G(v)=d_G(v)$,
so~\eqref{eq:refined-vertex-deletion} reduces to~\eqref{eq:vertex-deletion}.

\begin{theorem}\label{thm:degree-sequence-loops}
Let $G$ be a graph on $n$ vertices.
For every $j\ge 0$,
\begin{equation}
  \label{eq:loop-degree-generating-function}
  \sum_{\substack{v\in V(G)\\ \ell_G(v)=j}} t^{\,\delta_G(v)}   = (1-t)^{\,n-1}
  \left.   \frac{\partial}{\partial u_{1,j}}\widetilde{\LL}_G(t,\bu)  \right|_{u_{s,r}=(1-t)^{-1}} .
\end{equation}
Therefore the refined loopy polynomial $\widehat{\LL}_G$ determines the
joint distribution of the pair $\bigl(\ell_G(v),\delta_G(v)\bigr)$ over
$v\in V(G)$, and in particular the degree sequence of $G$.
\end{theorem}

\begin{proof}
Under $u_{s,r}=x_{s-1+r}$ the polynomial $\widetilde{\LL}_H$ becomes $\LL_H$,
so setting every $u_{s,r}$ equal to $(1-t)^{-1}$ is the specialization of
Lemma~\ref{lem:constant-specialization}, and
\[
  \widetilde{\LL}_H(t,\bu)\big|_{u_{s,r}=(1-t)^{-1}}=(1-t)^{-N}
\]
for every graph $H$ on $N$ vertices.

Apply this to each summand on the right-hand side
of~\eqref{eq:refined-vertex-deletion}.  Each $G-v$ has $n-1$ vertices, so
\[
  \left.   \frac{\partial}{\partial u_{1,j}}\widetilde{\LL}_G(t,\bu)
  \right|_{u_{s,r}=(1-t)^{-1}}  =  (1-t)^{-(n-1)}
  \sum_{\substack{v\in V(G)\\ \ell_G(v)=j}} t^{\,\delta_G(v)} ,
\]
which gives~\eqref{eq:loop-degree-generating-function}.  The coefficient of
$t^{\delta}$ in the left-hand side of~\eqref{eq:loop-degree-generating-function}
is the number of vertices $v$ with $\ell_G(v)=j$ and
$\delta_G(v)=\delta$. Letting $j$ vary gives the joint distribution, and
$d_G(v)=\delta_G(v)+2\ell_G(v)$ gives the degree sequence.  Finally
$\widehat{\LL}_G$ determines $\widetilde{\LL}_G$, as observed
after~\eqref{eq:graded-refined}.
\end{proof}

\begin{remark}[The degree sequences for $\LL$ and  $\widehat{\LL}$]
\label{rem:degree-L-vs-Lhat}
The proof of Theorem~\ref{thm:degree-sequence-loops} uses the singleton
markers $u_{1,j}$, which are lost under the diagonal specialization to $\LL_G$.
It therefore does not settle Problem~\ref{prob:loop-degree}.
\end{remark}

A partial result is available.
Proposition~\ref{prop:extremal-terms} already extracts
from $\LL_G$ the multiset $\{\ell_G(v)\}$ of loop multiplicities, and the
next layer of the forest expansion goes one step further.  Write
$\mu_G(u,v)$ for the number of edges joining distinct vertices $u$ and $v$.

\begin{lemma}\label{lem:second-layer}
Let $G$ be a graph on $n$ vertices with $m$ edges.  The part of $\LL_G$ of ordinary $\bx$-degree $n-1$ is
\begin{equation}\label{eq:second-layer}
  \sum_{\{u,v\}}\ \sum_{i=1}^{\mu_G(u,v)}  t^{\,m-\ell(G)-i}\; x_{\,\ell_G(u)+\ell_G(v)+i}  \prod_{w\ne u,v}x_{\ell_G(w)},
\end{equation}
where the outer sum is over unordered pairs of distinct adjacent vertices.
So $\LL_G$ determines the multiset of pairs
$\bigl(\{\ell_G(u),\ell_G(v)\},\,\mu_G(u,v)\bigr)$ over such pairs,
and in particular the sums
$\ds \sum_{\substack{v\in V(G)\\ \ell_G(v)=j}} \delta_G(v)$ for all $j\ge 0$.
\end{lemma}

\begin{proof}
A forest term has ordinary $\bx$-degree equal to
$n-|F|$,
its number of components, so degree $n-1$ forces $F=\{e\}$ for a single non-loop edge
$e=uv$.  The components of $F$ are $T_e=\{u,v\}$, with
$|E(T_e)|=1$, and the singletons $\{w\}$ for $w\ne u,v$.  Relative to $F$
every loop is externally active and is assigned to the component of its
vertex. A non-loop edge other than $e$ is externally active only if its two
endpoints lie in one component of $F$, that is, only if it joins $u$ to $v$,
and then it is active precisely when it precedes $e$ in the chosen order.
If $e$ is the $i$-th smallest of the $\mu_G(u,v)$ edges joining $u$ to $v$,
there are $i-1$ such, so
\[
  \eps_F(T_e)=\ell_G(u)+\ell_G(v)+i-1,  \qquad
  \eps_F(\{w\})=\ell_G(w),  \qquad   \eps(F)=\ell(G)+i-1 .
\]
Substituting into~\eqref{eq:forests} gives~\eqref{eq:second-layer}.

For the second assertion, note that $\ell(G)$ and the multiset
$\Lambda=\{\ell_G(v)\}$ are known by Proposition~\ref{prop:extremal-terms},
so the exponent of $t$ in a monomial of~\eqref{eq:second-layer} determines
$i$.  Suppose two pairs of loop numbers $\{a,b\}$ and $\{c,d\}$ drawn from
$\Lambda$ yield the same monomial, that is
$(\Lambda\setminus\{a,b\})\cup\{a+b+i\}
=(\Lambda\setminus\{c,d\})\cup\{c+d+i\}$
as multisets.  Adjoining $\{a,b\}\cup\{c,d\}$ to both sides and cancelling
$\Lambda$ leaves
$\{c,d\}\cup\{a+b+i\}=\{a,b\}\cup\{c+d+i\}$.
Since $i\ge 1$ we have $a+b+i>a$ and $a+b+i>b$, so the element $a+b+i$ of the
left-hand side can only be matched by $c+d+i$. Thus, $a+b=c+d$ and, after
cancelling, $\{a,b\}=\{c,d\}$.  The pair $\{\ell_G(u),\ell_G(v)\}$ is
therefore determined by the monomial, and the coefficients
of~\eqref{eq:second-layer} then give, for each such pair and each $i$, the
number of adjacent pairs of that type with $\mu_G(u,v)\ge i$, hence the
multiset of pairs.  Finally each pair $\bigl(\{a,b\},\mu\bigr)$ contributes $\mu$
to $\delta_G(v)$ for each of its endpoints, so the stated sums follow.
\end{proof}

Proposition~\ref{prop:extremal-terms} and Lemma~\ref{lem:second-layer}
recover the loop multiplicities and the sums of non-loop degrees within
each loop-multiplicity class, but do not in general recover the individual
non-loop degrees. Recovering the joint multiset
$\{(\ell_G(v),\delta_G(v)):v\in V(G)\}$ from the full polynomial $\LL_G$
would imply a positive answer to Problem~\ref{prob:loop-degree}.

\subsection{Equivalence on forests}\label{subsec:forests}

\begin{proposition}[Equivalence on forests]\label{prop:forest-L-X}
Let $G$ be a forest.  Then the invariants
\[
  \LL_G,\qquad \widehat{\LL}_G,\qquad U_G,\qquad U^{\ext}_G, \qquad \eta^{\ext}_G,\qquad X_G
\]
determine one another.
\end{proposition}

\begin{proof}
If $G$ is a forest, then every spanning subgraph has cyclomatic number
zero and every spanning forest of $G$ has external activity zero.
Therefore $U_G$ is independent of $y$, and the substitutions
\[
  z_s\mapsto x_{s-1},\qquad
  z_s\mapsto u_{s,0},\qquad
  z_s\mapsto z_{s,0}
\]
identify $U_G$ with $\LL_G(1,\bx)$, $\widehat{\LL}_G$, and
$U_G^{\ext}$, respectively. These variable relabelings are invertible.
Moreover,
\[
  X_G=(-1)^{|V(G)|}\left.U_G\right|_{z_s=-p_s},
\]
which is invertible in the power-sum basis. The full polynomial $\LL_G(t,\bx)$ can be recovered from $\LL_G(1,\bx)$
by Corollary~\ref{th:cor}(4), and $\eta_G^{\ext}$ is equivalent to $U_G^{\ext}$ by \cite[Corollary~4.6]{MN}.
\end{proof}

The universal $V$-function is excluded from this equivalence. For a forest
with $n$ vertices and $m$ edges, every spanning forest has zero external
activity, so~\eqref{eq:V-forest} gives
\[
  \cV_G(\bv)=\sum_{j=0}^m\binom{m}{j}v_0^{n-j}
  =v_0^{n-m}(1+v_0)^m.
\]
It therefore depends only on $n$ and $m$; in particular, it does not
distinguish $P_4$ from $K_{1,3}$.

\begin{remark}[The tree case of Crew's conjecture]\label{rem:crew-trees}
For a forest $G$, Proposition~\ref{prop:forest-L-X} and
Corollary~\ref{cor:gdp} show that $X_G$ determines $\GD_G$.
The same argument applies to the higher-order polynomials in~\cite{LT}.
For trees, Aliste-Prieto, Martin, Wagner, and Zamora~\cite{AMWZ} obtain
an explicit linear recovery from the chromatic symmetric function.
The deduction through the equivalences above establishes determination
but does not itself supply that formula. It uses the forest hypothesis
and does not settle the simple-graph question in
Remark~\ref{rem:gdp-ordinary}.
\end{remark}

\section{Adding loops and weights}
\label{sec:loops-weights}\label{subsec:weights}
Nonnegative integer index-shift weights can be encoded by adding loops.
We compare uniform loop addition with addition at a single vertex.

\begin{remark}[Loops as vertex weights]\label{rem:loops-as-weights}
Let $G^{-}$ be obtained from $G$ by deleting all loops, and compute
activities in $G^{-}$ in the following formula. Every removed loop is
externally active in $G$ and belongs to the component of its vertex.
Its contribution to the activity cancels its contribution to $m(G)$,
so
\begin{equation}\label{eq:loops-as-weights}
  \LL_G(t,\bx)=\sum_{F\in\cF(G^{-})}
  t^{m(G^{-})-|F|-\eps(F)}
  \prod_{T\in c(F)}x_{|E(T)|+\eps_F(T)+\ell_G(T)},
  \qquad \ell_G(T)=\sum_{v\in V(T)}\ell_G(v).
\end{equation}
More generally, giving the vertices nonnegative integer weights $w(v)$
and shifting each forest-component index by the sum of its vertex
weights gives $\LL_{G,w}=\LL_{G+w}$, where $G+w$ adds $w(v)$ loops at
$v$.

Ordinary $U$ does not record the location of an added loop: writing $G+\lambda_v$
for $G$ with one new loop at $v$, we have $U_{G+\lambda_v}=yU_G$ for every $v$, since summing over the presence
and absence of that loop contributes $1+(y-1)=y$. Vertex-weighted
extensions such as the $W$-polynomial~\cite{NW,Nob} and the weighted
chromatic symmetric function~\cite{CS2,ACSZ} use a different encoding.
\end{remark}

\subsection{Adding loops uniformly}\label{subsec:uniform-loops}

For a graph $G$ and an integer $j\ge 0$, denote by $G^{+j}$ the graph obtained
by adding $j$ loops at every vertex of $G$.

\begin{proposition}[Uniform loop addition]\label{prop:uniform-loops}
Let $G$ be a graph on $n$ vertices with $m$ edges, possibly with loops and multiple edges.
For every $j\ge 0$,
\begin{equation}\label{eq:uniform-loops}
  \LL_{G^{+j}}(t,\bx)  = \sum_{F\in\cF(G)} t^{\,m-|F|-\eps(F)}  \prod_{T\in c(F)} x_{(1+j)|V(T)|-1+\eps_F(T)} .
\end{equation}
In particular $\LL_{G^{+j}}$ is a specialization of $\widehat{\LL}_G$.
If moreover $j\ge m$, this specialization is invertible: $\widehat{\LL}_G$ is
determined by $\LL_{G^{+j}}$ together with the value of $j$.
\end{proposition}

\begin{proof}
No loop lies in a spanning forest, so $\cF(G^{+j})=\cF(G)$.  Every loop is
externally active and is assigned to the component containing its vertex, whence
\[
  \eps^{G^{+j}}_F(T)=\eps_F(T)+j\,|V(T)|, \qquad \eps^{G^{+j}}(F)=\eps(F)+jn .
\]
Since also $m(G^{+j})=m+jn$, the two occurrences of $jn$ cancel in the exponent of $t$, while
$|E(T)|+\eps^{G^{+j}}_F(T)=(|V(T)|-1)+\eps_F(T)+j|V(T)|$.
Substituting into the forest expansion~\eqref{eq:forests} for $G^{+j}$
gives~\eqref{eq:uniform-loops}.  The right-hand side depends on $F$ only
through the numbers $|V(T)|$ and $\eps_F(T)$, so it is the
specialization $u_{s,r}\mapsto t^{-(s-1+r)}x_{(1+j)s-1+r}$ of $\widehat{\LL}_G$, up to the factor $t^{m}$
as in~\eqref{eq:refined-to-loopy}.

Suppose now that $j\ge m$, and write $s=|V(T)|$ and $r=\eps_F(T)$, so
that the index appearing in~\eqref{eq:uniform-loops} is $(1+j)s-1+r$.
Then $0\le r\le\eps(F)\le m<1+j$, so $s$ and $r$ are, respectively, the
quotient and the remainder of $(1+j)s+r$ upon division by $1+j$.
Each monomial of $\LL_{G^{+j}}$ therefore determines the multiset
$\bigl\{(|V(T)|,\eps_F(T))\bigr\}_{T\in c(F)}$, that is, the corresponding monomial of $\widehat{\LL}_G$.
Here $j$ is assumed given, which is how the statement is used below.
The remaining data follow from it, since Proposition~\ref{prop:extremal-terms}
applied to $\LL_{G^{+j}}$ gives the order $n$ and the size $m+jn$, and thus $m$ as well.
If $G$ is loopless then $j$ can be read off from $\LL_{G^{+j}}$ as the common loop
multiplicity, by Proposition~\ref{prop:extremal-terms}; if $G$ has loops, this fails.
\end{proof}

\begin{corollary}\label{cor:extended-reformulated}
The loopless part of Problem~\ref{prob:extended} has an affirmative  answer if and only if
\[
  \LL_{G}=\LL_{H} \quad\Longrightarrow \quad  \LL_{G^{+j}}=\LL_{H^{+j}}
\]
  for all loopless graphs $G,H$ and all $j\ge 0$ (or equivalently for $j=m(G)$).
\end{corollary}

\begin{proof}
If $\LL_G=\LL_H$ then $G$ and $H$ have the same order $n$ and the same size $m$,
by Proposition~\ref{prop:extremal-terms}.
Should Problem~\ref{prob:extended} have an affirmative answer, then
$\widehat{\LL}_G=\widehat{\LL}_H$, and $\LL_{G^{+j}}$ and $\LL_{H^{+j}}$ are
the same specialization of these by
Proposition~\ref{prop:uniform-loops}, hence equal.  Conversely, if
$\LL_{G^{+m}}=\LL_{H^{+m}}$, the second assertion of
Proposition~\ref{prop:uniform-loops} gives
$\widehat{\LL}_G=\widehat{\LL}_H$.
\end{proof}

\subsection{Adding a loop at one vertex}\label{subsec:single-loop}

For $v\in V(G)$ and $j\ge 0$ let $G+j\lambda_v$ denote the graph obtained from $G$
by adjoining $j$ loops at $v$ and nothing elsewhere.
Taking the weight function $j\bOne_v$ in~\eqref{eq:loops-as-weights} gives
\begin{equation}\label{eq:single-loop}
  \LL_{G+j\lambda_v}(t,\bx)  =  \sum_{F\in\cF(G)}  t^{\,m-|F|-\eps(F)}
  \prod_{T\in c(F)} x_{|E(T)|+\eps_F(T)+j\,[v\in T]} ,
\end{equation}
where $[v\in T]$ is $1$ if $T$ contains $v$ and $0$ otherwise.
Thus this operation has a localized effect: in each forest term exactly
one component has its index raised, namely the component containing $v$,
and the exponent of $t$ is unchanged.

This suggests recording the marked component separately.  Marking a vertex is
standard practice for invariants of this family: the rooted $U$-polynomial of
Aliste-Prieto, de Mier and Zamora~\cite{ADZ} is the corresponding
construction on the $U$-side.  The following is its loopy counterpart.

\begin{definition}\label{def:pointed-loopy}
Let $G$ be a graph and $v\in V(G)$, and let $\by=(y_i)_{i\ge 0}$ be a
second family of variables.  Writing $T_v$ for the component of $F$
containing $v$, the \emph{pointed loopy polynomial} of $(G,v)$ is
\[
  \LL_{G,v}(t,\bx,\by)  := \sum_{F\in\cF(G)} t^{\,m-|F|-\eps(F)}\,
  y_{|E(T_v)|+\eps_F(T_v)}\prod_{T\ne T_v}x_{|E(T)|+\eps_F(T)} .
\]
\end{definition}
The definition refers to an edge order, through the componentwise activities
$\eps_F(T)$.  It does not depend on it.

\begin{proposition}\label{prop:pointed-well-defined}
For every graph $G$ and every $v\in V(G)$, the polynomial $\LL_{G,v}$ is
independent of the chosen ordering of $E(G)$.
\end{proposition}

\begin{proof}
Introduce the pointed refinement
\[
  \widehat{\LL}_{G,v}(\bu,\bu')  :=
  \sum_{F\in\cF(G)} u'_{|V(T_v)|,\,\eps_F(T_v)} \prod_{T\in c(F),\ T\ne T_v} u_{|V(T)|,\,\eps_F(T)} ,
\]
of which $\LL_{G,v}$ is the specialization
$u_{s,r}=t^{-(s-1+r)}x_{s-1+r}$, $u'_{s,r}=t^{-(s-1+r)}y_{s-1+r}$ multiplied by
$t^{|E(G)|}$, as in~\eqref{eq:refined-to-loopy}.  It therefore suffices
to prove order-independence of $\widehat{\LL}_{G,v}$.  Put
\[
  U^{\ext}_{G,v}(\bz,\bz')  :=  \sum_{A\subseteq E(G)}
  z'_{|V(D_v)|,\,\nu(D_v)}    \prod_{D\in c(A),\ D\ne D_v} z_{|V(D)|,\,\nu(D)} ,
\]
where $D_v$ is the component of $(V(G),A)$ containing $v$.  This is a sum over
subsets of $E(G)$ and involves no ordering.

Apply the argument of Theorem~\ref{thm:refined-extended-U}.  The intervals
$[F,F\cup\EA(F)]$ partition the subsets of $E(G)$, and adding externally active
edges never merges components, so along the interval of $F$ every component keeps its
vertex set. In particular the component containing $v$ has vertex set $V(T_v)$ throughout.
The distinguished factor therefore undergoes the same binomial transform as the others, and
\[
  U^{\ext}_{G,v}(\bz,\bz')  =  \left.\widehat{\LL}_{G,v}(\bu,\bu')\right|_{
    u_{s,r}=\sum_{j=0}^{r}\binom{r}{j}z_{s,j},\quad u'_{s,r}=\sum_{j=0}^{r}\binom{r}{j}z'_{s,j}} .
\]
Binomial inversion in each of the two families of variables recovers
$\widehat{\LL}_{G,v}$ from $U^{\ext}_{G,v}$, which is independent of the edge order.
\end{proof}
We have $\LL_{G,v}\big|_{y_i=x_i}=\LL_G$, while
$\LL_{G,v}\big|_{y_i=x_{i+j}}=\LL_{G+j\lambda_v}$ for every $j\ge 0$
by~\eqref{eq:single-loop}.  Thus the whole family of loop additions at $v$
is encoded by a single invariant of the pointed graph $(G,v)$, and the multiset
\[
  \bigl\{\LL_{G,v}\ :\ v\in V(G)\bigr\}
\]
is an isomorphism invariant of $G$.

The sum of the single-loop refinements is determined by a differential
operator. We obtain the ordinary sum by specialization below.

\begin{proposition}\label{prop:loop-derivation}
  Let
  \[
    \cD  :=  \sum_{s\ge 1}\ \sum_{r\ge 0}  s\,u_{s,r+1}\,\frac{\partial}{\partial u_{s,r}} .
  \]
  Then for every graph $G$, possibly with loops and multiple edges,
  \[
    \sum_{v\in V(G)}\widehat{\LL}_{G+\lambda_v}(\bu)   =  \cD\,\widehat{\LL}_G(\bu).
  \]
\end{proposition}

\begin{proof}
Adding a loop at $v$ leaves the spanning forests unchanged and increases
by one the activity of the component containing $v$. For a fixed forest,
a component of size $s$ and activity $r$ therefore contributes the
replacement $u_{s,r}\mapsto u_{s,r+1}$ for each of its $s$ vertices.
The product rule for differentiation counts these replacements,
including repeated equal factors, and gives precisely
$\cD\widehat{\LL}_G$.
\end{proof}

Since $\cD$ does not depend on the graph and $V(G+\lambda_v)=V(G)$, the proposition iterates:
\[
  \cD^{\,k}\,\widehat{\LL}_G  =   \sum_{(v_1,\dots,v_k)\in V(G)^k}
  \widehat{\LL}_{G+\lambda_{v_1}+\cdots+\lambda_{v_k}} ,
\]
the sum being over ordered $k$-tuples, repetitions allowed.

Single-vertex loop addition retains additional information about the
chosen vertex. Its aggregate, however, is determined by
$\widehat{\LL}_G$. Put $m=m(G)$. By
Proposition~\ref{prop:loop-derivation} and~\eqref{eq:refined-to-loopy},
\[
 \sum_{v\in V(G)}\LL_{G+\lambda_v}(t,\bx)
 =t^{m+1}\left.\bigl(\cD\widehat{\LL}_G\bigr)\right|_
 {u_{s,r}=t^{-(s-1+r)}x_{s-1+r}}.
\]
The factor $t^{m+1}$ accounts for the added edge. Thus the sum is
determined by the refinement, whereas the multiset of its summands need
not be, as the next proposition shows.

\begin{proposition}\label{prop:single-loop-exceeds}
There are non-isomorphic loopless graphs $G_1,G_2$ with
$\widehat{\LL}_{G_1}=\widehat{\LL}_{G_2}$ for which the multisets
\[
  \bigl\{\LL_{G_i+\lambda_v}\ :\ v\in V(G_i)\bigr\}, \qquad i=1,2,
\]
differ.  Hence this multiset is not determined by $\widehat{\LL}_G$,
although by Proposition~\ref{prop:loop-derivation} its sum is.
\end{proposition}

\begin{proof}
Take the graphs $G_1,G_2$ with \texttt{graph6} codes \  \verb+I?bB@_wYO+ \ and \  \verb+I?bB@_wk_+ \ which
form one of the five collision classes on ten vertices with fourteen edges of Section~\ref{sec:extended-U}.
They satisfy $\LL_{G_1}=\LL_{G_2}$ and $\widehat{\LL}_{G_1}=\widehat{\LL}_{G_2}$, but computation
from~\eqref{eq:single-loop} gives ten distinct polynomials $\LL_{G_i+\lambda_v}$ for each $i$,
and the two ten-element multisets are not equal. Their sums agree by Proposition~\ref{prop:loop-derivation} and the specialization above.
\end{proof}

The same failure of determination occurs in each of the five collision
classes on ten vertices with fourteen edges.

\section{The equivalence conjecture}\label{sec:tutte-symmetric}

\subsection{Statement of the conjecture}
\label{subsec:main-conjecture}

The specializations
\begin{equation*}
  u_{s,r}\longmapsto z_sy^r,
  \qquad u_{s,r}\longmapsto x_{s-1+r}
\end{equation*}
give $U_G$ and $\LL_G(1,\bx)$ from $\widehat{\LL}_G$; the latter
determines the graded loopy polynomial. We conjecture that these
specializations distinguish the same simple graphs.

\begin{conjecture}[Equivalence conjecture]
\label{conj:L-XB-equivalence}
For simple graphs $G$ and $H$,
\[
  \LL_G=\LL_H  \quad\Longleftrightarrow\quad U_G=U_H  \quad\Longleftrightarrow\quad
  \eta_G=\eta_H  \quad\Longleftrightarrow\quad  \XB_G=\XB_H.
\]
The hypothesis that $G$ and $H$ be simple cannot be dropped: loops already
give a counterexample in Example~\ref{ex:loop-necessary}, while multiple
edges give one in Theorem~\ref{thm:MN-solved}.
\end{conjecture}

The last three invariants are classically equivalent~\cite{NW,Sar,MN,Nob};
the conjectural assertion concerns their comparison with $\LL$.

\subsection{Reduction to connected graphs}
\label{subsec:connected-reduction}

The correspondence between irreducible factors and connected components
also holds for the ordinary $U$-polynomial and $\XB$.

\begin{proposition}[Connectedness and irreducibility]
\label{prop:U-XB-irreducible}
Let $G$ be a loopless graph.  Then $U_G(\bz,y)$ is irreducible in $\QQ[y,z_1,z_2,\ldots]$ if and only if
$G$ is connected.  Likewise, when the Tutte symmetric function is written in power-sum coordinates,
$\XB_G(q;\bw)$ is irreducible in $\QQ[q,p_1,p_2,\ldots]$ if and only if $G$ is connected.
\end{proposition}

\begin{proof}
For the empty graph, both polynomials are $1$ and hence are not
irreducible. If $G$ has at least two connected components, both factor
nontrivially over them by their spanning-subgraph formulas.  Suppose now that
$G$ is connected and has $n$ vertices.  In the state-sum
\eqref{eq:U}, the variable $z_n$ occurs only when the spanning
subgraph is connected, so
\[
  U_G(\bz,y) = C_G(y)z_n+Q(y,z_1,\ldots,z_{n-1}), \quad \text{where} \quad
  C_G(y)=  \sum_{\substack{A\subseteq E(G)\\(V(G),A)\text{ connected}}} (y-1)^{|A|-n+1},
\]
is nonzero. Moreover, because $G$ is loopless, the coefficient of $z_1^n$ in $U_G$ is $1$.

Suppose $U_G=AB$.  Since its degree in $z_n$ is one, one factor, say $A$,
is independent of $z_n$.  Comparing the coefficient of $z_n$ shows that
$A$ divides the polynomial $C_G(y)$.  Hence, up to a nonzero rational
constant, $A$ lies in $\QQ[y]$.  But then $A$ divides every coefficient of $U_G$
as a polynomial in the $z_i$, in particular the coefficient $1$ of $z_1^n$.
Thus $A$ is a unit.

The proof for $\XB_G$ is identical using~\eqref{eq:XB-subgraph}.
The coefficient of $p_n$ is the nonzero polynomial
\[
  \sum_{\substack{A\subseteq E(G)\\(V(G),A)\text{ connected}}}q^{|A|},
\]
therefore  a factor independent of $p_n$ must depend only on $q$ and,
since the coefficient of $p_1^n$ is $1$, this factor must be a unit.
\end{proof}

\begin{remark}\label{rem:tutte-irreducibility}
This result parallels (but does not follow from) the classical irreducibility theorem for
the Tutte polynomial of a connected matroid, see~\cite{MMN}.
\end{remark}

\begin{corollary}[Reduction to connected graphs]
\label{cor:connected-reduction}
Conjecture~\ref{conj:L-XB-equivalence} for simple graphs is equivalent to its
restriction to connected simple graphs.
\end{corollary}

\begin{proof}
Only the extension from connected to arbitrary simple graphs requires proof.
Suppose first that $\LL_G=\LL_H$.  By
Corollary~\ref{cor:L-component-factorization}, unique factorization pairs the
connected components of $G$ and $H$ so that the paired components have equal
loopy polynomials.  The conjecture for connected graphs then gives equal
$U$-polynomials for each pair, and multiplicativity gives $U_G=U_H$.

Conversely, if $U_G=U_H$, Proposition~\ref{prop:U-XB-irreducible} and unique
factorization pair the connected components so that their $U$-polynomials are equal.
The normalization is fixed by the coefficient $1$ of $z_1^n$, so associated irreducible factors
are in fact equal. Applying the connected case and multiplying gives $\LL_G=\LL_H$.
The equivalence with $\eta$ and $\XB$ is classical~\cite{NW,Sar,MN,Nob}.
\end{proof}

\subsection{Computational evidence and the Merino--Noble problem}\label{sec:extended-U}

The diagram~\eqref{eq:common-diagram-intro} puts Conjecture~\ref{conj:L-XB-equivalence}
next to an older open problem of Merino and Noble~\cite{MN}.
They asked for a pair of loopless graphs $G,H$ such that
\begin{equation}\label{eq:MN-open-problem}
  U_G=U_H  \qquad\text{but}\qquad  U_G^{\ext}\ne U_H^{\ext}.
\end{equation}
No such pair was known, and the problem is restated as open in Noble's
survey~\cite[\S26.7]{Nob}.  Theorem~\ref{thm:MN-solved} below gives a solution on
five vertices. Since the pair has multiple edges, the corresponding problem
for \emph{simple} graphs remains open.

If $U$ determined $U^{\ext}$ on simple graphs, then it would determine
$\LL$ by Theorem~\ref{thm:refined-extended-U}. Conversely, a simple pair
with equal $U$ but distinct $\LL$ would also be a Merino--Noble pair.
The implication $\LL_G=\LL_H\Rightarrow U_G=U_H$ requires a separate
argument.

In his exhaustive computation for simple graphs on at most ten vertices,
Markstr\"om~\cite{Mar} found that every collision for the ordinary
$U$-polynomial is also a collision for the extended $U$-polynomial.
We extended this to eleven vertices as well as to the loopy polynomial.

We call a \emph{collision class} of a graph invariant a maximal set of pairwise
non-isomorphic graphs from a given class sharing the same value of the invariant.
We only present data on \emph{nontrivial} collision classes, i.e.\ those with at
least two elements.
The nontrivial classes of connected simple graphs are as follows.

\begin{center}
  \begin{tabular}{rrrlrc}
    \toprule
    $n$ & connected simple graphs & classes & sizes & graphs & edge range \\
    \midrule
    $\le 7$    & $996$              & $0$      & ---                                   & $0$      & ---        \\
    $8$    & $11\,117$          & $8$      & $8\times 2$                            & $16$     & $13$--$15$ \\
    $9$    & $261\,080$         & $65$     & $65\times 2$                           & $130$    & $13$--$23$ \\
    $10$   & $11\,716\,571$     & $1\,285$ & $1\,281\times 2,\ 4\times 4$              & $2\,578$ & $14$--$32$ \\
    $11$   & $1\,006\,700\,565$ & $22\,499$& $22\,491\times 2,\ 4\times 3,\ 4\times 4$  & $45\,010$& $13$--$42$ \\
    \bottomrule
  \end{tabular}
\end{center}

Here ``sizes'' records how many classes have each cardinality and ``graphs''
the total number of graphs lying in a nontrivial class.
The partitions induced by $\LL_G$, $\widehat{\LL}_G$, and $\XB_G$ coincide
as partitions, not merely in their numbers of classes.
This gives the following statement, which supersedes the corresponding range in
Markstr\"om's computation.

\begin{proposition}[Computer-assisted verification]
  \label{prop:computational-evidence}
  For simple graphs $G,H$ with at most eleven vertices, the following
  conditions are equivalent:
  \begin{align*}
    \LL_G&=\LL_H, & U_G&=U_H, & \eta_G&=\eta_H, & \XB_G&=\XB_H,\\
    U_G^{\ext}&=U_H^{\ext}, & \eta_G^{\ext}&=\eta_H^{\ext}, & \widehat{\LL}_G&=\widehat{\LL}_H. &&
  \end{align*}
\end{proposition}

\begin{proof}
 The search was validated in three independent ways, described in Appendix~\ref{subsec:app-census}.

The computation is carried out on connected simple graphs.
By Corollary~\ref{cor:connected-reduction} and the multiplicativity of
$\LL$, $U$ and $\widehat{\LL}$, together with the fact that the
irreducible factorizations of $\LL$ and $U$ recover the corresponding
component polynomials (Corollary~\ref{cor:L-component-factorization}
and Proposition~\ref{prop:U-XB-irreducible}), this suffices for arbitrary
simple graphs on at most eleven vertices.
Therefore  $\LL$, $U$, and
$\widehat{\LL}$ induce the same partition in that range.
The remaining conditions follow from the classical equivalences $U\leftrightarrow\eta\leftrightarrow \XB$
of~\cite{NW,Sar,MN,Nob} and Theorem~\ref{thm:refined-extended-U}.
The equivalence of $U^{\ext}$ and $\eta^{\ext}$ is established
in~\cite[Corollary~4.6]{MN}.
\end{proof}

\begin{remark}\label{rem:markstrom-extended}
Proposition~\ref{prop:computational-evidence} also extends
Markstr\"om's verification~\cite[\S7.6.6]{Mar} of ordinary-to-extended
$U$ recovery through eleven vertices. In this range, neither
$U^{\ext}\to U$ nor $\widehat{\LL}\to\LL$ loses distinguishing power.
\end{remark}

Some regularities of the collision classes can be explained by the behavior of the invariants
under passing to complements and joins.

\subsection{Complementation and joins}
\label{subsec:complementation}

Brylawski~\cite{Bry} proved that the polychromate of a simple graph determines the polychromate of
its complement. To present this result for $\XB$ it is convenient to shift the edge variable and use
\begin{equation}\label{eq:XB-shift}
 \widetilde{\XB}_G(t;\bw):=\XB_G(t-1;\bw)  = \sum_{\kp:V(G)\to\ZZ_{>0}}t^{b_G(\kp)}\bw_{\kp},
\end{equation}
where $b_G(\kp)$ is the number of monochromatic edges in the coloring $\kp$
and $\ds \bw_{\kp}=\prod_{v\in V(G)}w_{\kp(v)}$.

\begin{proposition}[Complementation]\label{prop:XB-complement}
Let $G$ be a simple graph on $n$ vertices. Each of $\eta_G$, $U_G$ and
$\XB_G$ determines the corresponding invariant of $\overline G$.
For an integer partition $\lambda\vdash n$, put
$\ds q_\lambda=\prod_iq_{\lambda_i}$ and $\ds c(\lambda)=\sum_i\binom{\lambda_i}{2}$. Then
\begin{equation}\label{eq:eta-complement}
 [q_\lambda]\eta_{\overline G}(\bq,y)  = y^{c(\lambda)}[q_\lambda]\eta_G(\bq,y^{-1}).
\end{equation}
For a finitely supported exponent vector $\alpha$ with $\sum_i\alpha_i=n$, the corresponding identity for $\XB$ is
\begin{equation}\label{eq:XB-complement}
 [\bw^\alpha]\widetilde{\XB}_{\overline G}(t;\bw)  = t^{c(\alpha)}[\bw^\alpha]\widetilde{\XB}_G(t^{-1};\bw),
 \qquad c(\alpha)=\sum_i\binom{\alpha_i}{2}.
\end{equation}
\end{proposition}

\begin{proof}
For every $B\subseteq V(G)$ we have $e_{\overline G}(B) = \binom{|B|}{2}-e_G(B)$.
Therefore a partition $\pi$ of type $\lambda$ satisfies $e_{\overline G}(\pi)=c(\lambda)-e_G(\pi)$.
Summing over partitions of type $\lambda$ gives~\eqref{eq:eta-complement}.
For a coloring with multiplicity vector $\alpha$, the same argument applied to its color classes
gives $b_{\overline G}(\kp)=c(\alpha)-b_G(\kp)$, which proves~\eqref{eq:XB-complement}.
The assertion for $U$ follows from its equivalence with the polychromate~\cite{Sar,MN}.
\end{proof}

The same blockwise argument applies to the extended polychromate.
Combined with its equivalence with $\widehat{\LL}$, it gives the following corollary.

\begin{corollary}[Complementation for the refinement]
\label{cor:refined-complement}
For every simple graph $G$,
\begin{equation}\label{eq:extended-polychromate-complement}
 \eta^{\ext}_{\overline G}(\bq)  = \left.\eta^{\ext}_G(\bq)\right|_{q_{s,j}\mapsto q_{s,\binom{s}{2}-j}}.
\end{equation}
The substitution is an involution on the variables $q_{s,j}$ with $s\ge1$ and
$0\le j\le\binom{s}{2}$.
Consequently each of $\eta_G^{\ext}$, $U_G^{\ext}$ and $\widehat{\LL}_G$ determines the corresponding invariant of $\overline G$. In particular, for simple graphs $G,H$,
\[
 \widehat{\LL}_G=\widehat{\LL}_H \quad\Longleftrightarrow\quad
 \widehat{\LL}_{\overline G}=\widehat{\LL}_{\overline H}.
\]
\end{corollary}

\begin{proof}
In the defining sum
$\eta_G^{\ext}=\sum_{\pi\vdash V(G)}\prod_{B\in\pi}q_{|B|,e_G(B)}$,
apply $e_{\overline G}(B)=\binom{|B|}{2}-e_G(B)$ to each block.
This gives~\eqref{eq:extended-polychromate-complement}, and the
substitution is its own inverse. The remaining assertions follow from
Merino and Noble's equivalence of the extended polychromate and extended
$U$-polynomial~\cite[Corollary~4.6]{MN}, together with Theorem~\ref{thm:refined-extended-U}.
\end{proof}

Joining a fixed graph also preserves and reflects equality of $\XB$.
The following statement holds even when loops and multiple edges are allowed.

\begin{proposition}[Joins]\label{prop:join-equivalence}
  Let $J$ be any graph, and write $G\vee J$ for the graph join.  Then for all $G_1,G_2$,
  \begin{equation*}
    \XB_{G_1\vee J}=\XB_{G_2\vee J} \quad\Longleftrightarrow\quad \XB_{G_1}=\XB_{G_2}.
  \end{equation*}
\end{proposition}

\begin{proof}
We use the shifted form of $\XB$~\eqref {eq:XB-shift}.
For a coloring $\kp:V(G)\to \ZZ_{> 0}$ let
\[
  m(\kp)=\bigl(m_c(\kp)\bigr)_{c\ge 1}, \qquad m_c(\kp)=|\kp^{-1}(c)|,
\]
be its vector of color multiplicities, and for such a vector $\alpha$ put
\[
  P^{\alpha}_{G}(t)=\sum_{\kp:\,m(\kp)=\alpha}t^{\,b_G(\kp)}, \qquad\text{so that}\qquad
  \widetilde{\XB}_G=\sum_{\alpha}P^{\alpha}_{G}(t)\,\bw^{\alpha}.
\]

A coloring of $V(G\vee J)$ is a pair $(\kp_G,\kp_J)$, and for a pair $(u,v)\in V(G)\times V(J)$,
the corresponding cross edge of the join is monochromatic when $\kp_G(u)=\kp_J(v)$.
The number of monochromatic cross edges is thus the inner product
$\sum_{c}m_c(\kp_G)m_c(\kp_J)=\langle m(\kp_G),m(\kp_J)\rangle$, and therefore
\begin{equation}\label{eq:XB-join-identity}
  \widetilde{\XB}_{G\vee J} = \sum_{\alpha,\beta}t^{\,\langle\alpha,\beta\rangle}\,
  P^{\alpha}_{G}(t)\,P^{\beta}_{J}(t)\,\bw^{\alpha+\beta} .
\end{equation}
The right-hand side depends on $G$ only through the family
$\bigl(P^{\alpha}_G\bigr)_{\alpha}$, that is, only through
$\widetilde{\XB}_G$; so $\XB_{G_1}=\XB_{G_2}$ gives
$\XB_{G_1\vee J}=\XB_{G_2\vee J}$.  Note that no simplicity is used: loops
contribute to $b_G$ and $b_J$, and parallel edges are counted with multiplicity.

For the converse, let $T$ be the $\ZZ[t^{\pm1}]$-linear operator sending the monomial $\bw^{\alpha}$ to
$T(\bw^{\alpha})=t^{-\sum_i\binom{\alpha_i}{2}}\bw^{\alpha}$. Since
$\binom{\alpha_i+\beta_i}{2}=\binom{\alpha_i}{2}+\binom{\beta_i}{2}+\alpha_i\beta_i$,
we have $t^{\langle\alpha,\beta\rangle}\,T(\bw^{\alpha+\beta})=T(\bw^{\alpha})\,T(\bw^{\beta})$,
and~\eqref{eq:XB-join-identity} becomes
\[
  T\bigl(\widetilde{\XB}_{G\vee J}\bigr)=T\bigl(\widetilde{\XB}_{G}\bigr)\,T\bigl(\widetilde{\XB}_{J}\bigr).
\]
All three series are symmetric of degree at most $N=|V(G)|+|V(J)|$, so their coefficients are determined after setting $w_c=0$ for $c>N$, and the identity then holds in the domain
$\ZZ[t^{\pm1}][w_1,\dots,w_N]$. As $T$ is invertible and $\widetilde{\XB}_J\ne 0$,
the factor $T(\widetilde{\XB}_J)$ can be cancelled, so $\widetilde{\XB}_{G\vee J}$
determines $\widetilde{\XB}_G$.
\end{proof}

\begin{remark}\label{rem:join-via-complement}
For simple graphs, Proposition~\ref{prop:join-equivalence} also follows
from $\overline{G\vee J}=\overline G\sqcup\overline J$,
Proposition~\ref{prop:XB-complement}, and cancellation of the nonzero
factor $\XB_{\overline J}$ in $\QQ[q,p_1,p_2,\ldots]$.
The graph operation itself is cancellative on isomorphism classes:
taking complements reduces $G_1\vee J\cong G_2\vee J$ to cancellation
of the connected-component multiset of $\overline J$. Thus joining a
fixed simple graph gives an injection into the resulting collision class.
\end{remark}

A partial join need not be a nontrivial join and may have connected complement.

\begin{definition}\label{def:partial-join}
Given graphs $G, H$ and a subset $S\subseteq V(H)$, the \emph{partial join} of $H$ and $G$
over $S$ is the graph $H\vee_S G$ on $V(H)\sqcup V(G)$ with the edge set
\[
  E(H\vee_S G):=E(H)\sqcup E(G) \sqcup \{uw : u\in S \text{ and } w\in V(G)\}.
\]
\end{definition}
In particular $H\vee_{V(H)}G=H\vee G$ and $H\vee_{\varnothing}G=H\sqcup G$.

\begin{proposition}\label{prop:partial-join}
  Let $H$ be a graph, $S\subseteq V(H)$, and let $G_1,G_2$ be arbitrary graphs,
  loops and multiple edges allowed.  Then
  \[
    \eta^{\ext}_{G_1}=\eta^{\ext}_{G_2}  \quad\Longleftrightarrow\quad
    \eta^{\ext}_{H\vee_S G_1}=\eta^{\ext}_{H\vee_S G_2},
  \]
  and the same holds with $\eta$ in place of $\eta^{\ext}$.
  Consequently, equality of $\widehat{\LL}$, and separately equality of
  $U$, is preserved and reflected by $G\mapsto H\vee_S G$. Moreover,
  \[
    \widehat{\LL}_{G_1}=\widehat{\LL}_{G_2}  \quad\Longrightarrow\quad  \LL_{H\vee_S G_1}=\LL_{H\vee_S G_2}.
  \]
\end{proposition}
\begin{proof}
Put $W=H\vee_S G$ and $h=|V(H)|$. A partition $\rho$ of $V(W)$
is uniquely specified by partitions $\sigma$ of $V(H)$ and $\pi$ of
$V(G)$, together with a partial matching $M$ between their blocks.
Matched blocks $A\in\sigma$ and $B\in\pi$ form the block $A\sqcup B$
of $\rho$, while unmatched blocks remain unchanged. For a matched pair,
\[
  |A\sqcup B|=|A|+|B|,\qquad  e_W(A\sqcup B)=e_H(A)+e_G(B)+|A\cap S|\,|B|.
\]
Thus, after summing over $\sigma$ and all partial matchings, the
contribution of $\pi$ to $\eta_W^{\ext}$ depends only on the multiset $\{(|B|,e_G(B)):B\in\pi\}$.
This defines a $\ZZ$-linear map $\Psi_{H,S}^{\ext}$ on $\ZZ[q_{s,j}:s\ge1,\ j\ge0]$ such that
\[
  \eta_{H\vee_S G}^{\ext}  =\Psi_{H,S}^{\ext}(\eta_G^{\ext}).
\]

For the ordinary polychromate, the same construction gives
\[
  e_W(\rho)=e_G(\pi)+e_H(\sigma)  +\sum_{(A,B)\in M}|A\cap S|\,|B|.
\]
The contribution therefore depends on $\pi$ only through its block
sizes and $e_G(\pi)$: the factor $y^{e_G(\pi)}$ is unchanged by the
matching. Consequently the construction defines a $\ZZ[y]$-linear map
$\Psi_{H,S}$ on $\ZZ[y][q_1,q_2,\ldots]$ satisfying
\[
  \eta_{H\vee_S G}=\Psi_{H,S}(\eta_G).
\]
For both maps, identical factors of an input monomial are treated as distinct occurrences when
forming the matchings. Summing over all matchings supplies their
multiplicities and is independent of the ordering of those occurrences.

Both maps are injective. Give each $q$-variable degree one and, in the
ordinary case, give $y$ degree zero. An input monomial of degree $b$,
a partition $\sigma$ with $a$ blocks, and a matching of $k$ pairs produce a monomial of degree
\[
  b+a-k\le b+h.
\]
Equality holds only when $\sigma$ is the discrete partition of $V(H)$
and the matching is empty. This unique choice contributes the input monomial multiplied by
\[
  Q_H =
  \begin{cases}  \ds\prod_{v\in V(H)}q_{1,\ell_H(v)},  &\text{for }\Psi_{H,S}^{\ext},\\[2mm]
    y^{\ell(H)}q_1^h,   &\text{for }\Psi_{H,S}.  \end{cases}
\]
Hence, if $f\ne 0$ has highest homogeneous part $f_d$, the degree-$(d+h)$ part of its
image under the corresponding map is $Q_Hf_d\ne 0$, since both polynomial rings are integral domains.
Both maps are therefore injective, proving the two equivalences.
The argument also covers empty graphs, with empty products equal to one.
The assertions for $\widehat{\LL}$ and $U$ follow from their equivalence with $\eta^{\ext}$ and $\eta$, respectively; the stated consequence for $\LL$ follows by specialization of $\widehat{\LL}$.
\end{proof}

For $S=V(H)$, the $\eta$-part of Proposition~\ref{prop:partial-join} gives a second proof of
Proposition~\ref{prop:join-equivalence}, since $\eta$ and $\XB$ are equivalent.

\begin{remark}
  \label{rem:complementation}
Propositions~\ref{prop:XB-complement} and~\ref{prop:join-equivalence}
account for three further features of the computation.

\emph{(a) Closure.} Proposition~\ref{prop:XB-complement} and Corollary~\ref{cor:refined-complement} show
that complementation induces involutions on the collision classes of $\XB$ and of $\widehat{\LL}$ within simple graphs. For $\LL$ this is open (see Problem~\ref{prob:complement}) and would follow from
Conjecture~\ref{conj:L-XB-equivalence}. It holds within the range of
Proposition~\ref{prop:computational-evidence}.

\emph{(b) Joins.}
By Proposition~\ref{prop:join-equivalence}, joining a fixed graph carries
collision classes to collision classes injectively.
In particular, if graphs $G_1$ and $G_2$ with $n$ vertices and $m$ edges are in the
same class then so are their cones $K_1\vee G_1$ and $K_1\vee G_2$ with $n+1$ vertices
and $m+n$ edges. More generally, a collision class contains simple graphs obtained
from a nontrivial join precisely when they have disconnected complement.
This gives a check on the search that does not pass through the fingerprints:
all eight classes at $n=8$ reappear in the $n=9$ list at $21\le m\le 23$, and all
$1\,285$ connected classes at $n=10$ reappear in the $n=11$ list at $m+10$.

At $n=11$, joins account for $1\,447$ of the $22\,499$ connected collision classes.
Of these, $1\,366$ are cones, that is, have a $K_1$ join factor.
They are obtained by coning the nontrivial collision classes on arbitrary
ten-vertex simple graphs: the $1\,285$ connected classes and $81=65+8+8$ disconnected classes
of the types $G_9\sqcup K_1$, $G_8\sqcup K_2$, and $G_8\sqcup K_1\sqcup K_1$, where $G_9$ is one
of the $65$ connected classes at $n=9$ and $G_8$ one of the eight-vertex classes.
The remaining join-generated classes at $n=11$, those without $K_1$ factor, are of the following types:
$G_9'\vee (K_1\sqcup K_1)$, $(G_8\sqcup K_1)\vee (K_1\sqcup K_1)$, or $G_8\vee H_3$, where
$G_9'$ is one of the $65-8$ non-cone connected nine-vertex classes, and $H_3\in \{\overline{P}_3,
\overline{K_3}\}$, one of the two non-cone three-vertex graphs. Thus there are $57+8+8+8=81$ non-cone
classes and $1\,447=1\,366+81$ join classes at $n=11$ in total.
The remaining $21\,052=22\,499-1\,447$ classes have
connected complement and are not produced by a nontrivial join.

\emph{(c) Edge ranges.}  Complementation sends a class at $m$ to a class at
$\binom n2 - m$, so the observed ranges would be symmetric were it not that the
search includes only connected simple graphs.
For $n=8,9,11$ they are symmetric: $13$--$15$, $13$--$23$, $13$--$42$.
At $n=10$, the range is $14$--$32$ and
no connected simple graph with $13$ edges lies in a collision class, so complementing
the classes with $m=32$ must produce collisions at $m=45-32=13$
between disconnected simple graphs, which are outside the connected census.
That is precisely what happens, and the count is fully accounted for:
there are three classes with $(n,m)=(10,32)$, and their complements are
$  F\sqcup 2K_1 \ \ \text{for each of the two classes $F$ at }(8,13)$ and
$  F\sqcup K_1 \ \ \text{for the single class $F$ at }(9,13)$.
By multiplicativity and unique factorization (Corollary~\ref{cor:L-component-factorization}),
a disconnected simple graph on ten vertices with thirteen edges lies
in a nontrivial collision class if and only if one of its nontrivial components does.
There is no other way to assemble one since there are no collisions at $(8,12)$ or at
$(7,13)$, and none at all below eight vertices. Thus the three classes are accounted for.
\end{remark}

\subsection{Loops and multiple edges}
\label{sec:why-loops}

Conjecture~\ref{conj:L-XB-equivalence} is stated for simple graphs.
We show here  that neither part of the simplicity hypothesis can be dropped.
Loops break the equivalence on three vertices, and multiple edges break it on five.

\subsubsection{Loops violate the equivalence}
\label{subsec:loops-obstruction}

Merino and Noble's example showing that the extended $U$-polynomial
can distinguish graphs with equal ordinary $U$ also separates $\LL$
from $U$ when loops are allowed~\cite{MN}.

Let $P_3$ have vertices $1,2,3$ and edges $12,23$.  Let $G_1$ be obtained
by adding a loop at the leaf $1$,
and let $G_2$ be obtained by adding a loop at the middle vertex $2$, see Figure~\ref{fig:looped-P3}.

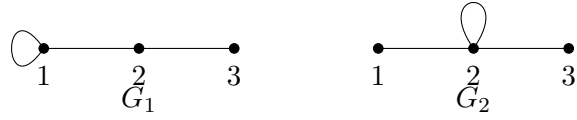
\begin{figure}[ht]
  \centering
  \begin{tikzpicture}[scale=.9]
    \coordinate (A) at (0,0);
    \coordinate (B) at (1.4,0);
    \coordinate (C) at (2.8,0);
    \draw (A)--(B)--(C);
    \draw (A) to[out=125,in=235,distance=1.1cm] (A);
    \filldraw (A) circle (2pt) node[below=3pt] {$1$};
    \filldraw (B) circle (2pt) node[below=3pt] {$2$};
    \filldraw (C) circle (2pt) node[below=3pt] {$3$};
    \node at (1.4,-.75) {$G_1$};
  \end{tikzpicture}
  \hspace{1.2cm}
  \begin{tikzpicture}[scale=.9]
    \coordinate (A) at (0,0);
    \coordinate (B) at (1.4,0);
    \coordinate (C) at (2.8,0);
    \draw (A)--(B)--(C);
    \draw (B) to[out=55,in=125,distance=1.1cm] (B);
    \filldraw (A) circle (2pt) node[below=3pt] {$1$};
    \filldraw (B) circle (2pt) node[below=3pt] {$2$};
    \filldraw (C) circle (2pt) node[below=3pt] {$3$};
    \node at (1.4,-.75) {$G_2$};
  \end{tikzpicture}
  \caption{Looped graphs with equal $U$-polynomials and distinct loopy polynomials.}
  \label{fig:looped-P3}
\end{figure}

Since a loop contributes a factor $y$ to the $U$-polynomial, we have
\[
  U_{G_1}(\bz,y)  =U_{G_2}(\bz,y)   = y\bigl(z_3+2z_2z_1+z_1^3\bigr).
\]
Equivalently, $\eta_{G_1}=\eta_{G_2}$ and $\XB_{G_1}=\XB_{G_2}$.  On the
Tutte-symmetric side this is also immediate from
\[
  \XB_{G+\lambda_v}(q;\bw)  =(1+q)\XB_G(q;\bw),
\]
which is independent of the vertex carrying the loop.

Merino and Noble observed that $U_{G_1}^{\ext}\ne U_{G_2}^{\ext}$~\cite{MN}: the extended invariant
remembers to which component the loop belongs in each spanning subgraph.
The ordinary loopy polynomial also retains enough of this information.
Indeed, Theorem~\ref{thm:forests-sum} gives
\[
  \LL_{G_1} = x_3+t x_0x_2+t x_1^2+t^2x_0^2x_1, \qquad \LL_{G_2} = x_3+2t x_0x_2+t^2x_0^2x_1.
\]
Hence
\[
  U_{G_1}=U_{G_2}  \qquad\text{but}\qquad  \LL_{G_1}\ne\LL_{G_2}.
\]

\begin{example}[Loops destroy the equivalence]
  \label{ex:loop-necessary}
  The pair $G_1,G_2$ above has the same $U$-polynomial, polychromate,
  and Tutte symmetric function, but different extended $U$-polynomials,
  refined loopy polynomials, and ordinary loopy polynomials.
  Thus the specialization $U^{\ext}\to U$ is already non-injective if
  loops are allowed, and Conjecture~\ref{conj:L-XB-equivalence} fails there as well.
\end{example}

\subsubsection{The Merino--Noble problem for loopless multigraphs}
\label{subsubsec:multiple-edges}

Multiple edges alone also prevent the equivalence in
Conjecture~\ref{conj:L-XB-equivalence}.

\begin{figure}[ht]
\centering
\begin{tikzpicture}[scale=1.35, every node/.style={inner sep=1.2pt}]
  \def\R{1.45}
  \foreach \nm/\ang in {a/90, b/162, c/234, d/306, e/18}
  \coordinate (\nm) at (\ang:\R);
  \draw (a) to[bend left=0] (b);
  \draw (b) to[bend left=11] (c);  \draw (b) to[bend right=11] (c);
  \draw (c) to[bend left=11] (d);  \draw (c) to[bend right=11] (d);
  \draw (d) to[bend left=16] (e);  \draw (d) -- (e); \draw (d) to[bend right=16] (e);
  \draw (e) to[bend left=0] (a);
  \draw (a) to[bend right=18] (c);
  \foreach \nm/\lab/\pos in {a/1/above, b/3/left, c/5/below, d/2/below, e/4/right}
  \filldraw[black] (\nm) circle (1.6pt) node[\pos] {\small $\lab$};
  \node at (0,-2.15) {$G_1$};
\end{tikzpicture}
\hspace{2.2cm}
\begin{tikzpicture}[scale=1.35, every node/.style={inner sep=1.2pt}]
  \def\R{1.45}
  \foreach \nm/\ang in {a/90, b/162, c/234, d/306, e/18}
  \coordinate (\nm) at (\ang:\R);
  \draw (a) to[bend left=0] (b);
  \draw (b) to[bend left=11] (c);  \draw (b) to[bend right=11] (c);
  \draw (c) to[bend left=0] (d);
  \draw (d) to[bend left=11] (e);  \draw (d) to[bend right=11] (e);
  \draw (e) to[bend left=16] (a);  \draw (e) -- (a); \draw (e) to[bend right=16] (a);
  \draw (a) to[bend right=18] (c);
  \foreach \nm/\lab/\pos in {a/1/above, b/3/left, c/5/below, d/2/below, e/4/right}
  \filldraw[black] (\nm) circle (1.6pt) node[\pos] {\small $\lab$};
  \node at (0,-2.15) {$G_2$};
\end{tikzpicture}
\caption{Two non-isomorphic loopless graphs on five vertices and ten edges
with $\XB_{G_1}=\XB_{G_2}$ and $U_{G_1}=U_{G_2}$, but $\LL_{G_1}\ne\LL_{G_2}$.
They share the underlying simple graph, a $5$-cycle with one chord,
and differ only in the multiplicities along the path $5,2,4,1$.}
\label{fig:parallel-counterexample}
\end{figure}
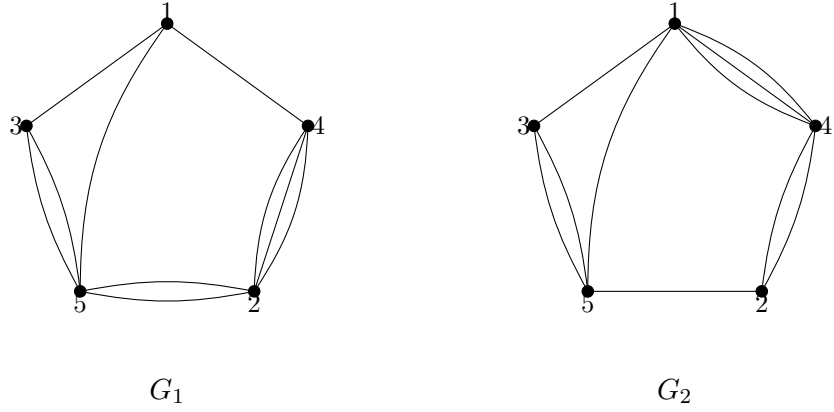

\begin{theorem}
\label{thm:MN-solved}
There exist non-isomorphic loopless multigraphs $G_1,G_2$ with
\[
  U_{G_1}=U_{G_2}    \qquad\text{but}\qquad   U^{\ext}_{G_1}\ne U^{\ext}_{G_2},
\]
giving a solution to the Merino--Noble problem.  For the same pair
$\LL_{G_1}\ne\LL_{G_2}$ while $\XB_{G_1}=\XB_{G_2}$, so the analogous equivalence fails for loopless multigraphs.
\end{theorem}

\begin{proof}
Take $G_1$ and $G_2$ to be the two graphs of Figure~\ref{fig:parallel-counterexample}.
The assertions are verified in Example~\ref{ex:parallel-necessary}.
\end{proof}

\begin{example}[An explicit pair]
\label{ex:parallel-necessary}
Let $G_1$ and $G_2$ be the two graphs of Figure~\ref{fig:parallel-counterexample}.
Both have the $5$-cycle $1,3,5,2,4$ with the chord $\{1,5\}$ as the underlying simple graph.
The edge multiplicities are
\[
  \begin{array}{lcccccc}
    & \{1,3\} & \{3,5\} & \{5,2\} & \{2,4\} & \{4,1\} & \{1,5\}\\[1mm]
    G_1 &   1     &    2    &    2    &    3    &    1    &    1   \\
    G_2 &   1     &    2    &    1    &    2    &    3    &    1
  \end{array}
\]
Then $G_1\not\cong G_2$, and
\[
  \XB_{G_1}=\XB_{G_2},
  \qquad
  U_{G_1}=U_{G_2},  \qquad\text{but}\qquad   \LL_{G_1}\ne\LL_{G_2}.
\]
Indeed,
\[
  \LL_{G_1}-\LL_{G_2}  =t^3\bigl(x_3x_4-x_2x_5\bigr) + t^4\bigl(x_3^2-x_2x_4\bigr)
  +t^5\bigl(x_1x_4-x_2x_3\bigr) + t^6\bigl(x_0x_1x_3-x_0x_2^2\bigr),
\]
so these graphs are separated already by the distribution of the statistic
$|E(T)|+\eps_F(T)$ between the components of a two-component spanning forest.
They also have the same Tutte polynomial and the same degree sequence $(3,3,4,5,5)$.

Since $\widehat{\LL}\leftrightarrow U^{\ext}$ by Theorem~\ref{thm:refined-extended-U},
this pair simultaneously solves the Merino--Noble problem of Section~\ref{sec:extended-U}: \
$U_{G_1}=U_{G_2}$ while $U^{\ext}_{G_1}\ne U^{\ext}_{G_2}$.
A single coefficient already witnesses the latter:
\[
  [z_{2,1}z_{3,0}]\,U^{\ext}_{G_1}=19   \ne   17=[z_{2,1}z_{3,0}]\,U^{\ext}_{G_2},
\]
that is, $G_1$ and $G_2$ have different numbers of spanning subgraphs with two components,
one of order $2$ and cyclomatic number $1$, the other of order $3$ and cyclomatic number $0$.
These answers were checked over $\ZZ$ by two independent implementations.
Finally, $G_1\not\cong G_2$: the triple edge joins vertices of degrees $5$ and $4$ in $G_1$,
but two vertices of degree $5$ in $G_2$.
\end{example}

We found this pair by an exhaustive search over loopless multigraphs with small vertex count
and bounded edge multiplicity, using the same fingerprints as in Section~\ref{sec:the-computation}.
Here $e_G(S)$ simply counts edges with multiplicity, and nothing else in the algorithm changes.
Writing $\mu$ for the maximum multiplicity allowed, the counts of connected loopless
multigraphs and of their nontrivial collision classes are as follows.

\begin{center}
  \begin{tabular}{ccrrrr}
    \toprule
    $n$ & $\mu$ & multigraphs & $\LL$-classes & $\XB$-classes & $U$-classes\\
    \midrule
    $4$    & $9$    & $45\,210$    & $0$   & $0$   & $0$\\
    $5$    & $2$    & $712$        & $0$   & $0$   & $0$\\
    $5$    & $3$    & $10\,364$    & $0$   & $12$  & $12$\\
    $5$    & $4$    & $88\,985$    & $0$   & $58$  & $58$\\
    $5$    & $5$    & $530\,657$   & $0$   & $166$ & $166$\\
    $5$    & $6$    & $2\,431\,555$& $0$   & $360$ & $360$\\
    $6$    & $2$    & $24\,576$    & $15$  & $30$  & $30$\\
    $6$    & $3$    & $1\,590\,368$& $500$ & $1\,130$ & $1\,130$\\
    $6$    & $4$    & $43\,477\,490$& $4\,848$ & $10\,164$ & $10\,164$\\
    $7$    & $2$    & $2\,275\,616$& $1\,213$ & $1\,469$ & $1\,469$\\
    \bottomrule
  \end{tabular}
\end{center}

The ranges are nested---the rows with the same $n$ and smaller $\mu$ are contained in
the largest one for that $n$, so the row totals must not be
added---and the graphs covered are those of
the four maximal rows $(4,9)$, $(5,6)$, $(6,4)$ and $(7,2)$.
These rows contain $48\,229\,871$ connected loopless multigraphs in all.
Every $\LL$-class found is a pair.
Two of the $\XB$-classes at $(6,4)$ have three members and all others are pairs.

The equivalences $U\leftrightarrow\eta\leftrightarrow \XB$ hold for all graphs, loops and
multiple edges included~\cite{NW,Sar,MN,Nob}, so the agreement of the independently
computed $U$- and $\XB$-columns is a check on the implementation rather than a new result.
What is observed is that throughout the tested range
\[
  \LL_G=\LL_H   \quad\Longrightarrow \quad   \widehat{\LL}_G=\widehat{\LL}_H .
\]
Since $\widehat{\LL}_G$ determines $\XB_G$ by Corollary~\ref{cor:refined-specialized},
this forces $\LL_G=\LL_H\Rightarrow \XB_G=\XB_H$ throughout the same range.
The implication is strict because Theorem~\ref{thm:MN-solved} gives pairs with
equal $\XB$ but distinct $\LL$.  Thus the implication
$U_G=U_H\Rightarrow\LL_G=\LL_H$ fails globally on loopless graphs,
whereas the reverse implication holds for every multigraph in the search and
remains open in general (see Problem~\ref{prob:extended}).
In other words, passing from $\widehat{\LL}$ to $\LL$ loses no distinguishing power on any
of the $48\,229\,871$ connected loopless multigraphs tested, even
though the specialization $U^{\ext}\to U$ demonstrably does.

The pair of Theorem~\ref{thm:MN-solved} has five vertices and ten edges.
It has the smallest vertex count found in our bounded-multiplicity search.
In this paragraph a \emph{counterexample} means a pair of loopless multigraphs with equal
$U$-polynomials and distinct loopy polynomials. By Theorem~\ref{thm:refined-extended-U} such a
pair also has distinct extended $U$-polynomials.
There is no counterexample on four vertices with all multiplicities at most nine,
and none on five vertices with all multiplicities at most two, while examples
appear as soon as $\mu=3$ is allowed.
It is not the smallest in terms of the number of edges.
A connected loopless multigraph with at most eight edges has at most nine vertices,
so that range is finite. An exhaustive check over all connected loopless multigraphs
with at most nine edges shows that none with at most eight edges is a counterexample,
whereas nine-edge counterexamples do exist, on six, seven, eight, and nine vertices.
There are no counterexamples on ten vertices, since in this case the graph is a tree,
and trees are ruled out by~Proposition~\ref{prop:forest-L-X}.
This check is the script \texttt{min\_edges.py} of~\cite{code}.
The nine-edge examples have maximum multiplicity two,
so one of them already occurs in the $(6,2)$ row of the table above.
We state Theorem~\ref{thm:MN-solved} with the five-vertex pair because
the search is organized by the number of vertices.

\subsection{Chordal graphs}
\label{subsec:chordal-counterexamples}

A simple graph is \emph{chordal} if it has no induced cycle of length greater than three.
Loebl~\cite{Loe} (see also~\cite[Section~2.3]{LS}) conjectured that the $U$-polynomial distinguishes
non-isomorphic chordal graphs.  Since $U$ and $\XB$ are equivalent, this is equally a conjecture
about the Tutte symmetric function. The collision census of Section~\ref{sec:extended-U} disproves it.

\begin{theorem}[Chordal counterexamples]
\label{thm:chordal-counterexamples}
For each $i\in\{1,2,3,4\}$, the graphs $G_i,H_i$ specified in
Table~\ref{tab:chordal-counterexamples} are connected simple chordal graphs on eleven vertices satisfying
$  \LL_{G_i}=\LL_{H_i},\quad U_{G_i}=U_{H_i},\quad \XB_{G_i}=\XB_{H_i}$.
These are the only pairs of non-isomorphic simple chordal graphs on eleven vertices with
equal $U$-polynomials. They are also the only ones with equal $\LL$ or with equal $\widehat{\LL}$.
For each of the three invariants there are no such pairs on at most ten vertices.
Both graphs in the second pair are $2$-trees.
\end{theorem}

\begin{table}[htbp]
\centering
\renewcommand{\arraystretch}{1.15}
\begin{tabular}{cllc}
  \toprule
  $i$ & $G_i$ (graph6) & $H_i$ (graph6) & $|E(G_i)|=|E(H_i)|$ \\
  \midrule
  $1$ & \verb|J?AA@BC[lr_| & \verb|J?AA@BCkZt_| & $18$ \\
  $2$ & \verb|J?`@?bG\Lf_| & \verb|J?`@?bGkjj_| & $19$ \\
  $3$ & \verb|J?`@dDfy\f_| & \verb|J?`@dDri~J_| & $25$ \\
  $4$ & \verb|JCOedLj{~n_| & \verb|JCOedLj{~v_| & $31$ \\
  \bottomrule
\end{tabular}
\caption{The four pairs of simple chordal graphs of minimum order with equal $U$-polynomials.}
\label{tab:chordal-counterexamples}
\end{table}

\begin{proof}
Filter the $22\,499$ nontrivial collision classes on eleven vertices in Section~\ref{sec:extended-U}
for chordal graphs.  Only the four pairs in the table remain.
Filtering the classes on at most ten vertices yields no pair of chordal graphs.
The classes computed in Section~\ref{sec:extended-U} are those of $\LL$, and by
Proposition~\ref{prop:computational-evidence} they are also the $\widehat{\LL}$- and the $\XB$-classes
throughout this range. Therefore a single filtering settles all three invariants at once.
Chordality is checked by simplicial vertex elimination, and non-isomorphism is checked separately.
Appendix~\ref{subsec:app-chordal} records an independent verification that does not use the census.

The minimality statements and the exclusion of disconnected examples follow
from the absence of smaller connected examples for each
of the three invariants. Indeed, the connected components of a chordal graph are chordal.
For $\LL$ and $U$, multiplicativity and the connected-factor recovery of Corollary~\ref{cor:L-component-factorization}
and Proposition~\ref{prop:U-XB-irreducible}, respectively, reduce any disconnected collision to a
collision between non-isomorphic connected components
as in Corollary~\ref{cor:connected-reduction}.  For $\widehat{\LL}$ nothing further is
needed, since $\widehat{\LL}_G=\widehat{\LL}_H$ implies $U_G=U_H$ by~\eqref{eq:U-refined}.

Finally, a $2$-tree is obtained from $K_2$ by repeatedly adding a vertex whose
neighbors are precisely the two endpoints of an existing edge.
In both $G_2$ and $H_2$, with vertices numbered $0,1,\ldots,10$ in graph6 order,
deleting the vertices in the order $ 4,\ 0,\ 5,\ 1,\ 6,\ 2,\ 7,\ 3,\ 8  $
successively removes a degree-two simplicial vertex and leaves the edge $\{9,10\}$.
Reversing these deletions verifies the $2$-tree assertion.
\end{proof}

The second pair bears on Stanley's tree conjecture~\cite{Sta}.
Trees are the $1$-trees, but the extension to $2$-trees fails:
the two seven-vertex $2$-trees \verb|FGezW| and \verb|Fy[go| have equal chromatic symmetric functions.
Replacing $X$ by the stronger invariant $\XB$ does not restore the assertion. Although $\XB$
distinguishes all $2$-trees on at most ten vertices,
Theorem~\ref{thm:chordal-counterexamples} gives counterexamples on eleven vertices.

Figure~\ref{fig:chordal-2trees} draws that pair, each vertex numbered in the
order of a construction from the edge $v_1v_2$.
\begin{figure}[htbp]
\centering
\begin{tikzpicture}[scale=0.60]
  \begin{scope}[xshift=0.00cm]
    \fill[black!6] (-3.15,-1.70) -- (-4.35,-0.55) -- (-2.10,0.00) -- cycle;
    \fill[black!6] (-3.15,-1.70) -- (-2.10,0.00) -- (0.00,-2.45) -- cycle;
    \fill[black!6] (-3.15,1.70) -- (-1.45,3.35) -- (0.00,2.45) -- cycle;
    \fill[black!6] (-3.15,1.70) -- (-2.10,0.00) -- (0.00,2.45) -- cycle;
    \fill[black!6] (5.05,0.00) -- (3.70,-2.35) -- (0.00,-2.45) -- cycle;
    \fill[black!6] (5.05,0.00) -- (0.00,2.45) -- (0.00,-2.45) -- cycle;
    \fill[black!6] (2.10,0.00) -- (1.35,-1.28) -- (0.00,-2.45) -- cycle;
    \fill[black!6] (2.10,0.00) -- (0.00,2.45) -- (0.00,-2.45) -- cycle;
    \fill[black!6] (-2.10,0.00) -- (0.00,2.45) -- (0.00,-2.45) -- cycle;
    \draw[thin,black!60] (-3.15,-1.70) -- (-4.35,-0.55);
    \draw[thin,black!60] (-3.15,-1.70) -- (-2.10,0.00);
    \draw[thin,black!60] (-3.15,-1.70) -- (0.00,-2.45);
    \draw[thin,black!60] (-3.15,1.70) -- (-1.45,3.35);
    \draw[thin,black!60] (-3.15,1.70) -- (-2.10,0.00);
    \draw[thin,black!60] (-3.15,1.70) -- (0.00,2.45);
    \draw[thin,black!60] (5.05,0.00) -- (3.70,-2.35);
    \draw[line width=0.9pt] (5.05,0.00) -- (0.00,2.45);
    \draw[line width=0.9pt] (5.05,0.00) -- (0.00,-2.45);
    \draw[thin,black!60] (2.10,0.00) -- (1.35,-1.28);
    \draw[line width=0.9pt] (2.10,0.00) -- (0.00,2.45);
    \draw[line width=0.9pt] (2.10,0.00) -- (0.00,-2.45);
    \draw[thin,black!60] (-4.35,-0.55) -- (-2.10,0.00);
    \draw[thin,black!60] (-1.45,3.35) -- (0.00,2.45);
    \draw[thin,black!60] (3.70,-2.35) -- (0.00,-2.45);
    \draw[thin,black!60] (1.35,-1.28) -- (0.00,-2.45);
    \draw[line width=0.9pt] (-2.10,0.00) -- (0.00,2.45);
    \draw[line width=0.9pt] (-2.10,0.00) -- (0.00,-2.45);
    \draw[line width=0.9pt] (0.00,2.45) -- (0.00,-2.45);
    \fill (-3.15,-1.70) circle (2.4pt);
    \fill (-3.15,1.70) circle (2.4pt);
    \fill (5.05,0.00) circle (2.4pt);
    \fill (2.10,0.00) circle (2.4pt);
    \fill (-4.35,-0.55) circle (2.4pt);
    \fill (-1.45,3.35) circle (2.4pt);
    \fill (3.70,-2.35) circle (2.4pt);
    \fill (1.35,-1.28) circle (2.4pt);
    \fill (-2.10,0.00) circle (2.4pt);
    \fill (0.00,2.45) circle (2.4pt);
    \fill (0.00,-2.45) circle (2.4pt);
    \node[font=\scriptsize,inner sep=1.6pt,left] at (-3.15,-1.70) {$v_{10}$};
    \node[font=\scriptsize,inner sep=1.6pt,left] at (-3.15,1.70) {$v_{8}$};
    \node[font=\scriptsize,inner sep=1.6pt,right] at (5.05,0.00) {$v_{5}$};
    \node[font=\scriptsize,inner sep=1.6pt,above right] at (2.10,0.00) {$v_{4}$};
    \node[font=\scriptsize,inner sep=1.6pt,below] at (-4.35,-0.55) {$v_{11}$};
    \node[font=\scriptsize,inner sep=1.6pt,above] at (-1.45,3.35) {$v_{9}$};
    \node[font=\scriptsize,inner sep=1.6pt,right] at (3.70,-2.35) {$v_{7}$};
    \node[font=\scriptsize,inner sep=1.6pt, below right]
    at (1.35,-1.28) {$v_{6}$};
    \node[font=\scriptsize,inner sep=1.6pt,left] at (-2.10,0.00) {$v_{3}$};
    \node[font=\scriptsize,inner sep=1.6pt,above] at (0.00,2.45) {$v_{1}$};
    \node[font=\scriptsize,inner sep=1.6pt,below] at (0.00,-2.45) {$v_{2}$};
    \node at (0,-5.2) {$G$};
  \end{scope}
  \begin{scope}[xshift=12.00cm]
    \fill[black!6] (-3.15,1.70) -- (-4.35,0.55) -- (-2.10,0.00) -- cycle;
    \fill[black!6] (-3.15,1.70) -- (-2.10,0.00) -- (0.00,2.45) -- cycle;
    \fill[black!6] (-3.15,-1.70) -- (-1.45,-3.35) -- (0.00,-2.45) -- cycle;
    \fill[black!6] (-3.15,-1.70) -- (-2.10,0.00) -- (0.00,-2.45) -- cycle;
    \fill[black!6] (5.05,0.00) -- (3.70,2.35) -- (0.00,2.45) -- cycle;
    \fill[black!6] (5.05,0.00) -- (0.00,2.45) -- (0.00,-2.45) -- cycle;
    \fill[black!6] (2.10,0.00) -- (1.35,-1.28) -- (0.00,-2.45) -- cycle;
    \fill[black!6] (2.10,0.00) -- (0.00,2.45) -- (0.00,-2.45) -- cycle;
    \fill[black!6] (-2.10,0.00) -- (0.00,2.45) -- (0.00,-2.45) -- cycle;
    \draw[thin,black!60] (-3.15,1.70) -- (-4.35,0.55);
    \draw[thin,black!60] (-3.15,1.70) -- (-2.10,0.00);
    \draw[thin,black!60] (-3.15,1.70) -- (0.00,2.45);
    \draw[thin,black!60] (-3.15,-1.70) -- (-1.45,-3.35);
    \draw[thin,black!60] (-3.15,-1.70) -- (-2.10,0.00);
    \draw[thin,black!60] (-3.15,-1.70) -- (0.00,-2.45);
    \draw[thin,black!60] (5.05,0.00) -- (3.70,2.35);
    \draw[line width=0.9pt] (5.05,0.00) -- (0.00,2.45);
    \draw[line width=0.9pt] (5.05,0.00) -- (0.00,-2.45);
    \draw[thin,black!60] (2.10,0.00) -- (1.35,-1.28);
    \draw[line width=0.9pt] (2.10,0.00) -- (0.00,2.45);
    \draw[line width=0.9pt] (2.10,0.00) -- (0.00,-2.45);
    \draw[thin,black!60] (-4.35,0.55) -- (-2.10,0.00);
    \draw[thin,black!60] (-1.45,-3.35) -- (0.00,-2.45);
    \draw[thin,black!60] (3.70,2.35) -- (0.00,2.45);
    \draw[thin,black!60] (1.35,-1.28) -- (0.00,-2.45);
    \draw[line width=0.9pt] (-2.10,0.00) -- (0.00,2.45);
    \draw[line width=0.9pt] (-2.10,0.00) -- (0.00,-2.45);
    \draw[line width=0.9pt] (0.00,2.45) -- (0.00,-2.45);
    \fill (-3.15,1.70) circle (2.4pt);
    \fill (-3.15,-1.70) circle (2.4pt);
    \fill (5.05,0.00) circle (2.4pt);
    \fill (2.10,0.00) circle (2.4pt);
    \fill (-4.35,0.55) circle (2.4pt);
    \fill (-1.45,-3.35) circle (2.4pt);
    \fill (3.70,2.35) circle (2.4pt);
    \fill (1.35,-1.28) circle (2.4pt);
    \fill (-2.10,0.00) circle (2.4pt);
    \fill (0.00,2.45) circle (2.4pt);
    \fill (0.00,-2.45) circle (2.4pt);
    \node[font=\scriptsize,inner sep=1.6pt,left] at (-3.15,1.70) {$v_{8}$};
    \node[font=\scriptsize,inner sep=1.6pt,left] at (-3.15,-1.70) {$v_{10}$};
    \node[font=\scriptsize,inner sep=1.6pt,right] at (5.05,0.00) {$v_{5}$};
    \node[font=\scriptsize,inner sep=1.6pt,above right] at (2.10,0.00) {$v_{4}$};
    \node[font=\scriptsize,inner sep=1.6pt,above] at (-4.35,0.55) {$v_{9}$};
    \node[font=\scriptsize,inner sep=1.6pt,below] at (-1.45,-3.35) {$v_{11}$};
    \node[font=\scriptsize,inner sep=1.6pt,right] at (3.70,2.35) {$v_{7}$};
    \node[font=\scriptsize,inner sep=1.6pt,below right] at (1.35,-1.28) {$v_{6}$};
    \node[font=\scriptsize,inner sep=1.6pt,left] at (-2.10,0.00) {$v_{3}$};
    \node[font=\scriptsize,inner sep=1.6pt,above] at (0.00,2.45) {$v_{1}$};
    \node[font=\scriptsize,inner sep=1.6pt,below] at (0.00,-2.45) {$v_{2}$};
    \node at (0,-5.2) {$H$};
  \end{scope}
\end{tikzpicture}
\caption{The two $2$-trees of the second pair of Table~\ref{tab:chordal-counterexamples}.
  The triangles of each are shaded; vertices are numbered in the order of the assembly
  starting with $v_1v_2$.
  The bold edges are the common frame.  Interchanging $v_4$ with $v_5$ and $v_6$ with
  $v_7$ is an automorphism of $G$ but $H$ is asymmetric.
  The two graphs are not isomorphic, but $\LL_G=\LL_H$, $U_G=U_H$ and $\XB_G=\XB_H$.}
\label{fig:chordal-2trees}
\end{figure}

The two constructions agree at six of the nine steps and differ in three attachments only:
\[
  \begin{array}{c|ccc}
    & v_7 & v_9 & v_{11} \\ \hline
    G & v_2v_5 & v_1v_8 & v_3v_{10} \\
    H & v_1v_5 & v_3v_8 & v_2v_{10}
  \end{array}
\]
Neither the degree sequence nor any of the invariants of
Proposition~\ref{prop:computational-evidence} distinguishes the pair.

The automorphism groups give a direct proof of non-isomorphism.  Every
automorphism fixes the vertices of degree $5$, $6$ and $7$, namely $v_3$, $v_1$
and $v_2$, each of which is unique.  Of the four vertices of degree $3$, only
$v_4$ and $v_5$ are non-adjacent to $v_3$, so the pair $\{v_4,v_5\}$ is
preserved. The other two are distinguished by adjacency to $v_1$ and to $v_2$
respectively, hence $v_8$ and $v_{10}$ are fixed, and with them their unique
neighbors $v_9$ and $v_{11}$ of degree $2$.  The only possible nontrivial
automorphism therefore transposes $v_4$ and $v_5$, and with them the attached
vertices $v_6$ and $v_7$.  In $G$ both $v_6$ and $v_7$ are attached to $v_2$
and that transposition is an automorphism. In $H$ one is attached to $v_2$ and
the other to $v_1$, so it is not.  Thus $G$ has an automorphism group of order
two while $H$ is asymmetric, and in particular the two are not isomorphic.
The graphs are also distinguished by their induced $P_6$ counts: $H$ contains one on six
vertices, namely $v_{11}v_{10}v_3v_1v_4v_6$, and $G$ contains none, the two
graphs having equally many induced paths on $k$ vertices for every $k\le 5$.

The pair is cospectral for both the adjacency and Laplacian matrices
and has the same induced-subgraph counts through order five. Its
edge--triangle incidence graphs, whose two vertex classes are the
nineteen edges and nine triangles, are also isomorphic trees.

\subsection{Unbounded collision classes}
\label{subsec:unbounded}

The partial-join operation of Proposition~\ref{prop:partial-join}
produces arbitrarily large collision families with connected complements.
We first record a reconstruction criterion.

\begin{lemma}\label{lem:partial-join-injective}
  Let $G$ be a connected graph without a universal vertex and let $H$ be arbitrary with $u_0\in V(H)$ such that every component of $H-u_0$ has fewer than $|V(G)|$ vertices.
  If $G'$ also satisfies these hypotheses and $H\vee_{\{u_0\}}G\cong H\vee_{\{u_0\}}G'$, then $G\cong G'$.
\end{lemma}

\begin{proof}
Put $n=|V(G)|$ and $W=H\vee_{\{u_0\}}G$.
In the underlying simple graph of $W$, let $C$ be the induced subgraph
on $\{u_0\}\cup V(G)$. It is the cone on the underlying simple graph of
$G$, hence a $2$-connected block of order $n+1$.
Every block coming from $H$ has at most $n$ vertices, since every component of $H-u_0$ has fewer
than $n$ vertices. Thus $C$ is the unique largest block. Its unique universal vertex is $u_0$,
since $G$ has no universal vertex. The multigraph induced on the remaining vertices of that block is $G$.
An isomorphism $W\to H\vee_{\{u_0\}}G'$ maps the unique largest block to the unique largest block
and its unique universal vertex to the unique universal vertex, so it restricts to an isomorphism $G\to G'$.
\end{proof}

\begin{corollary}\label{cor:unbounded-connected-complement}
The sizes of collision classes of $\widehat{\LL}$ within connected simple
graphs with connected complement are unbounded.
\end{corollary}

\begin{proof}
Let $A,B$ be the eleven-vertex $2$-trees in the second pair of Theorem~\ref{thm:chordal-counterexamples}.
They have equal refined loopy polynomials, connected complements, and no universal vertices.
For $r\ge 1$ and $0\le a\le r$, let
\[
  G_{r,a}=A^{\vee a}\vee B^{\vee(r-a)},
\]
where a zero number of factors is omitted. By Proposition~\ref{prop:partial-join}, for fixed $r$ all $r+1$ graphs have the same refined polynomial.
The complement $\overline{G_{r,a}}$ is a disjoint union of $a$ copies of $\overline A$ and
$r-a$ copies of $\overline B$.
Since $\overline A$ and $\overline B$ are connected and non-isomorphic,
the graphs $\overline{G_{r,a}}$ are non-isomorphic for different values of $a$.
Therefore the graphs $G_{r,a}$ are connected non-isomorphic graphs on $11r$ vertices with the
same refined polynomial. Since they have no universal vertex, we can apply
Lemma~\ref{lem:partial-join-injective} with $H$ the path graph $P_5$ and $u_0$ its middle
vertex. The resulting $r+1$ partial joins $H\vee_{\{u_0\}}G_{r,a}$ are therefore pairwise
non-isomorphic.
They have the same refined polynomial by Proposition~\ref{prop:partial-join}, and both they
and their complements are connected. Indeed, in the complement each
of the four noncentral vertices of $P_5$ is adjacent to every vertex of
$G_{r,a}$, and the middle vertex is adjacent to the two ends of $P_5$.
This produces $r+1$ members in the same collision class,
which proves unboundedness.
\end{proof}

\subsection{Split graphs}
\label{subsec:split-graphs}

A simple graph is called \emph{split} if its vertex set can be partitioned $V=K\sqcup I$,
where $K$ induces a clique and $I$ is a stable set. Split graphs are chordal,
but none of our examples in Theorem~\ref{thm:chordal-counterexamples} is split (each contains
an induced subgraph $K_2\sqcup K_2$), so those examples do not settle the split-graph restriction.
Aliste-Prieto, Crew, Spirkl, and  Zamora~\cite[\S6.1]{ACSZ} explicitly attribute to Loebl and
Sereni~\cite{LS} the conjecture that the $U$-polynomial distinguishes split graphs.
Indeed, as explained in~\cite{Loe,LS}, Loebl's original motivation for his conjecture was an
observation that any graph  $G=(V,E)$ with at least four vertices can be reconstructed
from the split graph $\sp(G)=(V_{\sp},E_{\sp})$, where
\begin{equation}
  \label{eq:sp-graph}
 V_\sp = V\sqcup E \quad \text{and} \quad  E_\sp = \{\{u,v\}:u,v\in V,\ u\ne v\}
\sqcup\{\{e,v\}:e\in E,\ v\text{ is incident with }e\}.
\end{equation}
In this subsection we show that split counterexamples for Loebl's conjecture exist starting from $14$ vertices.

\subsubsection{Split collisions}
\label{sec:split-collisions}

\begin{theorem}[Split collisions]
  \label{thm:split-counterexamples}
  For simple split graphs on at most thirteen vertices, each of the invariants $U_G$, $\LL_G$ and $\widehat{\LL}_G$ distinguishes isomorphism classes.

  On fourteen vertices, there are $32$ collision classes, each with two non-isomorphic split graphs.
  The collision classes are the same for $U_G$, $\LL_G$ and $\widehat{\LL}_G$.
\end{theorem}

The sixteen pairs with $|K|=6$ and $|I|=8$ are listed in Table~\ref{tab:split-counterexamples}
of Appendix~\ref{subsec:app-split}, and the other sixteen, with $|K|=8$ and $|I|=6$, are their
complements. The edge counts in the table are $33$, $36$, $38$, $39$, $40$, $42$ and $45$, with
multiplicities $2,3,1,4,1,3,2$, and the complementary pairs have $91-m$ edges.
All $64$ graphs are connected with connected complement. In none of them does every vertex of $I$
have degree two, so none is of the kind $\sp(G)$. Thus the case that motivated the conjecture remains open.
We found no collision among such graphs up to twenty-one vertices (see Theorem~\ref{thm:sp-degree-two} below).

The complement of a split simple graph is split, with its clique and stable parts interchanged.
By Corollary~\ref{cor:refined-complement}, equality
of $\widehat{\LL}$ in each pair is preserved by complementation, and hence so are the equalities of $U$ and $\LL$ for these pairs. Thus the two halves of the list in
Theorem~\ref{thm:split-counterexamples} determine each other.

Here is the second pair of the table in full.  Let the set $ K=\{1,2,3,4,5,6\} $
form a clique, and let $I=\{w_1,\ldots,w_8\}$ be independent.  In both graphs put
\[
  \begin{aligned}
    N(w_1)&=\{2,3,6\},&  N(w_2)&=\{1,3,5\},&  N(w_3)&=\{1,2,4\},&  N(w_4)&=\{1,2\},&  N(w_5)&=\{1\}.
  \end{aligned}
\]
For the remaining vertices take
\[
  \begin{aligned}
    \text{ in } G: \quad & N(w_6)=\{3,5\},    \quad N(w_7)=\{2,6\}, & N(w_8)=\{1,4\}, \\
    \text{ in } H: \quad  &  N(w_6)=\{3,6\},    \quad N(w_7)=\{2,4\}, & N(w_8)=\{1,5\}.
  \end{aligned}
\]
Both graphs have $14$ vertices and $33$ edges, with degree sequence
\[
  (10,9,8,7,7,7,3,3,3,2,2,2,2,1).
\]
They are nevertheless non-isomorphic.  The clique vertices are those with degree
at least $7$, and vertices $1,2,3$ are distinguished by their degrees $10,9,8$.
The three independent vertices of degree $3$ are then distinguished by their intersections
with $\{1,2,3\}$. Their remaining neighbors distinguish clique vertices $6,5,4$, respectively.
Hence every clique vertex must be fixed by an isomorphism, whereas $G$ has an independent
vertex with neighborhood $\{1,4\}$ and $H$ does not.
In graph6 notation these graphs are \quad
\verb|M~~y`Ow@CGI_B?b??| \quad  and \quad \verb|M~~xcOw@AGI_B?b??|.

Coning produces split counterexamples at every larger order.

\begin{corollary}
\label{cor:higher-split-collisions}
For every $n\ge 14$, there exist non-isomorphic connected simple split graphs $G_n,H_n$ on $n$ vertices
such that $   \eta_{G_n}^{\ext}=\eta_{H_n}^{\ext}$.    Consequently,
$ U_{G_n}^{\ext}=U_{H_n}^{\ext},   \ \widehat{\LL}_{G_n}=\widehat{\LL}_{H_n},
\  U_{G_n}=U_{H_n},  \  \LL_{G_n}=\LL_{H_n}. $
\end{corollary}
\begin{proof}
Let $G,H$ be the displayed fourteen-vertex pair and, for $k\ge0$, put
\[
  G_{14+k}=G\vee K_k,\qquad H_{14+k}=H\vee K_k.
\]
These graphs are connected and split: the added clique joins the clique
part of a split partition. Applying Proposition~\ref{prop:partial-join}
to the operation $K_k\vee_{V(K_k)}\,\cdot\,$ preserves equality of
their extended polychromates.

The original graphs have maximum degree $10$, so the universal vertices
of each join are precisely the $k$ added vertices. Deleting these from
an isomorphism of the joins would give an isomorphism $G\cong H$,
which does not exist. The remaining polynomial equalities follow from
Theorem~\ref{thm:refined-extended-U},~\cite[Corollary~4.6]{MN}, and the
specializations~\eqref{eq:U-refined} and~\eqref{eq:refined-to-loopy}.
\end{proof}

The proof of Theorem~\ref{thm:split-counterexamples} is based on an exhaustive computation
described in Appendix~\ref{subsec:app-split}.  The numbers of split graphs for $n=1,\ldots,14$ are
\begin{equation}
  \label{eq:split-graphs}
  1 , \ 2 , \ 4 , \  9 , \  21 , \  56 , \ 164 , \ 557 , \ 2\,223 ,
  \ 10\,766 , \ 64\,956 , \ 501\,696 , \ 5\,067\,146 , \  67\,997\,750.
\end{equation}

Another fact the proof uses is that both $\LL_G$ and $U_G$ determine the clique and independence numbers of a split graph.

\begin{proposition}
\label{prop:split-clique-stable}
Let $G$ be a split graph on $n$ vertices, with clique number $\omega$ and independence number $\alpha$.
Let $V(G)=K\sqcup I$ be a split partition with $|K|=\omega$. Then $\omega+\alpha\in\{n,n+1\}$,
and $\omega+\alpha=n+1$ if and only if some vertex of $K$ has no neighbor in $I$.
Moreover, each of $\LL_G$ and $U_G$ determines $\omega$ and $\alpha$.
\end{proposition}

\begin{proof}
The set $I$ is stable, and every stable set meets $K$ in at most one vertex.
Thus $|I| \le \alpha \le |I|+1$, with equality on the right precisely when
$I\cup\{u\}$ is stable for some $u\in K$.
Each of $\LL_G$ and $U_G$ determines the induced edge count profile, by
Theorem~\ref{thm:induced-edge-profile} and \eqref{eq:edge-profile-from-polychromate},
with the trivial profile entries supplied by $n$.
Its largest clique and stable-set orders are $\omega$ and $\alpha$.
\end{proof}

\begin{remark}\label{rem:split-X}
The corresponding phenomenon for Stanley's chromatic symmetric function appears much earlier.
Aliste-Prieto, Crew, Spirkl, and Zamora~\cite[\S6.1]{ACSZ} construct non-isomorphic split
graphs with equal chromatic symmetric functions. An exhaustive check of split graphs on
at most seven vertices shows that the first collision occurs on seven vertices,
and that there is only one such pair, graphs \verb|F~ZC_| \quad  and \quad \verb|FhxxG|,
shown in Figure~\ref{fig:split-X-collision}. In each drawing,
the four filled vertices induce a clique and the three open vertices form an independent set.

\begin{figure}[ht]
  \centering
  \begin{tikzpicture}[
    clique vertex/.style={circle,fill=black,inner sep=1.8pt},
    stable vertex/.style={circle,draw=black,fill=white,
      line width=.6pt,inner sep=1.8pt},
    clique edge/.style={draw=gray!70,line width=.55pt},
    split edge/.style={draw=black,line width=.65pt}
    ]
    \begin{scope}[xshift=-3.2cm]
      \coordinate (g1) at (0,1.25);
      \coordinate (g2) at (-1.25,-.75);
      \coordinate (g3) at (1.25,-.75);
      \coordinate (g4) at (0,-.15);
      \coordinate (ga) at (.23,.55);
      \coordinate (gb) at (-1.38,.62);
      \coordinate (gc) at (0,-1.4);

      \draw[clique edge] (g1)--(g2)--(g3)--cycle;
      \draw[clique edge] (g4)--(g1) (g4)--(g2) (g4)--(g3);
      \draw[split edge] (ga)--(g1) (ga)--(g4);
      \draw[split edge] (gb)--(g1) (gb)--(g2);
      \draw[split edge] (gc)--(g2) (gc)--(g3);

      \node[clique vertex] at (g1) {};
      \node[clique vertex] at (g2) {};
      \node[clique vertex] at (g3) {};
      \node[clique vertex] at (g4) {};
      \node[stable vertex] at (ga) {};
      \node[stable vertex] at (gb) {};
      \node[stable vertex] at (gc) {};
      \node at (0,-1.92) {$G$};
    \end{scope}

    \begin{scope}[xshift=3.2cm]
      \coordinate (h1) at (0,1.25);
      \coordinate (h2) at (-1.25,-.75);
      \coordinate (h3) at (1.25,-.75);
      \coordinate (h4) at (0,-.15);
      \coordinate (ha) at (-1.38,.62);
      \coordinate (hb) at (.72,-.25);
      \coordinate (hc) at (.49,-.6);

      \draw[clique edge] (h1)--(h2)--(h3)--cycle;
      \draw[clique edge] (h4)--(h1) (h4)--(h2) (h4)--(h3);
      \draw[split edge] (ha)--(h1) (ha)--(h2);
      \draw[split edge] (hb)--(h3) (hb)--(h4);
      \draw[split edge] (hc)--(h3) (hc)--(h4);

      \node[clique vertex] at (h1) {};
      \node[clique vertex] at (h2) {};
      \node[clique vertex] at (h3) {};
      \node[clique vertex] at (h4) {};
      \node[stable vertex] at (ha) {};
      \node[stable vertex] at (hb) {};
      \node[stable vertex] at (hc) {};
      \node at (0,-1.92) {$H$};
    \end{scope}
  \end{tikzpicture}
  \caption{The unique pair of non-isomorphic split graphs on seven vertices with the same chromatic
    symmetric function.  Filled vertices induce a $K_4$, while open vertices form an independent set.
  }
  \label{fig:split-X-collision}
\end{figure}

Thus $X$ first has a split-graph collision at order seven, compared with
order fourteen for $U$, $\LL$, and $\widehat{\LL}$ in
Theorem~\ref{thm:split-counterexamples}.
\end{remark}

\subsubsection{Stable vertices of degree two}

The split graphs of the type $\sp(G)$~\eqref{eq:sp-graph} form a far smaller family, and for them the search for collisions can be pushed much further.  A split graph in which every vertex of
the independent set $I$ has degree two is isomorphic to $\sp(G)$ for a loopless multigraph $G$ with $V(G)=K$ and the edges of $G$ corresponding to $I$.
None of the $32$ pairs of Theorem~\ref{thm:split-counterexamples} is of this form.

\begin{theorem}\label{thm:sp-degree-two}
Let $G$ and $H$ be non-isomorphic loopless multigraphs each with at least three
vertices and satisfying $|V|+|E|\le21$.  Then $U_{\sp(G)}\ne U_{\sp(H)}$, and therefore
$\widehat{\LL}_{\sp(G)}\ne\widehat{\LL}_{\sp(H)}$. If both graphs satisfy $|V|+|E|\le18$, then
$\LL_{\sp(G)}\ne\LL_{\sp(H)}$.
\end{theorem}

\begin{proof}
The computation is described in Appendix~\ref{subsec:app-split}.
To interpret it in terms of $G$, we verify that $\sp(G)$ determines
the underlying loopless multigraph whenever $|V(G)|\ge3$.
The clique number of $\sp(G)$ is $|V(G)|$. If $|V(G)|\ge4$, the
original vertex set is its unique maximum clique, since a clique
containing an edge-vertex has at most three vertices. Each remaining
vertex represents an edge of $G$, and its two neighbors in the clique
recover that edge's endpoints, with repeated neighborhoods recording
parallel copies.

If $|V(G)|=3$ and $G$ has no isolated vertices, the original vertex set
is the unique triangle all of whose vertices have degree greater than
two; its complement and their neighborhoods again recover the edges.
Otherwise $G$ is an isolated vertex together with two vertices joined
by $|V(\sp(G))|-3$ parallel edges, including the edgeless case.
\end{proof}

Beyond the verified range we note only that the eight collision classes of
Section~\ref{sec:extended-U} at eight vertices give eight pairs
$\sp(G),\sp(H)$ on $21$, $22$ and $23$ vertices, and that $U$, and hence $\widehat{\LL}$, separates all eight,
while $\LL$ separates the two pairs of order~$21$.
Nevertheless the two members of each pair share the degree sequence, the triangle and
quadrilateral counts, and the whole matching polynomial.
The enumeration behind Theorem~\ref{thm:sp-degree-two}, and the reason this range is reachable at all, are described in Appendix~\ref{subsec:app-split}.

\section{Structural evidence}
\label{sec:structural-evidence}
We give three structural results concerning
Conjecture~\ref{conj:L-XB-equivalence}: a block-transform reformulation,
a rooted-tree reduction for unicyclic graphs with an unconditional
branch-size criterion, and a triangular relation shared by $\XB$ and
the normalized loopy polynomial.

\subsection{A blockwise reformulation}
\label{subsec:block-transform}
We reformulate the implication $U_G=U_H\Rightarrow\LL_G=\LL_H$
using connected-block polynomials, without requiring recovery of
$U^{\ext}$ from $U$.

For a finite set $V$ let $\Pi(V)$ denote the set of its partitions and
let $\lambda(\pi)$ be the integer partition of $|V|$ by the block sizes of $\pi$.
For a nonempty subset $B\subseteq V(G)$ define the \emph{connected-block} polynomial
\begin{equation}\label{eq:connected-block-polynomial}
  \cC_G(B;q) := \sum_{\substack{A\subseteq E(G[B])\\(B,A)\text{ connected}}} q^{|A|}.
\end{equation}
Grouping the spanning-subgraph
expansion~\eqref{eq:XB-subgraph} by connected-component partitions we obtain
\begin{equation}\label{eq:XB-connected-block-expansion}
  \XB_G(q;\bw)   =   \sum_{\pi\in\Pi(V(G))}  \left(\prod_{B\in\pi}\cC_G(B;q)\right)  \prod_{B\in\pi}p_{|B|}.
\end{equation}
For an integer partition $\lambda\vdash n$, set
\begin{equation*}
  S^G_\lambda(q):=[p_\lambda]\XB_G(q;\bw).
\end{equation*}
Then
\begin{equation}
 \label{eq:block-sum}
 S^G_\lambda(q) = \sum_{\substack{\pi\in\Pi(V(G))\\ \lambda(\pi)=\lambda}}  \prod_{B\in\pi}\cC_G(B;q)
\end{equation}
and
\[
  \XB_G=\sum_{\lambda\vdash n(G)}S^G_\lambda(q)p_\lambda.
\]
Thus $\XB_G$ records for every block type $\lambda$ the symmetrized
sum~\eqref{eq:block-sum} of products of connected-block polynomials.

When $G[B]$ is connected and $|B|=s$, the polynomial $\cC_G(B;q)$ is divisible by $q^{s-1}$,
so it belongs to $q^{s-1}\QQ[1+q]$.  For $s\ge 1$, define a linear map
\begin{equation}\label{eq:block-transform}
  \fT_s:q^{s-1}\QQ[1+q]\longrightarrow\QQ[x_0,x_1,\ldots],
  \qquad \fT_s\bigl(q^{s-1}(1+q)^r\bigr)=x_{s-1+r}   \quad(r\ge 0).
\end{equation}
In particular, for $s=1$, $r=0$,  this gives $\fT_1(1)=x_0$.

\begin{proposition}\label{prop:block-transform-L}
  For every graph $G$,
  \begin{equation}\label{eq:L-block-transform}
    \LL_G(1,\bx)   =   \sum_{\pi\in\Pi(V(G))}  \prod_{B\in\pi}  \fT_{|B|}\bigl(\cC_G(B;q)\bigr).
  \end{equation}
\end{proposition}

\begin{proof}
Fix an edge order on $G$.  For a subgraph $H$, denote by $\ST(H)$ the set of its spanning
trees and by $\eps_H(\tau)$ the external activity of a spanning tree $\tau$,
computed in $H$ with the induced edge order.
If $|B|=s$, the activity-interval decomposition \eqref{eq:activity-intervals},
applied to connected spanning subgraphs of $G[B]$, gives
\[
  \cC_G(B;q)  = \sum_{\tau\in\ST(G[B])}  q^{s-1}(1+q)^{\eps_{G[B]}(\tau)},
\]
where the sum is empty when $G[B]$ is disconnected.  Hence
\begin{equation}\label{eq:tau-sum}
  \fT_s\bigl(\cC_G(B;q)\bigr)  =  \sum_{\tau\in\ST(G[B])}  x_{s-1+\eps_{G[B]}(\tau)}.
\end{equation}
A spanning forest of $G$ is the same as a set partition $\pi$ of $V(G)$
with a spanning tree of $G[B]$ for every block $B\in\pi$.
External activity is computed independently inside these components.

Therefore multiplying~\eqref{eq:tau-sum} over the blocks of a partition and summing over
all partitions, as in~\eqref{eq:block-sum},
gives the forest expansion of $\LL_G(1,\bx)$, proving~\eqref{eq:L-block-transform}.
\end{proof}

Thus $\XB$ records symmetrized products of the connected-block
polynomials, whereas $\LL$ applies $\fT_s$ to each factor before forming
the products. This gives the following reformulation.

\begin{corollary}[Block-transform reduction]
\label{cor:block-transform}
Let $\cG$ be a class of simple graphs. The implication
\[
  \XB_G=\XB_H \quad \Longrightarrow \quad \LL_G=\LL_H
\]
holds for all $G,H\in\cG$ if and only if, for all $G,H\in\cG$ of the same order $n$,
\begin{equation*}
\begin{gathered}
S^G_\lambda(q)=S^H_\lambda(q)
\text{ for all }\lambda\vdash n\\
\Longrightarrow\quad
\sum_{\pi\in\Pi(V(G))}
  \prod_{B\in\pi}\fT_{|B|}\bigl(\cC_G(B;q)\bigr)
=
\sum_{\pi\in\Pi(V(H))}
  \prod_{B\in\pi}\fT_{|B|}\bigl(\cC_H(B;q)\bigr).
\end{gathered}
\end{equation*}
For the class of all simple graphs, this is the implication
$U_G=U_H\Longrightarrow\LL_G=\LL_H$ in
Conjecture~\ref{conj:L-XB-equivalence}.
\end{corollary}
\begin{proof}
Linear independence of the power sums $p_\lambda$ identifies the premise
with $\XB_G=\XB_H$, equivalently $U_G=U_H$ by~\eqref{eq:U-eta-XB}.
By Proposition~\ref{prop:block-transform-L}, the sums in the
right-hand side are $\LL_G(1,\bx)$ and $\LL_H(1,\bx)$.
Their equality is equivalent to $\LL_G=\LL_H$ by
Corollary~\ref{th:cor}(4).
\end{proof}

The analogous implication does not hold for the class of all loopless multigraphs. Indeed, Theorem~\ref{thm:MN-solved} gives graphs for which every $S_\lambda$
agrees but the total transformed sums differ.  Therefore a proof for simple graphs must use
simplicity, not merely the graphical origin of the connected-block polynomials.

To see what can be recovered directly, write
\[
 R^G_\lambda(\bx):=
 \sum_{\substack{\pi\in\Pi(V(G))\\\lambda(\pi)=\lambda}}
 \prod_{B\in\pi}\fT_{|B|}\bigl(\cC_G(B;q)\bigr).
\]
For a simple graph, a singleton block contributes
$\cC_G(\{v\};q)=1$ and $\fT_1(1)=x_0$. Consequently,
\[
 R^G_{(1^n)}=x_0^n,\qquad
 R^G_{(s,1^{n-s})}
 =x_0^{n-s}\fT_s\bigl(S^G_{(s,1^{n-s})}(q)\bigr)
 \quad(2\le s\le n).
\]
Thus a coefficient $S^G_\lambda$ determines its transformed contribution
whenever $\lambda$ has at most one nonsingleton part. With two or more
nonsingleton parts, the coefficient records a sum of products, and
linearity alone does not allow the factors to be transformed separately.
The example below shows that for arbitrary polynomial-valued set functions
equality of all the $S_\lambda$ can still fail to determine the total transformed sum.

Corollary~\ref{cor:block-transform} requires all the $S_\lambda$ to agree
and compares the total sums $\sum_\lambda R_\lambda$; it does not assert
separate implications $S^G_\lambda=S^H_\lambda\Rightarrow R^G_\lambda=R^H_\lambda$.

\begin{example}[Symmetrized sums do not control transformed sums]
\label{ex:block-transform-insufficient}
Let $V=\{1,2,3,4\}$ and consider functions
$f,g$ on the nonempty subsets of $V$ defined by
\[
  f(\{v\})=g(\{v\})=1,  \quad f(B)=g(B)=0 \text{ for } |B|\ge 3,  \quad f(B)=0 \text{ for } |B|=2,
\]
and by
\[
    g(12)=g(34)=q^2, \quad  g(13)=q^3, \quad g(24) =-q, \quad
    g(14)=0, \quad  g(23) = q-2q^2-q^3,
\]
where we used the notation $ij= \{i,j\}$.
The symmetrized sums for $f$ and $g$ agree for every partition type. Indeed, the only types
involving two-element blocks are $(2,1^2)$ and $(2,2)$, and
\[
  \sum_{|B|=2}g(B)=0,\qquad
  g(12)g(34)+g(13)g(24)+g(14)g(23)=q^4-q^4=0.
\]
All values lie in the domains of the maps~\eqref{eq:block-transform}.
For those maps,
\[
  \fT_2(q)=x_1,\qquad
  \fT_2(q^2)=x_2-x_1,\qquad
  \fT_2(q^3)=x_3-2x_2+x_1.
\]
Thus the transformed sum of type $(2,2)$ is  zero for $f$ and
\[
  (x_2-x_1)^2-(x_3-2x_2+x_1)x_1
  =x_2^2-x_1x_3\ne0
\]
for $g$.
The transformed sums of all other types
agree. Hence simultaneous equality of the symmetrized sums does not
imply equality of the total transformed sums for arbitrary
polynomial-valued set functions.
\end{example}

\subsection{Unicyclic graphs: a pointed-forest reduction}
\label{subsec:unicyclic}

The equivalence conjecture holds for forests by
Proposition~\ref{prop:forest-L-X}. A simple unicyclic graph is a disjoint
union $G_0\sqcup F'$ of its cyclic component and a forest.
The component-factorization argument of
Corollary~\ref{cor:connected-reduction} pairs components with equal
$\LL$- or $U$-polynomials. Their orders and cyclomatic numbers agree,
so cyclic components are paired with cyclic components; the tree
components are handled by Proposition~\ref{prop:forest-L-X}.
It therefore suffices to consider the connected cyclic components.

Now let $G$ be a connected simple graph with
unique cycle $C$, and put
\[
  n=|V(G)|=|E(G)|,\qquad g=|V(C)|\ge3.
\]
Fix an ordering of $E(G)$ and denote by $e_0$ the least edge of $C$.

\begin{lemma}\label{lem:unicyclic-activity}
For every spanning forest $F$ of $G$ one has $\eps(F)\in\{0,1\}$, and
\begin{equation}\label{eq:unicyclic-active}
  \eps(F)=1    \quad\Longleftrightarrow\quad   F\cap E(C)=E(C)\setminus\{e_0\}.
\end{equation}
In that case all vertices of $C$ lie in a single component $T_*(F)$ of $F$, and
\[
  \eps_F(T_*(F))=1,  \qquad  \eps_F(T)=0  \quad\text{for }T\ne T_*(F).
\]
\end{lemma}

\begin{proof}
Every edge outside $C$ is a bridge of $G$.  If such an edge is omitted from a
spanning forest, its endpoints lie in different components, so it is not externally active.

Now let $e\in E(C)\setminus F$.  The endpoints of $e$ are connected in $F$ if
and only if all the other edges of $C$ belong to $F$.
Thus an omitted cycle edge can be externally active only when it is the unique omitted edge of $C$,
and then the unique cycle of $F\cup\{e\}$ is $C$ itself, so $e$ is externally
active precisely when $e=e_0$.  This proves~\eqref{eq:unicyclic-active}.
When it holds, $C-e_0\subseteq F$ is a spanning path on the vertices of $C$,
so all cycle vertices lie in one component of $F$, and the unique active edge $e_0$
is assigned to that component.
\end{proof}

By~\eqref{eq:activity-from-monomial} the total activity of a forest is already
visible in its $\bx$-monomial.
Thus the active and inactive parts of $\LL_G$  can be separated without using
the $t$-grading, and we will work  with $\LL_G(1,\bx)$ from now on.

Introduce the \emph{inactive polynomial}
\[
  F_G(\bz) := \sum_{\substack{F\in\cF(G)\\ \eps(F)=0}} \prod_{T\in c(F)} z_{|V(T)|},
\]
and, for the active forests, a second family of variables
$\bz^\bullet=(z_1^\bullet,z_2^\bullet,\ldots)$ recording the component
that carries the cycle:
\begin{equation}\label{eq:unicyclic-pointed}
  B_G^\bullet(\bz;\bz^\bullet) := \sum_{\substack{F\in\cF(G)\\ \eps(F)=1}}
  z^\bullet_{|V(T_*(F))|} \prod_{\substack{T\in c(F)\\T\ne T_*(F)}} z_{|V(T)|},
  \qquad  B_G(\bz):=B_G^\bullet(\bz;\bz).
\end{equation}
Although the definition refers to $e_0$, the polynomial $B_G^\bullet$ does not depend on that choice.
Indeed, by Lemma~\ref{lem:unicyclic-activity}
an active forest is $C-e_0$ together with an arbitrary subset of the bridges of
$G$, and $C-e$ is a spanning path on $V(C)$ for \emph{every} $e\in E(C)$.
Hence the vertex sets of the resulting components are the same whichever cycle edge is removed.

\begin{proposition}\label{prop:unicyclic-two-specializations}
In the above notation, we have
\begin{equation}\label{eq:unicyclic-U-decomposition}
  U_G(\bz,y) = F_G(\bz)+y\,B_G(\bz),
\end{equation}
\begin{equation}\label{eq:unicyclic-L-decomposition}
  \LL_G(1,\bx)  =    \left.F_G(\bz)\right|_{z_s=x_{s-1}} +
  \left.B_G^\bullet(\bz;\bz^\bullet)\right|_{z_s=x_{s-1},\ z_s^\bullet=x_s}.
\end{equation}
Moreover $U_G$ and $\LL_G$ each determine $n$, $g$, and $F_G$.
\end{proposition}

\begin{proof}
Formula~\eqref{eq:unicyclic-U-decomposition} follows from~\eqref{eq:U-forest} and
Lemma~\ref{lem:unicyclic-activity}: an inactive forest contributes $\ds \prod_T z_{|V(T)|}$
and an active one contributes $\ds y\prod_T z_{|V(T)|}$ obtained from~\eqref{eq:unicyclic-pointed}
by forgetting the distinguished component.

For $\LL_G(1,\bx)$ an inactive component $T$ contributes $x_{|V(T)|-1}$.
In an active forest the distinguished component has activity one and contributes $x_{|V(T_*)|}$,
while every other component contributes $x_{|V(T)|-1}$.  This gives~\eqref{eq:unicyclic-L-decomposition}.

Both invariants determine the ordinary Tutte polynomial, which here is
\begin{equation}\label{eq:unicyclic-Tutte}
  T_G(X,Y)=X^{n-g}\bigl(X^{g-1}+X^{g-2}+\cdots+X+Y\bigr),
\end{equation}
since the $n-g$ edges outside $C$ are bridges.  Hence each determines $n$ and $g$.
Finally $F_G=\left.U_G\right|_{y=0}$, while on the loopy side the
inactive monomials are recognized by~\eqref{eq:activity-from-monomial} and the
relabeling $x_{s-1}\mapsto z_s$ returns the same polynomial.
\end{proof}

Therefore in the unicyclic case the equality of the invariants can be restated as follows.

\begin{corollary}\label{cor:unicyclic-reduction}
Let $G$ and $G'$ be as above. Then
 \begin{align}
    U_G=U_{G'}   &\iff   n=n',\ g=g',\ F_G=F_{G'}\ \text{ and }\ B_G=B_{G'},
                   \notag\\
    \LL_G=\LL_{G'}    &\iff   n=n',\ g=g',\ F_G=F_{G'}\ \text{ and }\
                      \left.B_G^\bullet\right|_{z_s^\bullet=z_{s+1}}   =\left.B_{G'}^\bullet\right|_{z_s^\bullet=z_{s+1}} .
                      \label{eq:unicyclic-L-iff}
 \end{align}
The two invariants therefore differ only in the way the distinguished
component of an active forest is forgotten: $U_G$ substitutes
$z^\bullet_s\mapsto z_s$, whereas $\LL_G$ substitutes $z^\bullet_s\mapsto z_{s+1}$.
\end{corollary}

\begin{proof}
Immediate from Proposition~\ref{prop:unicyclic-two-specializations},
after relabeling $x_i\mapsto z_{i+1}$ in~\eqref{eq:unicyclic-L-decomposition}.
\end{proof}

\subsubsection{A rooted-tree reformulation}

Collapse the cycle $C$ to a single vertex $\rho$ and discard the cycle edges.
The result is a rooted tree $(H,\rho)$ with $n_H:=|V(H)|=n-g+1 $ vertices.
By Lemma~\ref{lem:unicyclic-activity}, the active forests of $G$ are the sets
\[
  F_J:=(C-e_0)\cup J \quad \text{with} \quad J\subseteq E(H).
\]
If $R_J$ is the component of $(V(H),J)$ containing the root $\rho$, then the distinguished component
$T_*$ of $F_J$ satisfies $|V(T_*)|=g-1+|V(R_J)|$, while every other component has the same
size in $G$ as in $H$.  So if
\begin{equation*}
  \Psi_{H,g}(\bz) :=
  \sum_{J\subseteq E(H)} z_{\,g-1+|V(R_J)|} \prod_{\substack{D\in c(J)\\D\ne R_J}} z_{|V(D)|},
\end{equation*}
then
\[
  B_G=\Psi_{H,g}, \qquad \left.B_G^\bullet\right|_{z^\bullet_s=z_{s+1}}=\Psi_{H,g+1}.
\]
Thus Corollary~\ref{cor:unicyclic-reduction} says that, given the common data $(n,g,F_G)$,
the comparison of $U_G$ with $\LL_G$ is precisely the comparison of $\Psi_{H,g}$ with $\Psi_{H,g+1}$.
Both of them are specializations of the \emph{rooted $U$-polynomial} of the rooted tree $(H,\rho)$,
\begin{equation*}
  \Omega_H(w;\bz)  :=  \sum_{J\subseteq E(H)} w^{\,|V(R_J)|}
  \prod_{\substack{D\in c(J)\\D\ne R_J}} z_{|V(D)|}
\end{equation*}
introduced by Aliste-Prieto, de Mier and Zamora~\cite{ADZ},
which records the size of the root
component separately from the sizes of the remaining components.
This invariant determines the rooted tree up to isomorphism~\cite[Theorem~9]{ADZ}.
The corresponding statement for the rooted polychromate is due to Bollob\'as and Riordan~\cite{BR},
who introduced that invariant for this purpose.  By construction
\[
  \Psi_{H,g} =  \left.\Omega_H(w;\bz)\right|_{w^a\mapsto z_{g-1+a}}.
\]
Here the substitution is $\ZZ[\bz]$-linear: it is defined on the powers
of $w$, not by an algebra homomorphism. Thus $\Psi_{H,g}$ is determined
by $\Omega_H$, but need not a priori recover it.

Therefore the unicyclic case reduces to determining whether the two transformations $w^a\mapsto z_{g-1+a}$
and $w^a\mapsto z_{g+a}$ lose the same information.
We conjecture that each retains the rooted tree.

\begin{conjecture}[Root injectivity]\label{conj:root-injectivity}
For every $g\ge 3$, the invariant $\Psi_{H,g}$ separates rooted trees with a fixed number of vertices.
\end{conjecture}

\begin{proposition}\label{prop:root-injectivity-implies}
  Conjecture~\ref{conj:root-injectivity} implies the equivalence conjecture for simple unicyclic graphs.
\end{proposition}

\begin{proof}
As explained at the beginning of this subsection, it is sufficient to consider connected $G,G'$. Suppose first that $U_G=U_{G'}$.
By Corollary~\ref{cor:unicyclic-reduction} we have $n=n'$, $g=g'$, $F_G=F_{G'}$ and $\Psi_{H,g}=\Psi_{H',g}$,
where $H,H'$ are the associated rooted trees, both on $n-g+1$ vertices.
By Conjecture~\ref{conj:root-injectivity}, $H$ and $H'$ are isomorphic as rooted trees,
so $\Psi_{H,g+1}=\Psi_{H',g+1}$ and~\eqref{eq:unicyclic-L-iff} implies that $\LL_G=\LL_{G'}$.
The converse is identical with $g$ and $g+1$ interchanged.
\end{proof}

The rooted tree $H$ does not by itself determine $G$: collapsing the cycle forgets
how the attached branches are distributed among the cycle vertices and their cyclic
arrangement. This causes no problem in the preceding argument, because the inactive
polynomial $F_G$ is retained separately as common data.

\begin{remark}[Computational evidence and sharpness at $g=2$]
\label{rem:g-two-fails}
Only a finite range of $g$ has to be tested.
For a rooted tree $H$ on $n_H$ vertices the case $g\ge n_H$ is automatic:
in every monomial of $\Psi_{H,g}$ the root-component factor has index at least $g$,
whereas every other factor has index at most $n_H-1$, so the root factor
is uniquely identifiable and $\Psi_{H,g}$ recovers $\Omega_H$.
It is therefore sufficient to test $3\le g\le n_H-1$.

We checked all rooted trees $H$ with $n_H\le 13$ vertices (the corresponding counts for
$n_H=2,\ldots, 13$ are $1,2,4,9,20,48,115,286,719,1\,842,4\,766,12\,486$) and confirmed
that no two with the same $n_H$ have equal $\Psi_{H,g}$ in the above range. This confirms Conjecture~\ref{conj:root-injectivity} for rooted trees with at most $13$ vertices.

In particular, by Proposition~\ref{prop:root-injectivity-implies}
and the above component reduction, this verification implies the equivalence conjecture
for simple unicyclic graphs whose cyclic component $G_0$ and cycle $C$ satisfy $|V(G_0)|-|V(C)|+1\le 13$.
This includes every simple unicyclic graph on at most $15$ vertices.

The hypothesis $g\ge 3$ cannot be dropped.  For $g=2$ the map $H\mapsto \Psi_{H,2}$
is not injective: there is one colliding pair of rooted trees on $8$ vertices
and four on $11$ vertices.  Since a cycle in a simple graph has length at least three,
the case in which the statement fails cannot arise here.
\end{remark}

\subsubsection{Toward a proof}
The root factor is identifiable in any monomial with exactly one factor
of index at least $g$. The following bound gives a range in which this
identification is unambiguous.

\begin{proposition}[High-factor terms]\label{prop:unambiguous-terms}
  Let $M$ be a monomial of degree $p$ occurring in $\Psi_{H,g}$.   If
  \[
    p\ \ge\ n_H-g+2,
  \]
  then $M$ has exactly one factor of index at least $g$, namely the one contributed by the root component.
\end{proposition}

\begin{proof}
  Every monomial of $\Psi_{H,g}$ has index sum $(g-1+|V(R_J)|)+\bigl(n_H-|V(R_J)|\bigr)=n_H+g-1$.
  The root factor has index $g-1+|V(R_J)|\ge g$.  If some other factor also had index at least $g$, then,
  bounding the remaining $p-2$ factors below by $1$, we would get $n_H+g-1\ge 2g+(p-2)$, that is $p\le n_H-g+1$.
\end{proof}

The two highest ordinary-degree parts of $\Psi_{H,g}$, of degrees $n_H$ and $n_H-1$, show how
the reconstruction can begin. The degree-$n_H$ part of $\Psi_{H,g}$ is $z_gz_1^{n_H-1}$, and the degree-$(n_H-1)$ part,
for $n_H\ge 3$, is
\[
  \deg_H(\rho)\,z_{g+1}z_1^{n_H-2} + \bigl(n_H-1-\deg_H(\rho)\bigr)z_gz_2z_1^{n_H-3}.
\]
The two summands correspond to whether the chosen edge meets the root.
For $n_H=1$ and $n_H=2$, respectively, we have
\[
  \Psi_{H,g}=z_g,\quad \text{ and } \Psi_{H,g}=z_gz_1+z_{g+1}.
\]
Thus the root degree is recovered in every case.

\subsubsection{An unconditional subclass}
A bound on the branch sizes makes the distinguished factor identifiable
in every active monomial.

\begin{theorem}\label{thm:unicyclic-separated}
Let $G$ be a simple unicyclic graph with the unique cycle $C$ of length $g$, and
let $G_0$ be the component containing $C$. Suppose that every component of $G_0-V(C)$ has fewer than $g$ vertices.
If $G'$ is any simple graph with $U_G=U_{G'}$ or with $\LL_G=\LL_{G'}$, then $G'$
also has exactly one cycle of length $g$ and satisfies the same branch-size hypothesis. Moreover,
  \[
    U_G=U_{G'}\quad\Longleftrightarrow\quad\LL_G=\LL_{G'}.
  \]
\end{theorem}

\begin{proof}
By Corollary~\ref{cor:L-component-factorization} and Proposition~\ref{prop:U-XB-irreducible},
equality of invariants pairs the connected components of $G$ and $G'$ with equal corresponding polynomials.
Each pair has the same order and cyclomatic number, so $G'$ also has a unique cycle $C'$.
By Proposition~\ref{prop:forest-L-X} and multiplicativity, it remains to
prove the assertion when $G$ and $G'$ are connected. Since they have the same order $n=|V(G)|=|V(G')|$
and the same Tutte polynomials, formula~\eqref{eq:unicyclic-Tutte} implies $g=|C|=|C'|$.
Let $H'$ be the associated rooted tree of $G'$.

For an active forest of $G$, the distinguished component contains all $g$
vertices of $C$, whereas every other component lies in a component of $G-V(C)$
and therefore has fewer than $g$ vertices.
Therefore every monomial of $B_G=\Psi_{H,g}$ has exactly one factor of index at least $g$.

Suppose first that $U_G=U_{G'}$, so $B_G=B_{G'}$ by Corollary~\ref{cor:unicyclic-reduction}.
If some component $K$ of $G'-V(C')$ had $|V(K)|\ge g$, choose
$J=E(K)\subseteq E(H')$, omitting every other bridge. The active
forest so obtained has the cycle component and $K$ as distinct components,
and its monomial in $B_{G'}$ contains the two factors $z_g$ and $z_{|V(K)|}$,
both of index at least $g$.  This cannot occur in $B_G$.
Therefore every component of $G'-V(C')$ has fewer than $g$ vertices.

Now suppose instead that $\LL_G=\LL_{G'}$.
After the relabeling $x_i\mapsto z_{i+1}$, the active part of $\LL_G$ is $\Psi_{H,g+1}$.
Under the hypothesis on $G$, each of its monomials has a unique factor of index at
least $g+1$, and every other factor has index at most $g-1$.
If a component $K$ of $G'-V(C')$ had $|V(K)|\ge g$, the same choice of forest would produce
in $\Psi_{H',g+1}$ the root factor $z_{g+1}$ together with the separate factor $z_{|V(K)|}$.
If $|V(K)|=g$, this introduces a factor $z_g$, which never occurs as a nondistinguished factor on
the $G$-side. If $|V(K)|>g$, the monomial has two factors of index at least $g+1$.  Either possibility
contradicts equality of the active parts. Thus $G'$ again satisfies the same size hypothesis.

Under this hypothesis the pointing is recoverable term by term. From $B_G$, replace in each monomial
the unique factor $z_s$ with $s\ge g$ by $z_s^\bullet$. This recovers $B_G^\bullet$.
From the active part of $\LL_G(1,\bx)$, the distinguished component contributes the unique factor
$x_j$ with $j\ge g$, while every other component of size $a<g$ contributes $x_{a-1}$ with $a-1\le g-2$.
Replacing that unique factor by $z_j^\bullet$ and every remaining $x_i$ by $z_{i+1}$ again recovers $B_G^\bullet$.
Together with the common inactive polynomial $F_G$, Proposition~\ref{prop:unicyclic-two-specializations}
and Corollary~\ref{cor:unicyclic-reduction} show that each invariant reconstructs the other.
\end{proof}

\begin{corollary}\label{cor:unicyclic-long-cycle}
The equivalence conjecture holds for simple unicyclic graphs whose cycle $C$ has
length $g>|V(G_0)|/2$, where $G_0$ is the component containing $C$.
In particular, it holds if $g>|V(G)|/2$.
\end{corollary}

\begin{proof}
The graph $G_0-V(C)$ has $|V(G_0)|-g<g$ vertices, so
Theorem~\ref{thm:unicyclic-separated} applies.
\end{proof}

\subsection{A common triangular relation}\label{subsec:further-common-structure}

The normalized loopy polynomial
$\overline{\LL}_G=t^{-|E(G)|}\LL_G$ and the Tutte symmetric function
satisfy the same triangular relation.

\begin{proposition}[Common triangular modular relation]
  \label{prop:common-modular-relation}
  Let $u,v,w$ be distinct vertices and put
  \[
    e_1=uv,\qquad e_2=vw,\qquad e_3=uw.
  \]
  Then
  \begin{equation}\label{eq:XB-modular}
    \XB_{G+\{e_1,e_3\}}-\XB_{G+\{e_1\}} = \XB_{G+\{e_2,e_3\}}-\XB_{G+\{e_2\}}.
  \end{equation}
  Thus $\XB_G$ and $\overline{\LL}_G$ satisfy the same triangular four-term relation,
  namely~\eqref{eq:XB-modular} and~\eqref{eq:normalized-4term}, respectively.
\end{proposition}

\begin{proof}
  The statement for $\overline{\LL}$ is \eqref{eq:normalized-4term}.  For $\XB$, use~\eqref{eq:XB-coloring}.
  If $u$ and $w$ receive different colors, both differences in~\eqref{eq:XB-modular} vanish.
  If they receive the same color, the edges $uv$ and $vw$ have the same monochromatic status,
  and the two differences are equal term by term.
\end{proof}

Every linear identity generated by this triangular relation therefore
holds for both invariants. This does not establish
Conjecture~\ref{conj:L-XB-equivalence}: it does not show that the relation
generates all equalities of Tutte symmetric functions. The kernel
relations of Crew and Spirkl~\cite{CS} involve multigraphs; their use here
would require checking which relations are also annihilated by
$\overline{\LL}$.

\section{Polyhedral and parking-cell interpretation}\label{sec:polyhedral}

In this section we give a polyhedral interpretation of the specialization
\[
  t=q,\qquad x_i=1+q+\cdots+q^i,
\]
which appears in the Hilbert series of the external bizonotopal algebra.
For a graph $G=(V,E)$ and a subset $S\subseteq V$, let $\kp_G(S)$ denote the number of edges
of $G$ incident to at least one vertex of $S$.  Define the \emph{score polytope} of $G$ as
\begin{equation}\label{eq:score-polytope}
  P_G = \left\{\ba\in\RR_{\ge 0}^{V}: \ba(S)\le \kp_G(S)  \text{ for every }S\subseteq V \right\}.
\end{equation}
The same function $\kp_G$ occurs in the higher vertex-deletion
identity~\eqref{eq:higher-vertex-deletion}.
As shown in~\cite{KNSV}, the function $\kp_G$ is integral, nondecreasing, and submodular,
so $P_G$ is an integral polymatroid independence polytope. We recall the short
submodularity argument below because it controls the contraction fibers.
In the same paper the lattice points of $P_G$ were identified with the partial score
vectors of $G$ and with a homogeneous basis of the external bizonotopal algebra $\BZ_G^e$.
Therefore
\begin{equation}\label{eq:bizonotopal-lattice-enumerator}
  \Hil(\BZ_G^e;q)  = \sum_{\ba\in P_G\cap\ZZ^V}q^{|\ba|}.
\end{equation}

We realize deletion and loopy contraction by embeddings of polytopes.
Iteration produces disjoint piecewise integral-affine images of boxes
whose lattice points partition those of $P_G$; their union need not
cover the real polytope.

We fix an ordering and an orientation of
the edges of $G$.  These choices are inherited by the intermediate graphs in
the deletion-contraction process.  The orientation is used only to decide
which endpoint receives the unit translation in a deletion step.

\subsection{Contraction fibers and the deletion embedding}
\label{subsec:contraction-fibers}

\begin{lemma}\label{lem:kappa-submodular}
  For all $A,B\subseteq V$,
  \[
    \kp_G(A\cup B)+\kp_G(A\cap B)
    \le \kp_G(A)+\kp_G(B).
  \]
\end{lemma}

\begin{proof}
  For an edge $f$, let $\delta_f(S)$ be $1$ if $f$ is incident to $S$ and $0$
  otherwise.  Then
  \[
    \kp_G(S)=\sum_{f\in E(G)}\delta_f(S).
  \]
  For each fixed edge $f$,
  \[
    \delta_f(A\cup B)+\delta_f(A\cap B)  \le \delta_f(A)+\delta_f(B).
  \]
  Indeed, if $f$ is incident to neither $A$ nor $B$, both sides are zero; if it is incident
  to only one of them, both sides are one; and if it is incident to both, the right-hand side
  is two while the left-hand side is at most two. Summing over $f$ proves the claim.
\end{proof}

\begin{lemma}[Deletion embedding]\label{lem:deletion-embedding}
  Let $e=uv$ be a non-loop edge oriented from $u$ to $v$, and let $\be_u$ be the coordinate vector at $u$.
  The translation
  \[
    \delta_e:P_{G-e}\longrightarrow\RR^V,  \qquad  \delta_e(\bb)=\bb+\be_u,
  \]
  maps $P_{G-e}$ into $P_G$ and raises the coordinate sum by one.
\end{lemma}

\begin{proof}
  If $S$ contains neither endpoint of $e$, deletion of $e$ does not change
  $\kp_G(S)$ and the translation does not change the sum over $S$.
  If $u\in S$, then deleting $e$ lowers $\kp_G(S)$ by one, while adding $\be_u$
  raises the left-hand side by one.  Thus
  \[
    \bb(S)\le\kp_{G-e}(S)=\kp_G(S)-1 \quad\Longrightarrow\quad (\bb+\be_u)(S)\le\kp_G(S).
  \]
  If $v\in S$ but $u\notin S$, deletion again lowers the right-hand side by one,
  whereas the left-hand side is unchanged.  Hence the inequality for $P_{G-e}$
  is even stronger than the corresponding inequality for $P_G$.  Therefore
  $\delta_e(P_{G-e})\subseteq P_G$, and $|\bb+\be_u|=|\bb|+1$.
\end{proof}

Accordingly we set
\begin{equation}
  \label{eq:deletion-lattice-piece}
  B_e=\be_u+(P_{G-e}\cap\ZZ^V),
  \qquad
  A_e=(P_G\cap\ZZ^V)\setminus B_e,
\end{equation}
so that $B_e$ is the deletion part of the lattice points of $P_G$ and, by the lemma,
\begin{equation}
  \sum_{\ba\in B_e}q^{|\ba|}  = q\sum_{\bb\in P_{G-e}\cap\ZZ^V}q^{|\bb|}.
  \label{eq:deletion-enumerator}
\end{equation}

We turn to contraction.
Let $w$ be the vertex obtained by identifying $u$ and $v$, and let
\[
  V':=V(G/e)=(V\setminus\{u,v\})\cup\{w\}.
\]
Define
\[
  \pi:\RR^V\longrightarrow\RR^{V'}, \qquad \pi(\ba)_w=a_u+a_v, \qquad \pi(\ba)_z=a_z\quad(z\ne w).
\]
If $R\subseteq V'$ and $\widetilde R\subseteq V$ is its inverse image, then loopy contraction gives
\[
  \kp_{G/e}(R)=\kp_G(\widetilde R),  \qquad \pi(\ba)(R)=\ba(\widetilde R).
\]
Thus $\pi(P_G)\subseteq P_{G/e}$, and $\pi$ preserves total coordinate sum.
The fibers of $\pi$ have the following description.

\begin{lemma}[Contraction fibers]\label{lem:contraction-fibers}
  For every $\bc\in P_{G/e}$ the fiber   $\pi^{-1}(\bc)\cap P_G$ is the nonempty interval
  \begin{equation}
    \pi^{-1}(\bc)\cap P_G = \{\ba(r):L(\bc)\le r\le U(\bc)\},
    \label{eq:contraction-fiber}
  \end{equation}
  where $\ba(r)$ is the lift~\eqref{eq:fiber-parametrization} and the bounds $L(\bc)$ and $U(\bc)$ are given
  by~\eqref{eq:fiber-endpoints} below. If $\bc$ is integral, then $L(\bc)$ and $U(\bc)$ are integers.
\end{lemma}

\begin{proof}
Fix $\bc\in P_{G/e}$ and put $N=c_w$.  A lift of $\bc$ is determined by a real parameter $r$:
\begin{equation}  \label{eq:fiber-parametrization}
  a_v=r,\qquad a_u=N-r,  \qquad a_z=c_z\quad(z\notin\{u,v\}).
\end{equation}
Constraints for subsets containing both $u,v$, or neither of them, do not
depend on $r$.  If $A$ contains $u$ but not $v$, its inequality becomes
\[
  r\ge L_A:=N+\bc(A\setminus\{u\})-\kp_G(A),
\]
whereas for $B$ containing $v$ but not $u$ it becomes
\[
  r\le U_B:=\kp_G(B)-\bc(B\setminus\{v\}).
\]
Together with $0\le r\le N$, the fiber is therefore the
interval~\eqref{eq:contraction-fiber} with
\begin{equation}
  \label{eq:fiber-endpoints}
  L(\bc)=\max\!\left(0, \max_{\substack{A\ni u\\v\notin A}}L_A\right),
  \qquad  U(\bc)=\min\!\left(N,  \min_{\substack{B\ni v\\u\notin B}}U_B\right).
\end{equation}

We check that this interval is nonempty. Feasibility of $\bc$ and monotonicity of $\kp_G$
give $L_A\le N$ and $U_B\ge 0$. It remains to prove $L_A\le U_B$.  For such $A$ and $B$,
inclusion--exclusion for coordinate sums gives
\begin{align*}
  &N+\bc(A\setminus\{u\})+\bc(B\setminus\{v\})\\
  &\qquad =N+\bc((A\cup B)\setminus\{u,v\})+\bc(A\cap B).
\end{align*}
The first two terms on the right are the $\bc$-sum over the subset of $V'$
whose inverse image is $A\cup B$.  Hence feasibility of $\bc$, followed by
Lemma~\ref{lem:kappa-submodular}, gives
\begin{align*}
 N+\bc(A\setminus\{u\})+\bc(B\setminus\{v\}) &\le \kp_G(A\cup B)+\kp_G(A\cap B)\\
 &\le \kp_G(A)+\kp_G(B).
\end{align*}
This is equivalent to $L_A\le U_B$.  Thus every $\bc\in P_{G/e}$ has a nonempty fiber.
If $\bc$ is integral, all quantities in~\eqref{eq:fiber-endpoints} are integers,
so both endpoints are integers.
\end{proof}

We write $s_e(\bc)$ for the \emph{maximal lift} of $\bc$, obtained by taking $r=U(\bc)$
in~\eqref{eq:fiber-parametrization}.

\subsection{The one-step polyhedral geometry}
\label{subsec:one-step-polyhedral}

Put
\[
  d_e:=\be_v-\be_u.
\]
The fibers of $\pi$ are parallel to $d_e$, and increasing the parameter $r$
in~\eqref{eq:fiber-parametrization} moves in the $d_e$-direction.

Define the \emph{upper contraction boundary}
\[
 \Gamma_e(G):=  \left\{\ba\in P_G:  \ba+\lambda d_e\notin P_G \text{ for every } \lambda>0\right\}
\]
which consists of the upper endpoints of all contraction fibers.
Also define the translated deletion polytope
\[
  D_e(G):=\be_u+P_{G-e}.
\]

\begin{proposition}[Deletion-contraction geometry]
\label{prop:polyhedral-section}
Let $e=uv$ be a non-loop edge, oriented from $u$ to $v$.
\begin{enumerate}
\item[(i)] $\Gamma_e(G)$ is the support of a polyhedral subcomplex of $\partial P_G$.
  Specifically,
\begin{equation}
  \label{eq:upper-boundary-faces}
\Gamma_e(G)= \{\ba\in P_G:a_u=0\} \;\cup\!
\bigcup_{\substack{B\subseteq V\\v\in B,\ u\notin B}} \{\ba\in P_G:\ba(B)=\kp_G(B)\}.
\end{equation}

\item[(ii)] The restriction
\[
  \pi|_{\Gamma_e(G)}:\Gamma_e(G)\longrightarrow P_{G/e}
\]
is a piecewise integral-affine homeomorphism. Its inverse is the maximal-lift
section $s_e(\bc)=\ba(U(\bc))$.  In particular,
\begin{equation}
\label{eq:gamma-lattice-points}
\Gamma_e(G)\cap\ZZ^V  = s_e(P_{G/e}\cap\ZZ^{V'})=A_e.
\end{equation}

\item[(iii)] The deletion polytope has the intrinsic description
\begin{equation}
  \label{eq:deletion-movable}
  D_e(G)=\{\ba\in P_G:\ba+d_e\in P_G\}.
\end{equation}
Therefore $D_e(G)\cap\Gamma_e(G)=\varnothing$ and  $D_e(G)\cap\ZZ^V=B_e$.

\item[(iv)] The residual region
\[
  R_e(G):=P_G\setminus\bigl(D_e(G)\cup\Gamma_e(G)\bigr)
\]
contains no lattice points.
\end{enumerate}
\end{proposition}

\begin{proof}
Fix $\bc\in P_{G/e}$.  By~\eqref{eq:contraction-fiber}, its fiber is the interval
$L(\bc)\le r\le U(\bc)$ with the upper endpoint  $r=U(\bc)$.

For (i), a feasible point fails to move a positive distance in the
$d_e$-direction precisely when one of the inequalities that can obstruct an
increase of $r$ is tight.  These are the inequalities
\[
  a_u\ge 0 \quad \text{and} \quad  \ba(B)\le\kp_G(B), \qquad v\in B,\quad u\notin B.
\]
If none of them is tight, then all the gaps, as well as $a_u$, are positive.
Since only finitely many inequalities occur, a small positive move in the $d_e$-direction remains in $P_G$.
This proves \eqref{eq:upper-boundary-faces}.  The sets appearing there are faces of $P_G$.
Together with all their faces they form a polyhedral subcomplex of the boundary.

For (ii), every fiber has exactly one upper endpoint, so $\pi$ restricts
to a bijection from $\Gamma_e(G)$ to $P_{G/e}$.  Its inverse is
\[
  s_e(\bc)_v=U(\bc),\qquad  s_e(\bc)_u=c_w-U(\bc),\qquad  s_e(\bc)_z=c_z\quad(z\in V\setminus\{u, v\}).
\]

The function
\begin{equation}
  U(\bc)=\min\!\left( c_w, \min_{\substack{B\subseteq V\\v\in B,\ u\notin B}}
    \bigl(\kp_G(B)-\bc(B\setminus\{v\})\bigr)  \right)
  \label{eq:U-piecewise-linear}
\end{equation}
is the minimum of finitely many integral-affine functions.
Hence it is continuous and piecewise integral-affine.
Thus $s_e$ is continuous and piecewise integral-affine, while its inverse $\pi$ is integral linear.
If $\bc$ is integral, every term in~\eqref{eq:U-piecewise-linear} is an integer, so $s_e(\bc)$ is integral.
Conversely, $\pi$ sends integral points to integral points.  This proves the first equality
in~\eqref{eq:gamma-lattice-points}. The identification with $A_e$ will follow from (iii)--(iv).

For (iii), suppose first that $\ba=\be_u+\bb$ with $\bb\in P_{G-e}$.
Applying Lemma~\ref{lem:deletion-embedding} with the opposite orientation of $e$ shows that
$\bb+\be_v\in P_G$. Since  $ \ba+d_e=\bb+\be_v, $ we obtain $D_e(G)\subseteq\{\ba\in P_G:\ba+d_e\in P_G\}$.

Conversely, suppose that $\ba\in P_G$ and $\ba+d_e\in P_G$.  The $u$-coordinate
of $\ba+d_e$ is $a_u-1$, so $a_u\ge 1$.  Put $\bb=\ba-\be_u$.  If $S$ contains $u$,
then  $\bb(S)=\ba(S)-1\le\kp_G(S)-1=\kp_{G-e}(S).$
If $S$ contains $v$ but not $u$, feasibility of $\ba+d_e$ gives  $\ba(S)+1=(\ba+d_e)(S)\le\kp_G(S),$
so again  $\bb(S)=\ba(S)\le\kp_G(S)-1=\kp_{G-e}(S).$
If $S$ contains neither endpoint, nothing changes.  Hence $\bb\in P_{G-e}$,
which proves~\eqref{eq:deletion-movable}.  A point of $D_e(G)$ can move by one
unit in the positive $d_e$-direction, whereas a point of $\Gamma_e(G)$ cannot
move by any positive amount.  The two sets are therefore disjoint.
The equality $D_e(G)\cap\ZZ^V=B_e$ follows from~\eqref{eq:deletion-lattice-piece}.

Finally, let $\ba=\ba(r)$ lie over $\bc=\pi(\ba)$.  If $\ba\notin D_e(G)$, then
$r+1$ is not feasible, so $r>U(\bc)-1$.  If also $\ba\notin\Gamma_e(G)$, then $r<U(\bc)$.
Thus every point of $R_e(G)$ satisfies
\begin{equation*}
  U(\bc)-1<r<U(\bc).
\end{equation*}
If $\ba$ were integral, then $\bc$ and $r=a_v$ would be integral and, by (ii),
$U(\bc)$ would be an integer. But there are no integers strictly between integers
$U(\bc)-1$ and $U(\bc)$, so $R_e(G)$ contains no lattice points.

By (iii)--(iv), every lattice point of $P_G$ belongs to exactly one of $D_e(G)$ and $\Gamma_e(G)$.
Since $D_e(G)\cap\ZZ^V=B_e$ and $A_e=(P_G\cap\ZZ^V)\setminus B_e$ by~\eqref{eq:deletion-lattice-piece},
we obtain
\[
  s_e(P_{G/e}\cap\ZZ^{V'})   =\Gamma_e(G)\cap\ZZ^V   =(P_G\cap\ZZ^V)\setminus B_e   = A_e.
\]
This proves the second equality in~\eqref{eq:gamma-lattice-points} and completes (ii).
\end{proof}

The proposition gives the following lattice-point recursion.

\begin{corollary}[Recursive lattice-point decomposition]
  \label{cor:recursive-cell-decomposition}
  There is a disjoint decomposition
  \[
    P_G\cap\ZZ^V=A_e\sqcup B_e,
  \]
  where $B_e$ is as in~\eqref{eq:deletion-lattice-piece} and $A_e$ consists of
  the maximal integral lifts $s_e(\bc)$, $\bc\in P_{G/e}\cap\ZZ^{V'}$.
  The map $s_e$ is a degree-preserving bijection, so
  \begin{equation}
    \label{eq:one-step-enumerators}
    \sum_{\ba\in A_e}q^{|\ba|}  = \sum_{\bc\in P_{G/e}\cap\ZZ^{V'}}q^{|\bc|},
    \qquad  \sum_{\ba\in B_e}q^{|\ba|}  = q\sum_{\bb\in P_{G-e}\cap\ZZ^V}q^{|\bb|}.
  \end{equation}
\end{corollary}

\begin{proof}
  By Proposition~\ref{prop:polyhedral-section}(iii)--(iv), $P_G$ is the disjoint
  union of $D_e(G)$, $\Gamma_e(G)$, and the lattice-free region $R_e(G)$, so the
  lattice points of $P_G$ are those of $D_e(G)$ together with those of $\Gamma_e(G)$.
  These are $B_e$ and, by~\eqref{eq:gamma-lattice-points}, $A_e$.
  Part~(ii) identifies $A_e$ with $s_e(P_{G/e}\cap\ZZ^{V'})$ and $s_e$ is
  inverse to $\pi$, which preserves the coordinate sum.
  This gives the first identity in~\eqref{eq:one-step-enumerators}.
  The second is~\eqref{eq:deletion-enumerator}.
\end{proof}

\begin{remark}\label{rem:one-step-geometry}
In a fiber of $\pi$, the deletion image consists of the points from which
one can move a full unit in the $d_e$-direction, while the contraction
image is the upper endpoint. The remaining interval is contained in
$U(\bc)-1<r<U(\bc)$ and has no lattice points. Thus the two images
partition the lattice points, not necessarily the entire real fiber;
see Figure~\ref{fig:contraction-fiber}.
\end{remark}

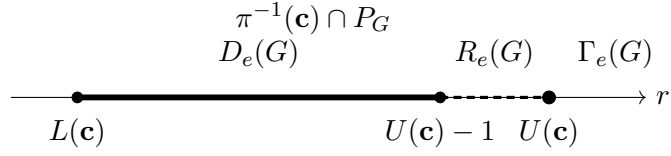
\begin{figure}[ht]
  \centering
  \begin{tikzpicture}[x=1.6cm,y=1cm]
    \draw[->] (0,0) -- (5.25,0) node[right] {$r$};
    \draw[line width=2.2pt] (.55,0) -- (3.55,0);
    \draw[densely dashed,line width=1.2pt] (3.55,0) -- (4.45,0);
    \filldraw (.55,0) circle (2pt);
    \filldraw (3.55,0) circle (2pt);
    \filldraw (4.45,0) circle (2.4pt);
    \node[below=4pt] at (.55,0) {$L(\bc)$};
    \node[below=4pt] at (3.55,0) {$U(\bc)-1$};
    \node[below=4pt] at (4.45,0) {$U(\bc)$};
    \node[above=7pt] at (2.05,0) {$D_e(G)$};
    \node[above=7pt] at (4.0,0) {$R_e(G)$};
    \node[above right=7pt] at (4.45,0) {$\Gamma_e(G)$};
    \node at (2.5,1.05) {$\pi^{-1}(\bc)\cap P_G$};
  \end{tikzpicture}
  \caption{Schematic geometry of a contraction fiber.  The dashed portion is the
    lattice-free residual strip; the upper endpoint represents loopy contraction.}
  \label{fig:contraction-fiber}
\end{figure}

\subsection{Geometric parking complexes}
\label{subsec:geometric-parking}

We now iterate the one-step geometry along the same decreasing edge order used
in the deletion-contraction proof of the spanning-forest formula.
At each deletion step in an intermediate graph $H$,
the child polytope is embedded by the integral-affine translation
\begin{equation}
  \label{eq:deletion-embedding}
  \delta_e:P_{H-e}\longrightarrow P_H,  \qquad \bb\longmapsto \bb+\be_u.
\end{equation}
At each contraction step, the child polytope is embedded by the maximal-lift section
\begin{equation}
  s_e:P_{H/e}\longrightarrow\Gamma_e(H)\subseteq P_H.
  \label{eq:contraction-section}
\end{equation}
The first map raises the total coordinate sum by one, whereas the second preserves it.

Fix a spanning forest $F$ and follow the corresponding leaf of the deletion-contraction tree.
The edges of $F$ are the edges chosen for contraction.  A nonforest edge is deleted unless,
when its turn is reached, its endpoints have already been identified.  Since edges are processed
from largest to smallest, this happens precisely when the edge is externally active with respect to $F$.
Such an edge has already become a loop and is retained.

So if $T$ is a component of $F$, the vertex corresponding to $T$ in the terminal loop-only graph carries
\[
  m_T=|E(T)|+\eps_F(T)
\]
loops.  The score polytope of the terminal graph is the box
\begin{equation*}
  \Box_F:= \prod_{T\in c(F)}[0,m_T] = \prod_{T\in c(F)} [0,\,|E(T)|+\eps_F(T)]  \ \subset\ \RR^{c(F)},
\end{equation*}
with coordinates indexed by the components of $F$.
Starting from this box and reversing the leaf path, compose the maps
\eqref{eq:deletion-embedding} and~\eqref{eq:contraction-section}.  Once the edge
order and orientations have been fixed, this gives a canonically defined
continuous injective piecewise integral-affine map
\[
  \Phi_F:\Box_F\longrightarrow P_G.
\]
We call the image
\[
  \fC_F:=\Phi_F(\Box_F)
\]
the \emph{geometric parking complex} associated with $F$. The set of its lattice points
\begin{equation}  \label{eq:cells}
  C_F:=\fC_F\cap\ZZ^V
\end{equation}
is called a \emph{parking cell} corresponding to $F$.
\begin{theorem}[Polyhedral parking cells]
  \label{thm:polyhedral-parking-cells}
  Fix an ordering and an orientation of the edges of $G$.
  For every spanning forest $F$, the geometric parking complex $\fC_F\subseteq P_G$
  has the following properties.
  \begin{enumerate}
  \item[(i)] The map
    \[
      \Phi_F:\Box_F\longrightarrow\fC_F
    \]
    is a piecewise-linear homeomorphism.  After subdividing $\Box_F$ by
    the finitely many domains of linearity of the maximal-lift sections occurring  along the path,
    $\fC_F$ is the support of a finite rational polyhedral complex.
    On every cell, $\Phi_F$ is integral-affine, and $\Phi_F^{-1}$ is the restriction to $\fC_F$
    of an integral-affine map.

  \item[(ii)] If $F\ne F'$, then
    \[
      \fC_F\cap\fC_{F'}=\varnothing.
    \]

  \item[(iii)] The lattice points of the geometric parking complexes give the partition into parking cells~\eqref{eq:cells}
    \begin{equation}
      \label{eq:parking-lattice-partition}
      P_G\cap\ZZ^V  = \bigsqcup_{F\in\cF(G)}C_F.
    \end{equation}
    Moreover $\Phi_F$ restricts to a lattice bijection
    \begin{equation}
      \label{eq:parking-chain-product}
      \prod_{T\in c(F)} \{0,1,\ldots,|E(T)|+\eps_F(T)\} \xrightarrow{\ \sim\ } C_F.
    \end{equation}

  \item[(iv)] Put
    \[
      d(F):=|E(G)|-|F|-\eps(F),    \qquad    \eps(F)=\sum_{T\in c(F)}\eps_F(T).
    \]
    Then for every $y=(y_T)_{T\in c(F)}\in\Box_F$ we have
    \begin{equation}
      \label{eq:parking-grading}
      \sum_{v\in V}(\Phi_F(y))_v  = d(F)+\sum_{T\in c(F)}y_T.
    \end{equation}
    Therefore
    \begin{equation}
      \label{eq:parking-cell-enumerator}
      \sum_{\ba\in C_F}q^{|\ba|} = q^{d(F)} \prod_{T\in c(F)}\left(1+q+\cdots+ q^{|E(T)|+\eps_F(T)}\right).
    \end{equation}

  \item[(v)] The complement
    \[
      P_G\setminus\bigcup_{F\in\cF(G)}\fC_F
    \]
    is lattice-free.
  \end{enumerate}
\end{theorem}

\begin{proof}
The terminal polytope at the leaf corresponding to $F$ is the box $\Box_F$. Reversing
the leaf path, every deletion step applies an injective integral-affine translation,
while every contraction applies the injective maximal-lift section of
Proposition~\ref{prop:polyhedral-section}. Thus their composition $\Phi_F$ is continuous
and injective.

Each maximal-lift section is piecewise integral-affine because its upper coordinate
is the minimum of finitely many integral-affine functions, as in~\eqref{eq:U-piecewise-linear}. Subdivide $\Box_F$ by pulling
back all the domains of linearity that occur along the path and taking their common
refinement. This gives a finite rational polyhedral subdivision on which $\Phi_F$ is
integral-affine cell by cell. Because $\Phi_F$ is globally injective, for any two cells
$C_1,C_2$ of this subdivision one has
\[
  \Phi_F(C_1)\cap\Phi_F(C_2)=\Phi_F(C_1\cap C_2).
\]
The intersection $C_1\cap C_2$ is a common face of $C_1$ and $C_2$, and the restrictions
of $\Phi_F$ to the cells are affine bijections.
Hence its image is a common face of $\Phi_F(C_1)$ and $\Phi_F(C_2)$.
The image cells therefore form a finite rational polyhedral complex.
Each map used along the path has a global left inverse: the translation by $-\be_u$
for a deletion step, and the integral linear projection $\pi$ for a contraction step.
Hence $\Phi_F^{-1}$ is the restriction to $\fC_F$ of the composition of these translations
and integral linear projections, taken in the reverse order. In particular, $\Phi_F^{-1}$
is continuous and integral-affine on all of $\fC_F$, and $\Phi_F$ is a homeomorphism onto
$\fC_F$. This proves (i).

For (ii), take distinct forests $F$ and $F'$ and compare their leaf paths.
Let the first difference occur when a non-loop edge $e$ is processed in some intermediate
graph $H$. One path takes deletion, so at that stage its image is contained in
\[
  D_e(H)=\be_u+P_{H-e},
\]
whereas the other takes contraction, so its image is contained in $\Gamma_e(H)$.
These two sets are disjoint by
Proposition~\ref{prop:polyhedral-section}(iii).  Up to that stage the two leaf
paths coincide, so the sequence of maps already applied is the same for both.
The common injective composition carrying these images back to $P_G$
therefore preserves their disjointness.

For (iii), every map used in the construction respects the relevant integer
lattices in both directions.  This is immediate for translations.  For a
maximal-lift section it follows from
Proposition~\ref{prop:polyhedral-section}(ii): an integral point of the
contracted polytope has an integral maximal lift, and the inverse projection
sends integral points to integral points.  Hence
\[
  \fC_F\cap\ZZ^V  = \Phi_F(\Box_F\cap\ZZ^{c(F)}),
\]
and the lattice points of $\Box_F$ are the product in~\eqref{eq:parking-chain-product}.

It remains to see that every lattice point of $P_G$ occurs.
At one non-loop edge, Proposition~\ref{prop:polyhedral-section}(iii)--(iv) says that every lattice
point belongs to exactly one of the deletion image $D_e(H)$ and the contraction boundary $\Gamma_e(H)$.
Repeating this decision down the deletion-contraction tree eventually sends every lattice point to a unique
loop-only leaf. The leaf is indexed by a unique spanning forest, proving
\eqref{eq:parking-lattice-partition}.

For (iv), a contraction section preserves total coordinate sum,
whereas every deletion translation raises it by one.
Along the leaf associated with $F$, the contracted edges are the $|F|$ forest edges,
and the nonforest edges that have become loops are precisely
the $\eps(F)$ externally active edges.
Hence the number of deletions is
\[
  |E(G)|-|F|-\eps(F)=d(F).
\]
At the terminal box the total coordinate sum is $\sum_T y_T$,
which proves
\eqref{eq:parking-grading}.
Restricting to the lattice points of the box gives
\eqref{eq:parking-cell-enumerator}.

Finally, (v) follows immediately from (iii):
every lattice point of $P_G$ lies
in one of the $\fC_F$, so their complement contains none.
\end{proof}

\subsection{The Hilbert-series specialization}
\label{subsec:hilbert-specialization}

Summing the parking-cell enumerators recovers~\eqref{eq:external-biz-specialization}.

\begin{corollary}\label{cor:polytope-specialization}
For every graph $G$,
\[
  \sum_{\ba\in P_G\cap\ZZ^V}q^{|\ba|}    =\LL_G(q,1,1+q,1+q+q^2,\ldots).
\]
Equivalently,
\[
  \Hil(\BZ_G^e;q)  =\LL_G(q,1,1+q,1+q+q^2,\ldots).
\]
\end{corollary}

\begin{proof}
By Theorem~\ref{thm:polyhedral-parking-cells},
\[
  \sum_{\ba\in P_G\cap\ZZ^V}q^{|\ba|} = \sum_{F\in\cF(G)}    q^{|E(G)|-|F|-\eps(F)}  \prod_{T\in c(F)}
  \left(1+q+\cdots +  q^{|E(T)|+\eps_F(T)}\right).
\]
By the spanning-forest formula of Theorem~\ref{thm:forests-sum}, the right-hand side is precisely
\[
  \LL_G(q,1,1+q,1+q+q^2,\ldots).
\]
The Hilbert-series statement follows from
\eqref{eq:bizonotopal-lattice-enumerator}.
\end{proof}

In this realization, $x_{|E(T)|+\eps_F(T)}$ corresponds to an interval
factor of the terminal box, and the exponent of $t$ counts the deletion
steps.

\section{Outlook}\label{sec:outlook}

We collect questions concerning invariant recovery and the geometry of
the score polytope.

\subsection{The equivalence conjecture and the common refinement}
\label{subsec:outlook-equivalence}

The central remaining open problem is Conjecture~\ref{conj:L-XB-equivalence}:
do $\LL_G$ and $U_G$ distinguish the same simple graphs?

\medskip

A related question asks whether the ordinary loopy polynomial recovers
its refinement. Theorem~\ref{thm:MN-solved} shows that
$U^{\ext}\to U$ loses distinguishing power on loopless multigraphs,
whereas no loss was found for $\widehat{\LL}\to\LL$ in the
$48\,229\,871$ connected loopless multigraphs tested.

\begin{problem}\label{prob:extended}
Does the implication
 \quad $  \LL_G=\LL_H \quad \Longrightarrow\quad  \widehat{\LL}_G=\widehat{\LL}_H $ \quad
hold for all loopless multigraphs?
More generally, when does $\LL$ determine  $\widehat{\LL}$ if loops are allowed?
\end{problem}
A positive answer for simple graphs would give the part
$\LL_G = \LL_H\Longrightarrow \XB_G=\XB_H$ of Conjecture~\ref{conj:L-XB-equivalence}
by Corollary~\ref{cor:refined-specialized}.

The opposite implication is the block-transform problem of
Corollary~\ref{cor:block-transform}. For simple unicyclic graphs,
Conjecture~\ref{conj:root-injectivity} would establish both implications
by Proposition~\ref{prop:root-injectivity-implies}. A next natural case is that
of simple graphs with $\nu(G)=2$. Their forest expansion has total activities
$0$, $1$, and $2$. In the activity-one part, one component is distinguished.
In the activity-two part, either one component carries both active edges or two
components carry one each. Extending the pointed-forest method requires keeping
these cases separate and formulating the corresponding recovery problem.

Every invariant determined by $\widehat{\LL}$ gives a weaker recovery
question for $\LL$. We single out two examples.

The first is the generalized degree polynomial.
By Corollary~\ref{cor:gdp} $\GD_G$ is determined by $\widehat{\LL}_G$.
The pair of Remark~\ref{rem:ising-strict} with the same generalized degree polynomial but
different loopy polynomials shows that $\GD_G$ does not determine $\LL_G$.
The reverse determination $\LL_G\rightarrow\GD_G$ remains open.

\begin{problem}\label{prob:L-gdp}
 Does $\LL_G$ determine the generalized degree polynomial $\GD_G$?
\end{problem}
\noindent
In the ranges of Proposition~\ref{prop:computational-evidence} and
Section~\ref{subsubsec:multiple-edges} the answer is positive,
because the $\LL$- and $\widehat{\LL}$-classes coincide there.

The second example is the degree sequence of a graph with loops.
By Theorem~\ref{thm:degree-sequence-loops} it is determined by $\widehat{\LL}_G$, whereas Theorem~\ref{thm:L-degree-sequence} shows only that it can be recovered from $\LL_G$ in the loopless case.

\begin{problem}\label{prob:loop-degree}
Does $\LL_G$ determine the degree sequence of an arbitrary graph $G$, loops included?
\end{problem}

One may also ask whether $\LL_G$ determines $\cV_G$ on loopless
multigraphs; an affirmative answer to the loopless assertion in
Problem~\ref{prob:extended} would imply this by~\eqref{eq:refined-to-V}.

\subsection{Complementation, induced statistics, and reconstruction}
\label{subsec:outlook-complementation}

Proposition~\ref{prop:XB-complement} and
Corollary~\ref{cor:refined-complement} give complement recovery for both
the ordinary $U$-polynomial and the common refinement on simple graphs.
Does this property descend to the loopy specialization?

\begin{problem}\label{prob:complement}
Does $\LL_G$ determine $\LL_{\overline G}$ for every simple graph $G$?
Is there a transformation intrinsic to the loopy polynomial?
\end{problem}

The determination statement follows from
Conjecture~\ref{conj:L-XB-equivalence}. It also follows from the
simple-graph case of Problem~\ref{prob:extended}, together with
Corollary~\ref{cor:refined-complement}. It holds through eleven
vertices by Proposition~\ref{prop:computational-evidence}.
Independently of these conjectures,
Corollary~\ref{cor:independence-clique}(3) gives a direct partial answer:
$\LL_G$ determines the entire induced edge count profile of $\overline G$.

The higher vertex-deletion identity suggests a broader reconstruction question.
It encodes weighted sums of the polynomials $\LL_{G-S}$ and yields the complete
induced edge count profile, but the profile does not determine all induced-subgraph data.
For example, two eight-vertex graphs with the same profile have, respectively,
$6$ and $8$ induced copies of $P_4$ (graph6 strings
\texttt{GEzvV\char123} and \texttt{GEztv\char123}).  It would be interesting
to determine which symmetric functions of the vertex-deleted collection
$\{\LL_{G-v}:v\in V(G)\}$ are recoverable from $\LL_G$, and whether the loopy
polynomial is reconstructible from a vertex- or edge-deleted deck.

\subsection{Collision statistics}
\label{subsec:outlook-statistics}

The collision data of Section~\ref{sec:extended-U} raise questions independent
of which one of the equivalent invariants is used.  Over the computed range,
the proportion of connected simple graphs in nontrivial classes decreases from
approximately one in $695$ at $n=8$ to one in $22\,366$ at $n=11$. Does this proportion tend to zero?

Collision classes of size greater than two first occur at ten vertices.
Corollary~\ref{cor:unbounded-connected-complement} shows that their sizes
are unbounded, even among connected simple graphs with connected
complement.

What this does not settle is how much of the census is generated in this way.
At ten vertices all four classes of size four consist of graphs with connected complement,
as do all four classes of size three at eleven, whereas the four classes of size four at
eleven do not. None of the classes in the computed range is large enough for the
construction above to be relevant to it.
Which large collision families arise from smaller ones by complementation,
joins, or partial joins, and which require other constructions?

\begin{remark}
\label{rem:tutte-profile-n8}
On eight vertices there are fourteen classes of non-isomorphic connected simple graphs
with equal Tutte polynomial and equal induced edge count profile $N_G(s,j)$.
The graphs in each class also have equal bivariate Ising polynomials, so at this order $Z$ adds
nothing to the combination of $T_G$ and $N_G(s,j)$.
Eight of the fourteen classes are the collision classes of Proposition~\ref{prop:computational-evidence}.
The remaining six pairs, for example \verb+G?bbe[+ and \verb+G?bcuw+ on twelve edges,
agree for $T_G$, $Z$, and $N_G(s,j)$ but not for $\LL_G$.
\end{remark}

\subsection{Collisions among \texorpdfstring{$2$}{2}-trees}
\label{subsec:outlook-2trees}

The $2$-tree pair in Theorem~\ref{thm:chordal-counterexamples} suggests
further tests of Conjecture~\ref{conj:L-XB-equivalence} and
Problem~\ref{prob:extended} within this class.

At eleven vertices $45\,010$ of the
$1\,006\,700\,565$ connected simple graphs lie in a nontrivial class, one in
$22\,366$, whereas two of the $2\,171$ $2$-trees do, one in $1\,085$.  The observed proportion is therefore approximately twenty times larger
among $2$-trees.

The number of $2$-trees on $n$ vertices
(sequence \texttt{A054581} of the OEIS) grows slowly enough that several
more values of $n$ are within reach, whereas the unrestricted census is already
out of range at twelve vertices.  Moreover $2$-trees have tree-width two, so by
Noble~\cite{NobleTW} the $U$-polynomial can be evaluated on them in polynomial time.

We do not know whether the $\LL$-, $\widehat{\LL}$- and $U$-classes of $2$-trees
continue to coincide beyond eleven vertices, nor how the proportion of
$2$-trees lying in a nontrivial class behaves as $n$ grows.

\subsection{Geometry of the forest pieces}
\label{subsec:outlook-geometry}

The geometric parking complexes of
Theorem~\ref{thm:polyhedral-parking-cells} partition the lattice points
of $P_G$, but need not be convex or cover the real polytope. Can the
residual regions be assigned recursively to obtain a subdivision or
shelling? Can the dependence on the edge order and orientations be
removed? Is there a basis, filtration, or degeneration of the external
bizonotopal algebra whose graded pieces realize the forest factors in
\eqref{eq:parking-cell-enumerator}?

\subsection{Four-term relations}
\label{subsec:outlook-structural}

Which specializations, quotients, or refinements of $\LL$ satisfy the
graph four-term relation for Vassiliev weight systems, rather than only
the triangular relation of Proposition~\ref{prop:common-modular-relation}?
The distinction is explained in Remark~\ref{rem:vassiliev-4term}.

\appendix

\section{Computational method}
\label{sec:the-computation}

\subsection{The general census}
\label{subsec:app-census}

This subsection describes how the search reported in Section~\ref{sec:extended-U} was carried out.
It also underlies the completeness and minimality assertions in Theorem~\ref{thm:chordal-counterexamples}.
Readers willing to accept these computer-assisted results may omit this appendix.

Our computation is organized around the connected-block polynomials $\cC_G(B;q)$
of~\eqref{eq:connected-block-polynomial}, which we compute for all $2^{n}$ vertex subsets
simultaneously by the standard subset recursion
\[
  \cC_G(B;q)  = (1+q)^{e_G(B)} -\!\!\sum_{\substack{S\subsetneq B\\ \min B\in S}}\!\!
  \cC_G(S;q)\,(1+q)^{e_G(B\setminus S)},
\]
at a cost of $O(3^{n})$ ring operations.  The three invariants are then read
off as set-partition sums over the blocks: $\XB_G$ by
\eqref{eq:XB-connected-block-expansion}, $\LL_G(1,\bx)$ by
\eqref{eq:L-block-transform}, and $\widehat{\LL}_G$ by the same formula with
the refined block transform, each a further $O(3^{n})$ subset dynamic program.

Evaluating $\LL_G$ and $\widehat{\LL}_G$ appears at first to require the
expansion of $\cC_G(B;q)$ in the basis $q^{s-1}(1+q)^r$, rather than merely
its values.  This expansion can be avoided.  If $s=|B|$, write
\[
  Q_G(B;q)  :=   q^{1-s}\cC_G(B;q)   =   \sum_r a_r(1+q)^r.
\]
The first expression is a polynomial, since $\cC_G(B;q)$ is divisible by $q^{s-1}$.
Hence, for every $\theta$,
\begin{equation}\label{eq:evaluation-trick}
  \sum_r a_r\theta^r   =   Q_G(B;\theta-1).
\end{equation}
Thus any substitution of the form
\[
  u_{s,r}\longmapsto  \sum_{k<K}\beta_{k,s}\theta_k^r
\]
is obtained from $K$ evaluations of the polynomials $Q_G(B;\,\cdot\,)$.
Choosing $\beta_{k,s}=\alpha_k\theta_k^{s-1}$ realizes the diagonal specialization
$u_{s,r}\mapsto x_{s-1+r}$ defining $\LL_G$, while independent random
$\beta_{k,s}$ give random evaluations of $\widehat{\LL}_G$.  The same
$K$ evaluations can therefore be reused for all three invariants,
so computing all three invariants costs little more than computing one.

All arithmetic is carried out in $\FF_p$ with $p=2^{61}-1$, and the
substitution points are chosen at random.  Each graph receives \emph{three}
separate tuples of finite-field evaluations, one for $\LL_G$, one for
$\widehat{\LL}_G$ and one for $\XB_G$. We call these tuples \emph{fingerprints} of the graph.
Because equal polynomials agree under every evaluation, a genuine polynomial collision
survives the fingerprint stage. Unequal polynomials may accidentally have the same
fingerprint, and such false positives are removed at the end by recomputing every
reported class over $\ZZ$.

It is important that the three fingerprints are kept apart and that the graphs are bucketed
three times, once on each of them, producing three partitions of the census which are then
compared. Bucketing instead on the combined triple would certify the classes on which all
three invariants agree, but a pair with $\XB_{G}=\XB_{H}$ and $\LL_{G}\ne\LL_{H}$
that would disprove Conjecture~\ref{conj:L-XB-equivalence} would be split between
two buckets and not found.
Thus the claim in Proposition~\ref{prop:computational-evidence} that the $\LL$-,
the $\widehat{\LL}$- and the $\XB$-classes \emph{coincide} is based on a comparison
of three independently computed partitions.

Two reductions make the range attainable.  First, all three invariants determine $m(G)$
(Corollary~\ref{th:cor}(4) for $\LL_G$ and the top power of $q$ for $\XB_G$), so collisions occur only
within a fixed edge count and the computation splits into independent jobs, one for each pair $(n,m)$.
Together with Corollary~\ref{cor:connected-reduction} it suffices to run these on connected graphs,
which we generate with \texttt{geng} from McKay and Piperno's \texttt{nauty}~\cite{nauty}.
Second, and more importantly, most graphs can be discarded before any $O(3^{n})$ work is done,
using invariants that $\LL_G$ determines but that are far cheaper to compute.
We apply these tests in increasing order of cost, recomputing collision classes after each stage:
\begin{enumerate}[(i)] \item $m$, the degree sequence, and the induced edge count profile
  $\bigl(\cE_{G,s}(q)\bigr)_{s}$ of~\eqref{eq:induced-edge-polynomial}, at a cost of $O(2^{n})$
  operations (Theorem~\ref{thm:induced-edge-profile});
\item the number of spanning trees, $O(n^{3})$ operations (determined by $T_G$,
  Corollary~\ref{cor:tutte-specialization});
\item Stanley's chromatic symmetric function $X_G=\XB_G(-1;\bw)$, one further evaluation
  of the block recursion (Theorem~\ref{thm:L-determines-X}).
\end{enumerate}
Every filter in this list is determined by $\LL_G$, so a graph discarded at any stage
cannot belong to a collision class.  Each is also determined by $\XB_G$,
so the same run computes the collision classes of both invariants.
This sieve substantially reduces the workload: on the complete list of connected
simple graphs with nine vertices, stage~(i) retains $9.2\%$ of the graphs, stage~(ii)
retains $0.18\%$, and stage~(iii) retains $0.058\%$,
so that fewer than one graph in a thousand ever requires the full fingerprints.
Without the sieve the eleven-vertex computation would be roughly an order
of magnitude more expensive.

The implementation is in C and is parallelized over graphs.
On an Apple MacBook Pro with an M2 processor and 16~GB of memory,
the complete search takes about $1.5$ minutes for $n=10$ and about $7.5$ hours for $n=11$.
The peak memory requirement is roughly $5$~GB, and is governed by the largest single edge count bucket
($m=27$, approximately $9.6\times 10^7$ graphs) rather than by the total.
The code and the complete lists of colliding graphs are available at~\cite{code}.
As checks on the implementation we verified that the numbers of connected simple graphs
processed agree with \texttt{A001349}.
The fingerprint computation was validated against an independent implementation directly
from the defining expansions on all simple graphs with at most seven vertices.
The whole search was repeated with independent random substitution points,
giving identical classes.

\subsection{Chordal graphs}
\label{subsec:app-chordal}

The equalities asserted in Theorem~\ref{thm:chordal-counterexamples} for the eight
specified graphs also admit an independent exact check that does not rely on the census.
Enumerate all $678\,570$ set partitions of their eleven vertices and compute
\[
  \eta^{\mathrm{ext}}_G(\mathbf q) =\sum_{\pi\vdash V(G)}\prod_{B\in\pi}q_{|B|,e_G(B)}.
\]
For every $i$, the resulting integer coefficient arrays for $G_i$ and $H_i$
are identical.  The equivalence of the extended polychromate with the extended
$U$-polynomial~\cite[Corollary~4.6]{MN}, and Theorem~\ref{thm:refined-extended-U},
therefore give equality of the refined loopy polynomials as well.
In particular, all the ordinary invariants in the statement agree.
This verification uses exact integer coefficients, not finite-field fingerprints.

\subsection{Split graphs}
\label{subsec:app-split}

We briefly describe the enumeration, since the assertion of minimality in
Theorem~\ref{thm:split-counterexamples} depends on its completeness.
Let $b(k,\ell)$ be the number of bipartite graphs with distinguished parts of orders $k$ and $\ell$.
Burnside's lemma applied to the permutation action of $S_k\times S_\ell$ on $0$-$1$ matrices of
size $k\times \ell$ gives
\[
  b(k,\ell)  =  \sum_{\lambda\vdash k,\ \mu\vdash\ell} \frac{ 2^{\sum_{i,j}\gcd(\lambda_i,\mu_j)} }{ z_\lambda z_\mu },
  \quad \text{where} \quad z_\lambda=\prod_{r\ge 1}r^{m_r(\lambda)}m_r(\lambda)!
\]
is the order of the centralizer of a permutation of cycle type $\lambda$.
Maximum-clique split partitions of a fixed split graph form a single orbit under its automorphism group because
any two such partitions differ by switching $w\in K$ and $v\in I$ with $N(w)=N(v)=K\setminus\{w\}$, and the transposition $(w\,v)$ in such a case is an automorphism.
Therefore the number of split graphs on $n$ vertices with clique number $c$ is
\[
  b(c,n-c)-b(c,n-c-1),
\]
with the second term omitted when $c=n$.  Summing over $c$ gives the numbers in~\eqref{eq:split-graphs}.

By Proposition~\ref{prop:split-clique-stable} $(\omega,\alpha)$ is constant on collision classes of split
graphs for $\LL$ as well as for $U$, so in both cases the search splits over the clique number with
no comparisons between different values.
Complementation exchanges $(\omega,\alpha)$ with $(\alpha,\omega)$ and can give a further
$50\%$ reduction for $\XB$ and $U$ since by Proposition~\ref{prop:XB-complement} it
preserves their collision classes. For $\LL$ the complementation property is
unavailable (it is Problem~\ref{prob:complement}),
so the enumeration checked every isomorphism class of split graphs on $n$ vertices.
We used complementation only as a check. All $32$ classes of Theorem~\ref{thm:split-counterexamples}
are found independently and the two halves, sixteen with $\omega=6$ and sixteen with
$\omega=8$, are exchanged by complementation as predicted by Proposition~\ref{prop:XB-complement}.

\subsubsection{The search}

A split graph with a chosen split partition $V=K\sqcup I$, $|K|=k$, $|I|=s$,
is encoded by the $k\times s$ incidence matrix between $K$ and $I$, the column
$c_j\subseteq K$ recording the neighborhood of the $j$th vertex of $I$.
The clique $K$ is maximum precisely when no column is all ones, and two such
matrices describe isomorphic graphs if and only if they differ by permutations of
the rows and of the columns.  We enumerate the columns as a \emph{multiset},
which quotients by the column permutations, and remove almost all of the
residual $k!$ by two filters that cost nothing: the row sums must be non-increasing,
and ties among them must be broken consistently by a permutation-equivariant refinement of the row key.
Every isomorphism class keeps a representative, so nothing is lost.  On thirteen vertices this leaves
$7\,196\,355$ matrices for the $5\,067\,146$ isomorphism classes, a redundancy
of $1.42$, and no canonicity test is ever performed.

For a split graph every sum over vertex subsets and every sum over colorings
factor over $I$ once the part belonging to $K$ is fixed.
The induced edge count polynomials~\eqref{eq:induced-edge-polynomial},
for instance, are assembled from
\begin{equation}\label{eq:split-profile}
  \sum_{s}\cE_{G,s}(q)\,t^{s}
  = \sum_{A\subseteq K}t^{|A|} q^{\binom{|A|}{2}} \prod_{j=1}^{|I|} \bigl(1+t\,q^{\,|c_j\cap A|}\bigr),
\end{equation}
at a cost of $O(2^{k}|I|)$ rather than $O(2^{n})$, for a split graph presented
with a split partition $V=K\sqcup I$ with $|K|=k$.
The saving is by a factor of order $2^{\,n-k}/|I|$ and so is governed by the clique number.
It is largest for balanced split graphs, where $k$ is about $n/2$ and the count
is a square root of the generic one up to a linear factor, and it disappears as $k$ approaches $n$.
The enumeration ranges over all $k$, so no bound $k\le\lceil n/2\rceil$ is available.
What makes the device effective in practice is that the clique number is concentrated near $n/2$:
at fourteen vertices the values $k=6,7,8$ already account for $63$ of the $68$ million classes,
and summing $2^{k}|I|$ over the whole census gives $6.9\times10^{10}$ against $1.1\times10^{12}$
for the unfactored count.

The same device evaluates two further specializations of $\XB_G$, both determined by $U_G$.
The first is the two-color specialization
\[
  \sum_{S\subseteq V(G)} t^{|S|}\,Y^{e_G(S)+e_G(V\setminus S)},
\]
in which every vertex is colored by $S$ or by its complement, at cost $O(2^{k}|I|)$.
The second prescribes two disjoint color classes $S_1,S_2$ and gives every remaining vertex
a color of its own, so that the number of monochromatic edges is $e_G(S_1)+e_G(S_2)$; its
cost is $O(3^{k}|I|)$.
Let $f_1(G)$ be the tuple of finite-field evaluations of~\eqref{eq:split-profile}, let $f_2(G)$
consist of $f_1(G)$ together with evaluations of the two-color specialization, and let
$f_3(G)$ consist of $f_2(G)$ together with evaluations of the specialization with two
prescribed color classes. The invariants are thus \emph{cumulative} by construction:
$f_1$ is a component of $f_2$, which is a component of $f_3$.
Taken alone, the two-color specialization records $e_G(S)+e_G(V\setminus S)$ but not
$e_G(S)$ itself (compare Remark~\ref{rem:ising-strict}), which is why the invariants are
taken cumulatively.

The search itself needs no isomorphism testing of split graphs against one another.
Key the enumeration on the pair $(f_2,\iota)$, where $\iota$ is an isomorphism invariant---we use one
round of individualization followed by color refinement.
Equivalent matrices give the same key, so the number of distinct keys does not exceed $\mathrm{sg}(n)$,
the number of split graphs on $n$ vertices. These numbers are equal exactly when
the combined key separates all such graphs and the enumeration meets every class.
Therefore the equality certifies the generator and the invariant at once, and it held at every $n\le 14$.
A pair of split graphs with equal $U$-polynomials has equal $f_2$, so it contributes
two distinct keys sharing one value of $f_2$. Therefore we can isolate collision
candidates by grouping the keys by the value of $f_2$.  The groups are few, and $\LL$,
$\widehat{\LL}$ and $\XB$ are computed for their members directly, by the
method of Section~\ref{subsec:app-census}. The results are presented in the table below.
    \[
    \begin{array}{r|rrrr}
      n & \mathrm{sg}(n) & \text{groups} & \text{classes in them} & \text{collisions}\\\hline
      11 &      64\,956 &        0 &        0 &  0\\
      12 &     501\,696 &       24 &       48 &  0\\
      13 &  5\,067\,146 &      382 &      764 &  0\\
      14 & 67\,997\,750 &  5\,956 &  11\,934 & 32
    \end{array}
    \]
At fourteen vertices the $11\,934$ candidates take $11\,902$ distinct values of each
of $\LL$, $\widehat{\LL}$ and $\XB$, so there are $32$ colliding groups.
All of them are pairs and they are the same $32$ for the three invariants.
This proves the assertions of Theorem~\ref{thm:split-counterexamples} about $U_G$ and $\widehat{\LL}_G$.

For $\LL_G$ the grouping has to be redone for the following reason.
Grouping by $f_2$ is legitimate for $U_G$ and $\XB_G$ because $f_2$ is a specialization of $\XB_G$
and so two split graphs with equal $U$-polynomials belong to the same $f_2$-group.
For $\LL_G$ this fails since it is not known if $\LL_G$ determines $f_2$.
By Theorem~\ref{thm:induced-edge-profile} $\LL_G$ determines $f_1$ and, since $f_1$ is a component
of $f_2$, the partition into $f_2$-groups \emph{refines} the partition into $f_1$-groups.
Thus a pair of classes may share an $f_1$-value while lying in two different $f_2$-groups
and it is the kind of pair that matters, since if their $\LL$-values agreed,
this would be a counterexample to Conjecture~\ref{conj:L-XB-equivalence}.
So for the $\LL$-sweep we grouped the classes by $f_1$ alone. The groups are coarser
and considerably larger, and every $f_2$-group is contained in one of them.
The invariants $\LL$, $\widehat{\LL}$ and $\XB$ were then computed directly on every
class lying in a non-singleton $f_1$-group:
\[
\begin{array}{r|rrr}
 n & \mathrm{sg}(n) & \text{classes sharing an } f_1\text{-value} & \text{distinct } \LL \text{ among them}\\\hline
 \le 7 &            &        0 &  \text{---}\\
  8 &           557 &        2 &        2\\
  9 &         2\,223 &        6 &        6\\
 10 &        10\,766 &       44 &       44\\
 11 &        64\,956 &      176 &      176\\
 12 &       501\,696 &  1\,322 &  1\,322\\
 13 &     5\,067\,146 & 10\,876 & 10\,876
\end{array}
\]
The counts of distinct $\widehat{\LL}$- and $\XB$-values among them are the same in every row.
Hence $\LL_G$, $\widehat{\LL}_G$ and $U_G$ each separate all split graphs on at most thirteen vertices,
which is the first assertion of Theorem~\ref{thm:split-counterexamples}.

A cross-check came out of the same pass: at twelve vertices the $f_1$-grouping leaves $1\,322$
classes ambiguous, and the stage-(i) sieve of Appendix~\ref{subsec:app-census}, edge count,
degree sequence and induced edge count profile, run independently
on the same $501\,696$ split graphs gives the same $1\,322$ survivors.

At fourteen vertices the same point applies to the second assertion of
Theorem~\ref{thm:split-counterexamples}.  That the $32$ pairs found above have
equal $\LL$ is a direct computation on their $64$ members.
That there is no \emph{further} $\LL$-collision does not follow from the $f_2$-sweep,
and we used the $f_1$-grouping at fourteen vertices as well.
Grouping the $67\,997\,750$ classes by $f_1$ leaves $120\,114$ of them in $59\,684$
non-singleton groups, about ten times as many as for $f_2$.
Computing $\LL$ on all of them gives $120\,082$ distinct values.
The $32$ coincidences are all pairs, they are the same $32$ for $\widehat{\LL}$ and for $\XB$,
and they coincide with the pairs of Table~\ref{tab:split-counterexamples} together with their complements.
Hence the $\LL$-, $\widehat{\LL}$- and $U$-classes agree at fourteen vertices as well,
which completes Theorem~\ref{thm:split-counterexamples}.
This pass also re-certifies the enumeration independently of the earlier one:
the number of distinct keys is $67\,997\,750$ which equals $\mathrm{sg}(14)$.

\begin{table}[htbp]
\centering
\renewcommand{\arraystretch}{1.15}
\begin{tabular}{cllc}
\toprule
$i$ & $G_i$ (graph6) & $H_i$ (graph6) & $|E(G_i)|=|E(H_i)|$ \\
\midrule
$1$ & \verb+M~~xB@OI@oO_T?b??+ & \verb+M~~xB?oa@oG_T?b??+ & $33$ \\
$2$ & \verb+M~~y`Ow@CGI_B?b??+ & \verb+M~~xcOw@AGI_B?b??+ & $33$ \\
$3$ & \verb+M~~{`Pw@AGU_B?r??+ & \verb+M~~ycPw@@GU_B?r??+ & $36$ \\
$4$ & \verb+M~~|bPW@BGQ_B?b??+ & \verb+M~~|bQW@BGI_B?b??+ & $36$ \\
$5$ & \verb+M~~~Aow@BGQ_B?b??+ & \verb+M~~|eOw@BGI_B?b??+ & $36$ \\
$6$ & \verb+M~~|bPWXCg@_b?F??+ & \verb+M~~|bPWhAg@_b?F??+ & $38$ \\
$7$ & \verb+M~~xB@W]@gU_b?r??+ & \verb+M~~xB?w]CgU_R?r??+ & $39$ \\
$8$ & \verb+M~~y`Ow`AgP_N?^??+ & \verb+M~~xcOwPAgP_N?^??+ & $39$ \\
$9$ & \verb+M~~yroCLDg@_b?r??+ & \verb+M~~xtoCTBg@_b?r??+ & $39$ \\
$10$ & \verb+M~~{aRGHBGQ_N?^??+ & \verb+M~~zCRGX?gQ_N?^??+ & $39$ \\
$11$ & \verb+M~~}QowhBGQ_F?N??+ & \verb+M~~}ROwhCgI_F?N??+ & $40$ \\
$12$ & \verb+M~~{`PwPDgX_N?^??+ & \verb+M~~ycPwHDgX_N?^??+ & $42$ \\
$13$ & \verb+M~~}QowhBGQ_N?^??+ & \verb+M~~}SowhBGI_N?^??+ & $42$ \\
$14$ & \verb+M~~~AowXCgP_N?^??+ & \verb+M~~|eOwXAgP_N?^??+ & $42$ \\
$15$ & \verb+M~~yroslCWX_N?^??+ & \verb+M~~xrqSlAWX_N?^??+ & $45$ \\
$16$ & \verb+M~~}UqcTBgT_N?^??+ & \verb+M~~}epc\CWT_N?^??+ & $45$ \\
\bottomrule
\end{tabular}
\caption{The sixteen collision pairs of Theorem~\ref{thm:split-counterexamples}
with a maximum clique of order six and a stable set of order eight. Their loopy,
refined loopy and  $U$-polynomials all agree.  The complementary pairs,
with clique $8$ and stable set $6$, give the other sixteen. The complete lists are available
at~\cite{code}.}
\label{tab:split-counterexamples}
\end{table}

For $n\le 13$ there is a shortcut. If the number of distinct values of a $U$-determined invariant $f$ equals $\mathrm{sg}(n)$, then $f$, and hence $U$, separates every pair.
This is how the range was first cleared, and it records how much of $\XB_G$ is needed.
Denote by $d_r(n)$ the number of distinct values of $f_r$. We have
\[
  \begin{array}{r|rrrr}
    n & \mathrm{sg}(n) & d_1 & d_2 & d_3\\\hline
    6 &            56 &            56 &            56 &            56\\
    7 &           164 &           164 &           164 &           164\\
    8 &           557 &           556 &           557 &           557\\
    9 &        2\,223 &        2\,220 &        2\,223 &        2\,223\\
    10 &       10\,766 &       10\,744 &       10\,766 &       10\,766\\
    11 &       64\,956 &       64\,868 &       64\,956 &       64\,956\\
    12 &      501\,696 &      501\,031 &      501\,672 &      501\,696\\
    13 &   5\,067\,146 &   5\,061\,669 &   5\,066\,764 &   5\,067\,146\\
    14 &  67\,997\,750 &  67\,937\,320 &  67\,991\,772 &  67\,997\,712
  \end{array}
\]
So $f_1$ alone is complete for $n\le7$ and first fails at eight vertices.
Adding the two-color specialization suffices through eleven vertices and leaves $24$ pairs of
classes, that is $48$ classes, unseparated at twelve. Adding also the two prescribed color
classes suffices through thirteen vertices.
At fourteen vertices collisions exist and no entry can equal $\mathrm{sg}(14)$.
There $f_3$ falls $38$ short, the $32$ genuine collisions together with six further pairs of classes
that it fails to tell apart, while $f_1$ falls $60\,430$ short and $f_2$ falls $5\,978$ short.
The corresponding groups of classes are the ones the two sweeps of this subsection have to examine.

Two remarks on how to read this table.  The entries $d_r(n)$ count distinct \emph{evaluations},
not distinct polynomials, so they are lower bounds for the number of distinct values
of the underlying invariant.  That is the safe direction for the certification:
equality $d_r(n)=\mathrm{sg}(n)$ forces $f_r$, and hence $U$, to separate every pair.
A shortfall, on the other hand, shows only that the evaluations coincide, and the statements
above that $f_1$ or $f_2$ fails to separate a given number of pairs are read in that sense.
In fact the shortfalls are stable: recomputing the whole table with independent
random substitution points leaves every entry unchanged, so the coincidences
are almost certainly coincidences of the invariants themselves.
Second, the program stores only a $64$-bit hash of the tuple rather than the tuple itself,
so the nesting of the $f_r$ holds for the keys only up to hash collisions.
However, a collision can merge two classes but never separate them, which again is the safe direction.

At thirteen vertices the enumeration checks $2.39\times 10^{9}$ matrices,
at fourteen there are $2.96\times 10^{10}$, and the two passes
take about half a minute and about ten minutes on four cores.
Every collision reported in Theorem~\ref{thm:split-counterexamples} was confirmed  over $\ZZ$.
The code and the lists of colliding graphs are available at~\cite{code}.

\subsubsection{Stable vertices of degree two}

Two features allowed us to reach the range of Theorem~\ref{thm:sp-degree-two}.
First, the number $\mathrm{sg}_2(n)$ of graphs $\sp(G)$ of order $n$ can be found directly:
apply Burnside's lemma to the action of $S_{k}$ on multisets of $n-k$ elements of the
$\binom{k}{2}$ possible edge slots, and sum over $3\le k\le n$. So no external count is needed.%
\footnote{The case  $k=n$ must not be dropped since it corresponds to the edgeless graph
  with $\sp(\overline{K_n})=K_n$.}

Second, the cost of evaluating the invariants stops growing with the number of
edges of $G$, although the invariants themselves depend on it.
Since every column has exactly two ones, the products over $I$ in formula~\eqref{eq:split-profile}
collapse. Denote  by $\Gamma$ the graph $\sp(G)$ and by $e_{i}(A)$ the number of edges of $G$
with $i$ endpoints in $A\subseteq K$. We have
\[
  \sum_{a}\cE_{\Gamma,a}(q)\,t^{a} = \sum_{A\subseteq K}t^{|A|}q^{\binom{|A|}{2}}
  (1+t)^{e_{0}(A)}(1+tq)^{e_{1}(A)}(1+tq^{2})^{e_{2}(A)},
\]
a cost of $O\bigl(2^{k}(k+\log|E(G)|)\bigr)$, while the specialization of $\XB_{\Gamma}$
with $r$ prescribed color classes costs $O\bigl((r+1)^{k}(k^{2}+\log|E(G)|)\bigr)$.
The multiplicities of $G$ enter only through the three exponents $e_{0}(A),e_{1}(A),e_{2}(A)$,
which are obtained for all $A\subseteq K$ by a subset recursion. The logarithmic term is
the cost of raising to those exponents, and nothing else depends on $|E(G)|$.
We group the classes by $f_1$, the induced edge count profile, which both
$U_\Gamma$ and $\LL_\Gamma$ determine. For $n\le15$ no two classes share a value
of $f_1$. For $16\le n\le21$ the table gives the number $\mathrm{sg}_2(n)$ of
classes, the number of groups (values of $f_1$ shared by at least two classes),
and the number of classes lying in these groups:
\[
\begin{array}{r|rrr}
 n & \mathrm{sg}_2(n) & \text{groups} & \text{classes in them}\\\hline
 16 &      12\,009 &      6 &       12\\
 17 &      29\,061 &     25 &       50\\
 18 &      72\,199 &     94 &      189\\
 19 &     183\,885 &    304 &      610\\
 20 &     479\,543 &    942 &   1\,894\\
 21 &  1\,279\,227 & 2\,806 &   5\,659
\end{array}
\]
Within every group the classes are separated by the specialization with two prescribed color
classes. Since $f_1$ is a component of $f_3$, the invariant $f_3$, and hence $U_\Gamma$ and
$\widehat{\LL}_\Gamma$, separates all $\mathrm{sg}_2(n)$ classes for $n\le21$.
For $\LL_\Gamma$ it suffices to compare classes within a group, because $\LL_\Gamma$
determines $f_1$. We computed $\LL_\Gamma$ directly for the $12$, $50$ and $189$
classes in groups at $n=16,17,18$, and in every group the values are pairwise distinct.

\medskip

\section*{Acknowledgements}
We thank Lila Crew for her comments on the manuscript and for drawing our attention to Loebl's conjecture,
which led to the counterexamples in Theorems~\ref{thm:chordal-counterexamples} and~\ref{thm:split-counterexamples}.
We also thank Klas Markstr\"om, who asked about a possible relation between the construction in his paper with Wierman~\cite{MW} and the loopy polynomial, which we address in Remark~\ref{rem:rotors}.

The third and fourth authors gratefully acknowledge the hospitality
of the Beijing Institute for Mathematical Sciences and Applications during
the initial and the final stages of this project.

We used the AI models ChatGPT 5.6 and Claude Opus 5 in the development and drafting of the paper
and in writing the accompanying computer code~\cite{code}.
All statements and proofs are the sole responsibility of the authors.

\end{document}